\documentclass[reqno,11pt]{amsart}
\usepackage{arxiv_style}
\usepackage{cancel}
 \usepackage{mathdots} 
\usetikzlibrary[calc,intersections,through,backgrounds,arrows,decorations.pathmorphing]

\def\Sfrak{\mathfrak{S}}

\newtheorem*{thma}{Theorem A}
\newtheorem*{thmb}{Theorem B}
\newtheorem*{thmc}{Theorem C}
\newtheorem*{thmd}{Theorem D}
\newtheorem*{thme}{Theorem E}

\def\ds{\displaystyle}

\DeclareMathOperator{\MTub}{MTub}  %maximal tubings
\DeclareMathOperator{\topT}{top} %top
\DeclareMathOperator{\tors}{\mathsf{tors}}
\renewcommand{\top}{\operatorname{top}}
\renewcommand{\pidown}{\pi_\downarrow}
\renewcommand{\piup}{\pi^\uparrow}
\newcommand{\meet}{\wedge}
\renewcommand{\join}{\vee}
\renewcommand{\Join}{\bigvee}

\newcommand{\greaterdot}{\;\reflectbox{$\lessdot$}}

\graphicspath{ {./images/}}

\DeclareMathOperator{\inv}{inv}
\DeclareMathOperator{\inc}{inc}
\DeclareMathOperator{\coinv}{coinv}
\DeclareMathOperator{\asc}{asc}
\DeclareMathOperator{\dec}{des}
\DeclareMathOperator{\Cut}{Cut}
\DeclareMathOperator{\cut}{cut}
\DeclareMathOperator{\Sew}{Sew}
\DeclareMathOperator{\sew}{sew}

\newcommand{\abs}[1]{\left| #1 \right|}

\newcommand{\ct}{\mathcal{T}}
\newcommand{\cj}{\mathcal{J}}
\newcommand{\ck}{\mathcal{K}}
\newcommand{\cl}{\mathcal{L}}
\newcommand{\cm}{\mathcal{M}}
\newcommand{\cx}{\mathcal{X}}
\newcommand{\cy}{\mathcal{Y}}
\newcommand{\cz}{\mathcal{Z}}

\newcommand{\downT}{\ct_\downarrow}
\newcommand{\downJ}{\cj_\downarrow}
\newcommand{\downK}{\ck_\downarrow}
\newcommand{\downX}{\cx_\downarrow}

\newcommand{\twoheadrightarrowcn}{\twoheadrightarrow}
\newcommand{\hookrightarrowcn}{\hookrightarrow}
\newcommand{\rightsquigarrowcn}{\rightsquigarrow}

\title{The poset of maximal tubings of the cycle graph is a lattice}

\author[B. Adenbaum]{Ben Adenbaum} \address{Department of Mathematics and Statistics \\ Villanova University\\ 800 E. Lancaster Ave \\ Villanova, PA 19085 \\ USA} 
\email{benadenbaummath@gmail.com} 
\urladdr{\url{https://badenbaum.github.io/}}

\author[E. Barnard]{Emily Barnard} \address{Department of Mathematical Sciences\\ DePaul University \\ 2320 N Kenmore Ave \\ Chicago, IL 60614-3210 \\ USA }
\email{e.barnard@depaul.edu}
\urladdr{\url{https://emilybarnard.github.io/}}

\author[M. Hlavacek]{Max Hlavacek} \address{Mathematics and Statistics Department\\ Pomona College \\ 333 N. College Way \\ Claremont, CA \\ USA }
\email{mhap2023@pomona.edu}
\urladdr{\url{https://maxhlav.github.io/}}

\author[B. Kagy]{Bryson Kagy} \address{Department of Mathematics\\ Texas State University\\ Pickard St \\ San Marcos, TX 78666 \\ USA } 
\email{BrysonKagy@txstate.edu}
\urladdr{\url{https://brysonkagy.github.io/}}

\author[N. Lesnevich]{Nathan Lesnevich} \address{Department of Mathematics\\ Oklahoma State University \\ 401 Mathematical Sciences \\ Stillwater, OK 74075 \\ USA }
\email{nlesnev@okstate.edu}
\urladdr{\url{https://nlesnevich.github.io/}}

\author[G. D. Nasr]{George David Nasr} \address{Department of Mathematics\\ Augustana University \\ 2001 S Summit Ave
 \\ Sioux Falls, SD 57197 \\ USA }
\email{george.nasr@augie.edu}
\urladdr{\url{https://sites.google.com/view/george-d-nasr-math}}

\author[K. Waddle]{Katie Waddle} \address{Department of Mathematics\\ University of Michigan \\ 2074 East Hall, 530 Church St \\ Ann Arbor, MI 48109 \\ USA }
\email{waddle@umich.edu}
\urladdr{\url{https://sites.google.com/view/katie-waddle/home}}

\begin{document}

\begin{abstract}
    The poset of maximal tubings of a graph generalizes several well-known and remarkable partial orders. Notable examples include the weak Bruhat order and the Tamari lattice, posets of maximal tubings for the complete graph and the path graph, respectively. It is an open problem to characterize graphs for which the poset of maximal tubings is a lattice. In this paper, we prove that the poset of maximal tubings for the cycle graph is a lattice, and moreover that it is semidistributive and congruence uniform. As main tools, we characterize all order relations in the poset, and introduce a useful map from maximal tubings of the cycle graph to maximal tubings of the path graph.
\end{abstract}

\maketitle 
\tableofcontents

\section{Introduction}\label{sec:Intro}
Fix a simple graph $G$ with vertex set $V=[n]:=\{1, 2, \ldots, n\}$.
A subset $X\subseteq V$ is a \emph{tube} of $G$ if the subgraph induced by $X$ is connected.
A pair of tubes~$X$ and~$Y$ are \emph{compatible} if~$X\subseteq Y$,~$X\supseteq Y$, or~$X\cup Y$ is not a tube.
A \emph{maximal tubing}~$\ct$ of a graph~$G$ is a maximal collection of compatible tubes of~$G$.
For each proper tube $X\subset V$ in $\ct$, there exists a unique tube $Y$ such that $(\ct \setminus \{X\})\cup \{Y\}$ is a maximal tubing.
We call the transition $\ct \to (\ct \setminus \{X\})\cup \{Y\}$ a \emph{flip}.
See Definitions~\ref{def:maximal tubing} and~\ref{def:MTub}.

With adjacency given by the flip relation, the set of maximal tubings  form a simple graph that is isomorphic to the 1-skeleton of a simple polytope called the \emph{graph associahedron} $P_G$.
That is, the set of vertices of the graph associahedron $P_G$ is in bijection with the set of maximal tubings of $G$.
By fixing a linear function, $\lambda(x_1, x_2, \ldots, x_n) = nx_1 +(n-1)x_2+ \cdots 1x_n$, and orienting the edges of $P_G$ so that $\lambda$ is increasing, we obtain a partial order which we call the \emph{poset of maximal tubings}, $\MTub(G)$. 
See \cite[Lemma~2.8]{BM2021}.

Each graph associahedron is an example a generalized permutahedron, as defined by Postnikov in \cite{postnikov:2009permutohedra}; they also appeared independently in \cite{carr.devadoss:2006coxeter} and \cite{davis.janus.scott:2003fundamental}. 
Classical examples of graph associahedra and their corresponding poset include the permutahedron and the weak order on the type A Coxeter group, the associahedron, and the Tamari lattice.
Both the weak order and the Tamari lattices are remarkable lattice posets:
They are semidistributive and congruence-uniform, meaning that the combinatorics of their lattice quotients can be understood in terms of their join irreducible elements \cite{Caspard, GEYER}.
In addition, they are both the lattice of torsion classes for some finite-dimensional associative algebra $\Lambda$ \cite{Thomas_Torsion, Mizuno}.

For a general graph $G$, it is not true that $\MTub(G)$ is even a lattice (let alone semidistributive and congruence uniform).
See \cite[Figure~9]{BM2021}.
A fundamental question is to characterize for which $G$ the poset of maximal tubings $\MTub(G)$ is a lattice-poset \cite{Forcey_Lattice_Question}.
This problem remains open, in part, because while cover relations in $\MTub(G)$ have an explicit combinatorial description, in general it is difficult to tell when two maximal tubings $\ct$ and $\ct'$ are even comparable.

In \cite{BM2021}, the authors provide a partial answer to this question.
A graph $G$ is \emph{filled} if whenever $\{i,j\}$ is an edge of $G$, with $i<j$, then every pair of elements in the interval $[i,j]$ is also an edge in $G$.
\cite[Theorem~1.1]{BM2021} says that if $G$ is filled, then $\MTub(G)$ is a lattice, and indeed, it is a lattice quotient of the weak order on the type A Coxeter group.
Notably, the cycle graph $C_n$ with vertices labeled in cyclic order is not filled.

This paper studies the remarkable lattice-theoretic properties of $\MTub(C_n)$.
Our first main result says that $\MTub(C_n)$ is indeed a lattice-poset.

\begin{thma}[Theorem~\ref{thm:lattice}]
\label{theorem_a}
Let $C_n$ denote the cycle graph on $n$ vertices, labeled $1, 2, \ldots, n$ in cyclic order.
The poset of maximal tubings $\MTub(C_n)$ is a lattice-poset.
\end{thma}

The poset of maximal tubings $\MTub(C_n)$ stands out for many reasons.
The number of elements in $\MTub(C_n)$ is $\binom{2n-2}{n-1}$, which is the type B Catalan number.
Like a classical Tamari lattice or a $c$-Cambrian lattice (of any finite type Coxeter group) the Hasse diagram of $\MTub(C_n)$ is a regular graph.
However, $\MTub(C_n)$ is not isomorphic to a type B Tamari lattice or type B $c$-Cambrian lattice.
In fact, $\MTub(C_n)$ is not \emph{trim} (a lattice-theoretic property shared by any $c$-Cambrian lattice, regardless of type).
See \cite{Thomas_Trim, Muhle}.

In order to prove Theorem~\ref{thm:lattice}, we first characterized global order relations in $\MTub(C_n)$.
Importantly, each maximal tubing $\ct$ for any $\MTub(G)$, where $G$ is connected, can be uniquely encoded by a tree poset, called the $G$-tree of $\ct$, and denoted $G_\ct$.
See Definition~\ref{def:Gtree} for the precise definition.
When $G$ is the path graph $P_n$ with vertices labeled $1, 2, \ldots, n$ in the usual linear order,  each $G$-tree is a binary search tree \cite{BjornerWachs_shellableII}.
When $G$ is the complete graph on $[n]$, then each $G$-tree corresponds to a permutation.

For a filled graph $G$, the set $\{\ct'\in \MTub(G) \mid\ \ct'<\ct\}$ is entirely determined by the set of pairs $1\le i<j\le n$ such that $j$ is less than $i$ in the $G$-tree $G_\ct$.
We call such a pair $(i,j)$ an \emph{inversion} of $\ct$.
See Definition~\ref{def:inversions etc}.
For $\MTub(C_n)$ inversions alone cannot determine whether a pair of maximal tubings are comparable.
See Example~\ref{ex:Global Relations}.
We prove the following:

\begin{thmb}[Theorem~\ref{thm:inversion_order}]
\label{theorem_b}
  Let~$\cj$ and~$\ck$ be tubings of the cycle graph. Then~$\cj \leq \ck$ if and only if~${\inv(\cj) \subseteq \inv(\ck) \cup \inc(\ck)}$, where $\inc(\ck)$ denotes the set of pair $i<j$ such that $i$ and $j$ are not comparable in the $G$-tree of $\ck$.
\end{thmb}

In order to prove Theorem~\ref{thm:inversion_order}, we define and analyze a surjective map ${\Cut : \MTub(C_n) \to \MTub(P_n)}$ which we call the $\Cut$ map. For a precise definition see Definition~\ref{def:cutmap}.
Intuitively, to visualize the $\Cut$ map, we imagine taking a pair of scissors and snipping the edge between $1$ and $n$. 
Tubes of a maximal tubing $\ct$ of the cycle graph which do not contain both $1$ and $n$ are unchanged by the $\Cut$ map.
A tube of $\ct$ containing both $1$ and $n$ must have the form $\{i, i+1, \ldots, n, 1, \ldots j\}$.
The $\Cut$ map splits this tube into two tubes of $P_n$, and selects one of these two tubes to include in the image. 

It turns out that each fiber $\Cut^{-1}(\cx)$ for $\cx\in\MTub(P_n)$ is isomorphic to an interval of the weak order a type $A$ Coxeter group. 
In fact, we show the following.
Recall that a surjection $\phi : L \to L'$ from lattice $L$ to $L'$ is a \emph{lattice quotient} map if $\phi$ respects the join and meet operations of $L$ (e.g. $\phi(x\join y) = \phi(x) \join \phi(y)$).  

\begin{thmc}[Corollary~\ref{cut_lattice_map}]
\label{theorem_c}
The surjection $\Cut: \MTub(C_n) \to \MTub(P_n)$ is a lattice quotient map.
\end{thmc}

Recall that an element $j$ in a finite lattice-poset is called \emph{join irreducible} if whenever $j=\Join A$ for any subset $A\subset L$, then $j\in A$.
Equivalently, $j$ covers precisely one element in $L$.
The combinatorics of the $G$-trees of $\ct \in \MTub(C_n)$ allow us to characterize and enumerate the set of all join irreducible maximal tubings in $\MTub(C_n)$.
See Definition~\ref{def:join_irr}, and Lemma~\ref{lem:all_ji_gtrees}.
Surprisingly, the number of join irreducible elements in $\MTub(C_n)$ is equal to the number of join irreducible elements in a type-B Tamari lattice (or any type-B Cambrian lattice), which is equal to the number of reflections in the type B Coxeter group \cite{Reading_Speyer}.

\begin{thmd}[Corollary~\ref{cor:jienum}]
The number of join irreducible elements in $\MTub(C_n)$ is equal to $(n-1)^2$.
\end{thmd}

Finally we combine our global description of the order in $\MTub(C_n)$ and our characterization of the join irreducible elements of $\MTub(C_n)$ to prove our final main result.

\begin{thme}[Theorem~\ref{thm:semidistributive} and Theorem~\ref{thm:congruence_uniform}]
The poset of maximal tubings $\MTub(C_n)$ is a semidistributive, congruence uniform lattice.
\end{thme}

Congruence uniform lattices form an important family of lattice-posets. 
(A friendly introduction to the definition and properties of lattice congruences and congruence uniform lattices can be found in \cite[Section~5]{Reading_Lattice_Theory}.)
The combinatorics of all of their lattice-quotients is entirely determined by their join irreducible elements.
As already noted, both the weak order and the Tamari lattice (more generally, any $c$-Cambrian lattice) are congruence uniform. 
Other notable examples are quotients of the weak order of the a finite Coxeter group.
Thus, $\MTub(C_n)$ is a rare example of a congruence uniform lattice, whose Hasse diagram is a regular graph, which is \emph{not} such a lattice-quotient.
Another family of congruence uniform lattices, whose Hasse diagrams are also regular graphs, are lattices of torsion classes $\tors(\Lambda)$ where $\Lambda$ is a $\tau$-titling finite algebra.
It would be interesting to know whether there exists $\Lambda_{C_n}$ such that $\MTub(C_n)$ is isomorphic to $\Lambda_{C_n}$.

Our paper is organized as follows.
In Section~\ref{sec:graphs and tubings} we state the key definitions for tubings of a graph $G$, the poset of maximal tubings, and $G$-trees.
In Sections~\ref{sec:G-trees for paths} and~\ref{sec:Gtrees for cycle}, we review the (well-known) characterization of the $G$-trees of the maximal tubings of the path and cycle graphs.
In Section~\ref{sec:Cut Map}, we define the $\Cut$ map and characterize its fibers.
In Section~\ref{sec:Global Relations}, we describe the global relations for $\MTub(C_n)$, as in Theorem~\ref{thm:inversion_order}.
In Section~\ref{sec:Lattice}, we prove our first main result, Theorem~\ref{thm:lattice}, that $\MTub(C_n)$ is a lattice poset.
In Section~\ref{sec:joinirr} we characterize the join irreducible elements of $\MTub(C_n)$, and finally, in Section~\ref{sec:semidis}, we prove that $\MTub(C_n)$ is a semidistributive and congruence uniform lattice.
Our final section describes several new and interesting research directions.

\subsection*{Acknowledgments}
This paper began as a project at the AMS Mathematical Research Communities (MRC) in June 2024, and we thank the organizers for inviting us to participate, and providing a safe and supportive research environment. 
In particular, we thank Rebecca Garcia for her leadership and guidance. 
We thank Katherine Orme{\~n}o Bast{\'i}as for her key insights.
This material is based upon work supported by the National Science Foundation grant DMS-1916439. %this grant supported the MRC
Katie Waddle was also partially supported by the National Science Foundation grant DMS-2348501.

\section{Graphs and maximal tubings}\label{sec:graphs and tubings}
We will denote~$[n] \coloneqq \{1,\ldots,n\}$. For partial or total orders, we use~$<$ and~$\leq$ normally and~$\lessdot$ to denote cover relations. If the order is on~$[n]$, then~$<$ and~$\leq$ denote the usual order~$1 < \cdots < n$, and any other partial order will be denoted with a subscript, e.g.~$\leq_X$ as the partial order on~$[n]$ derived from object~$X$.

A graph is a tuple~$G = (V,E)$ where~$V$ is the set of vertices and~$E \subseteq V \times V$ is the set of edges. Graphs here are understood to be undirected and simple (i.e., finite, loopless, and without multiple edges), and will always have vertex set~$[n] \coloneqq \{1,\ldots,n\}$.

Good references for the definitions in this section are \cite{postnikov.reiner.williams:2008faces,carr.devadoss:2006coxeter, BM2021}.

\begin{definition}\label{def:maximal tubing}%the objects
    A \emph{tube}~$X$ of a graph~$G$ is a set of vertices that form a connected induced subgraph of~$G$. Two tubes~$X$ and~$Y$ are \emph{compatible} if~$X\subseteq Y$,~$X\supseteq Y$, or~$X\cup Y$ is not a tube.
     A \emph{maximal tubing}~$\ct$ of a graph~$G$ is a maximal collection of compatible tubes of~$G$.
\end{definition}
Given a maximal tubing~$\ct$ and a vertex~$x \in G$, there is a single unique smallest tube that contains~$x$, denoted~$\downT(x)$. Similarly, given a tube~$X$ there exists a unique least-nested element~$x = \topT_\ct(T)$ such that~$\downT(x) = X$.

Our main object of study is the following.
\begin{definition}\label{def:MTub} %The order
    Let~$\MTub(G)$ be the partial order of all maximal tubings of the graph~$G$ defined by the cover relations $\ct \lessdot \cj$ if
    \begin{enumerate}
        \item $\ct$ and $\cj$ differ by precisely one tube $T \in \ct$ and $J \in \cj$, and
        \item $\top_\ct(T) < \top_\cj(J)$ in the standard order on integers.
    \end{enumerate}
    The process of exchanging the single tube $T$ for $J$ to move from $\ct$ to $\cj$ is called a \emph{flip}.
\end{definition}
Lemma \ref{lem:tubing_covers} below states that we get precisely one cover relation involving $\ct$ in $\MTub(G)$ for every nontrivial tube in $\ct$. 
\begin{lemma}\label{lem:tubing_covers}
    Let $\ct \in \MTub(G)$. For each nontrivial tube $T \in \ct$, there is a unique tube $J \neq T$ where such that $(\ct \setminus \{T\})\cup \{J\}$ is also a maximal tubing for $G$.
\end{lemma}

Rather than a graph operation which changes the set of edges or vertices, there is a way of relabeling the graph which yields an interesting relationship between the original and new graphs' tubing poset.
\begin{lemma}\label{lem:w0_dual}
    Let~$G$ be a labeled graph with vertex set~$[n]$. Let~$G'$ be the labeled graph obtained from~$G$ by relabeling~$i$ with~$n+1-i$. Then~$\MTub(G')$ is anti-isomorphic to~$\MTub(G)$. 
\end{lemma}
\begin{proof}
    Denote by~$w_0$ the permutation which sends~$i\mapsto n+1-i$. We extend~$w_0$ to act on all subsets and collection of subsets of~$V(G)$, which by abusing notation we also denote by~$w_0 S$, where if~$S = \{i_1,\dots , i_k\}\subseteq [n]$ then~$w_0 S  = \{n+1-i_1, \dots ,n+1-i_k\}$. If~$\ct \lessdot \cj$ in~$\MTub(G)$, since~$G' = w_0 G$  then in~$\MTub(G')$ all labels are just~$w_0(i)$, so in particular for any maximal tubing $\ct$ and tube $T\in \ct$, $\top_{w_0\ct}(w_0T)=w_0\top_\ct(T)$. The pairwise order of all labels then are reversed implying~$w_0(\cj)\lessdot w_0(\ct)$ in~$\MTub(G')$. Consequently we have a bijection from~$\MTub(G)$ to~$\MTub(G')$ that reverses the order relations of every cover relation. Furthermore since the global order relation on~$\MTub(G')$ is defined by being related by a sequence of covers then~$w_0$ is an order reversing bijection from~$\MTub(G)$ to~$\MTub(G')$. That the inverse function is also order reversing follows immediately since the action of~$w_0$ is involutive. 
\end{proof}

\begin{corollary}\label{cor:involution} Denote $w_0$ the order-reversing permutation that relabels $i$ with $n-i+1$.
    If $w_0$ induces an automorphism of labeled graphs on $G$, a graph isomorphism that preserves adjacency of labels, then $\MTub(G)$ is self-dual. In particular, $\MTub(P_n)$ and $\MTub(C_n)$ are self-dual.
\end{corollary}
    
\subsection{\texorpdfstring{$G$}{G}-trees}\label{subsec:Background Gtrees}
The tubings of a graph~$G$ are uniquely determined by particular posets on~$[n]$.
\begin{definition}\label{def:Gtree}
    Given a graph~$G$ on~$[n]$, a maximal tubing~$\ct \in \MTub(G)$ defines a partial order~$\leq_\ct$ on~$[n]$, where~$x \leq_\ct y$ if and only if~$x \in \downT(y)$. The Hasse diagram of the partial order~$\leq_\ct$ is a tree, and we will flexibly refer to the poset~$([n],\leq_\ct)$ and the diagram as a \emph{$G$-tree}, denoted $G_\ct$. Let $x$ be a vertex in $G$. We denote the principal lower order ideal of $x$ in the partial order $\leq_\ct$ by $\downT(x)$.
\end{definition}

\begin{remark}
    For any finite poset, every element having at most~$1$ cover implies that the Hasse diagram is a tree (or forest), since any minimal element in a cycle must have two covers. For maximal tubings~$\ct$, every element has at most one cover since if~$t < j_1$ and~$t < j_2$ then~$t \in \downT(j_1) \cap \downT(j_2) \neq \emptyset$, and so one of~$\downT(j_1)$ or~$\downT(j_2)$ is contained within the other. 
\end{remark}
\begin{remark}
$G$-trees are rooted. The root element of the~$G$-tree corresponding to a maximal tubing~$\ct$ is~$\topT_\ct([n])$.
\end{remark}
\begin{definition}\label{def:inversions etc}
Given a poset~$P$ and a total order~$<$ on its elements, the set of \emph{inversions}~$\inv(P)$, set of \emph{descents}~$\dec(P)$, set of \emph{coinversions}~$\coinv(P)$, set of \emph{ascents}~$\asc(P)$, and set of \emph{incomparable elements}~$\inc(P)$ are
\begin{center}
\begin{tabular}{rlrl}
$\inv(P)$&\hspace{-7pt}$\coloneqq\{(i<j) \mid i>_P j\}$,&
$\dec(P)$&\hspace{-7pt}$\coloneqq \{(i< j) \mid i\greaterdot_P\; j\}$, \\
$\coinv(P)$&\hspace{-7pt}$\coloneqq\{(i<j) \mid i<_P j\}$,&
$\asc(P)$ &\hspace{-7pt}$\coloneqq \{(i< j) \mid i\lessdot_P j\}$, and \\
$\inc(P)$&\hspace{-7pt}$\coloneqq \{(i,j) \mid i \not\leq_P j \text{ and }j \not\leq_P i\}$, & \hspace{24pt}respectively.&
\end{tabular}
\end{center}
When referring to above sets for a G-tree $G_\ct$, we simply write the tubing $\ct$, i.e. $\inv(\ct) \coloneqq \inv(G_\ct)$.
\end{definition}
\begin{remark}\label{rem:set_rels.}
    These sets have several apparent relations, and each naturally sit within~$[n]^2$. For example, the set~$\inv(P) \cup \coinv(P) = \inc(P)^c$ is the set of all order relations on~$P$, as is~$\dec(P) \cup \asc(P)$ the set of cover relations, similarly~$\inv(P) \cup \inc(P) = \coinv(P)^c$, and so on.
\end{remark}

\section{\texorpdfstring{$G$}{G}-trees for maximal tubings of \texorpdfstring{$P_n$}{Pn}}\label{sec:G-trees for paths} 
This section describes maximal tubings of the path graph~$P_n$ on~$n$ vertices, and characterizes their cover relations in terms of $G$-trees. We will use $\cx,\cy,\cz$ for maximal tubings of $P_n$.
\begin{definition}
    \emph{Binary search trees} are rooted binary trees on~$[n]$ with the property that if any parent vertex~$p$ has a left child~$x$, then every vertex in the left subtree~$T_2$ rooted at~$x$ is less than~$p$, and if~$p$ has a right child~$y$, then every vertex in the right subtree~$T_3$ rooted at~$y$ is greater than~$p$. Visually, given the following:
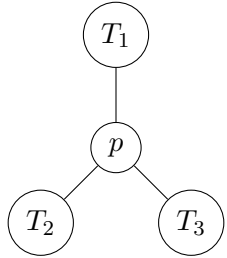
\begin{center}
    \begin{tikzpicture}
        \draw (0,0.5) node[draw, circle] (RT) {$T_1$};

        \draw (0,-1) node[draw, circle] (p) {$p$};
        \draw (-1,-2) node[draw, circle] (X) {$T_2$};
        \draw (1,-2) node[draw, circle] (Y) {$T_3$};

        \draw (RT)--(p);
        \draw (p)--(X);
        \draw (p)--(Y);
    \end{tikzpicture}
\end{center}
all elements in the sub-tree~$T_2$ are less than~$p$ and all elements in the sub-tree~$T_3$ are greater than~$p$. 
\end{definition}

\begin{proposition}[\cite{BjornerWachs_shellableII}, \S9]\label{prop: BST}
Let~$\cx\in\MTub(P_n)$ with~$G$-tree~$G_\cx$. Then~$G_\cx$ is a binary search tree. Every binary search tree is the~$G$-tree for a maximal tubing of the path. 
\end{proposition}

\begin{example}\label{ex: G-Trees for P3}
Below are all binary search trees with three elements, and the nontrival maximal tubings of~$P_3$ they correspond to. All also include the trivial tubing $\{1,2,3\}$.
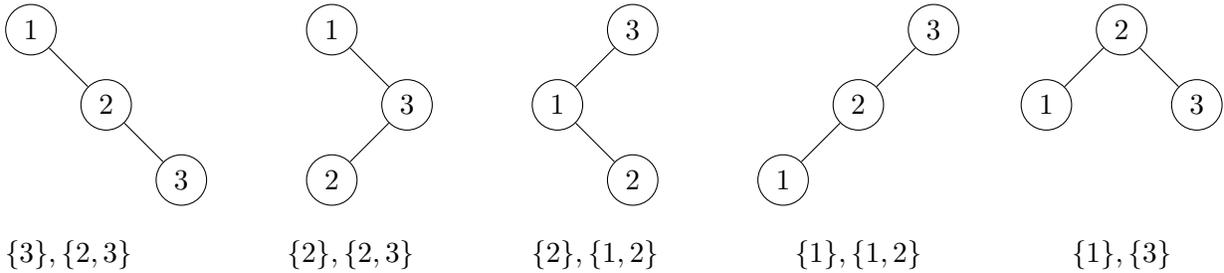
\begin{center}
    \begin{tikzpicture}
        \begin{scope}
            \draw (0,0) node[draw, circle] (t1) {$1$};
            \draw (1,-1) node[draw,circle] (t2) {$2$};
            \draw (2,-2) node[draw,circle] (t3) {$3$};
            \draw (t1)--(t2)--(t3);
            \draw (0.5,-3) node (tub) {$\{3\},\{2,3\}
            $};
        \end{scope}
        \begin{scope}[xshift = 4cm]
            \draw (0,0) node[draw, circle] (t1) {$1$};
            \draw (1,-1) node[draw,circle] (t2) {$3$};
            \draw (0,-2) node[draw,circle] (t3) {$2$};
            \draw (t1)--(t2)--(t3);
            \draw (0.25,-3) node (tub) {$\{2\},\{2,3\}$};
        \end{scope}
        \begin{scope}[xshift = 8cm]
            \draw (0,0) node[draw, circle] (t1) {$3$};
            \draw (-1,-1) node[draw,circle] (t2) {$1$};
            \draw (0,-2) node[draw,circle] (t3) {$2$};
            \draw (t1)--(t2)--(t3);
            \draw (-0.5,-3) node (tub) {$\{2\},\{1,2\}$};
        \end{scope}
        \begin{scope}[xshift = 12cm]
            \draw (0,0) node[draw, circle] (t1) {$3$};
            \draw (-1,-1) node[draw,circle] (t2) {$2$};
            \draw (-2,-2) node[draw,circle] (t3) {$1$};
            \draw (t1)--(t2)--(t3);
            \draw (-1,-3) node (tub) {$\{1\},\{1,2\}$};
        \end{scope}
        \begin{scope}[xshift = 14.5cm]
            \draw (0,0) node[draw, circle] (t1) {$2$};
            \draw (1,-1) node[draw,circle] (t2) {$3$};
            \draw (-1,-1) node[draw,circle] (t3) {$1$};
            \draw (t1)--(t2);
            \draw (t1)--(t3);
            
            \draw (0,-3) node (tub) {$\{1\},\{3\}$};
        \end{scope}
    \end{tikzpicture}
\end{center}
\end{example}

\begin{definition}\label{def:zipper}
    Given a binary search tree~$B$, the \emph{left zipper} of~$B$ is the saturated chain~$a = (a_1,\ldots,a_l)$ from~$a_1 = 1$ up to~$a_l$ the left child of the root~$m$. The \emph{right zipper} of~$B$ is the saturated chain~$b = (b_1,\ldots,b_r)$ from~$b_1 = n$ up to~$b_r$ the right child of the root~$m$.
\end{definition}
The order relations on the left and right zipper are  
\[1=a_1  < a_2 <\cdots < a_{l-1} < a_l < m < b_r < b_{r-1} < \cdots < b_2 < b_1 = n.\] \par

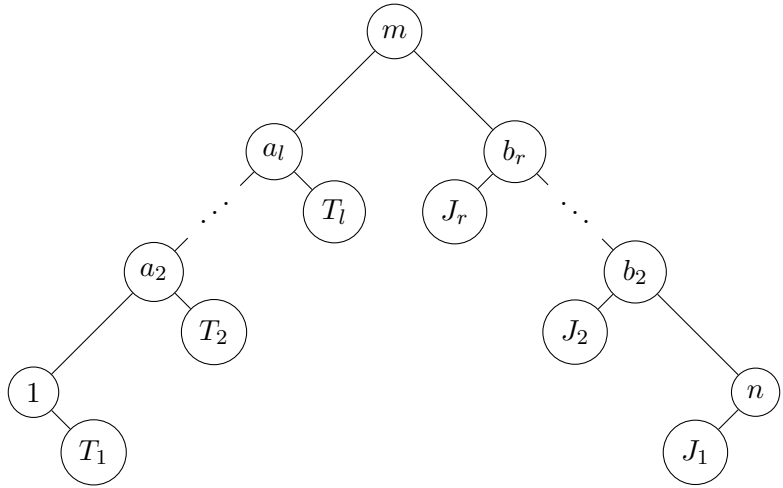
\begin{figure}[htbp!]
    \centering
    \begin{tikzpicture}[scale=0.8]
        \draw (0,0) node[draw,circle] (m) {$m$};
        
        \draw (-2,-2) node[draw,circle] (al) {$a_l$};
        \draw (-3,-3) node (ldts) {\raisebox{6pt}{$\iddots$}};
        \draw (-4,-4) node[draw,circle] (a2) {$a_2$};
        \draw (-6,-6) node[draw,circle] (a1) {$1$};

        \draw (2,-2) node[draw,circle] (br) {$b_r$};
        \draw (3,-3) node (rdts) {\raisebox{6pt}{$\ddots$}};
        \draw (4,-4) node[draw,circle] (b2) {$b_2$};
        \draw (6,-6) node[draw,circle] (b1) {$n$};

        \draw (a1)--(a2)--(ldts)--(al)--(m)--(br)--(rdts)--(b2)--(b1);

        \draw (-1,-3) node[draw,circle] (Tl)  {$T_l$};
        \draw (-3,-5) node[draw,circle] (T2)  {$T_2$};
        \draw (-5,-7) node[draw,circle] (T1) {$T_1$};
        \draw (1,-3) node[draw,circle] (Jr)  {$J_r$};
        \draw (3,-5) node[draw,circle] (J2)  {$J_2$};
        \draw (5,-7) node[draw,circle] (J1) {$J_1$};

        \draw (a1)--(T1);
        \draw (a2)--(T2);
        \draw (al)--(Tl);
        \draw (b1)--(J1);
        \draw (b2)--(J2);
        \draw (br)--(Jr);
        
    \end{tikzpicture}
    \caption{A general structure for the~$G$-tree of an element of~$\MTub(P_n)$.}
    \label{fig:path_zipper}
\end{figure}

Let~$B$ be a binary search tree on~$[n]$.  For each of the~$n-1$ edges in~$B$ we next define a corresponding \emph{tree move}, a ``local'' move that results in a new binary tree. 
\begin{definition}\label{def: tree move}
Let~$B$ be a binary search tree on~$[n]$, and~$x < y \in [n]$ be such that~$x$ is a left child of~$y$ in~$B$. Let~$T_1$ be the left child subtree of~$x$,~$T_2$ the right child subtree of~$x$, and~$T_3$ the right child subtree of~$y$ (see Figure \ref{fig:tree_move}).  Form the binary tree~$B'$ from~$B$ as follows: replace~$y$ with~$x$, make~$T_1$ the left child subtree of~$x$ and~$y$ its right child, and make~$T_2$ and~$T_3$ the left and right child subtrees of~$y$ respectively. We will refer to going from~$B$ to~$B'$ as performing a ``tree move''.
\end{definition}
\begin{figure}[htbp!]
    \centering
     \begin{tikzpicture}
      \begin{scope}
         \draw (0,0.5) node[draw, circle] (T4)     {$T_4$};
         \draw (0,-1) node[draw, circle] (y)     {$y$};
         \draw (-1.5,-2) node[draw, circle] (x)  {$x$};
         \draw (-2.5,-3) node[draw, circle] (T1) {$T_1$};
         \draw (-0.5,-3) node[draw, circle] (T2) {$T_2$};
         \draw (1,-2) node[draw, circle] (T3) {$T_3$};

         \draw (T4)--(y)--(x)--(T1);
         \draw (x)--(T2);
         \draw (y)--(T3);
     \end{scope} 
     \begin{scope}[xshift = 1.5in, yshift = -0.5in]
         \draw (0,0) node {\Large${ \longleftrightarrow}$};
     \end{scope}
     \begin{scope}[xshift = 3in]
        \draw (0,0.4) node[draw, circle] (T4)     {$T_4$};
         \draw (1.5,-2) node[draw, circle] (y)     {$y$};
         \draw (0,-1) node[draw, circle] (x)  {$x$};
         \draw (-1,-2) node[draw, circle] (T1) {$T_1$};
         \draw (0.5,-3) node[draw, circle] (T2) {$T_2$};
         \draw (2.5,-3) node[draw, circle] (T3) {$T_3$};

         \draw (T4)--(x)--(y)--(T3);
         \draw (x)--(T1);
         \draw (y)--(T2);
     \end{scope}
 \end{tikzpicture}
    \caption{A tree move on a binary search tree.}
    \label{fig:tree_move}
\end{figure}
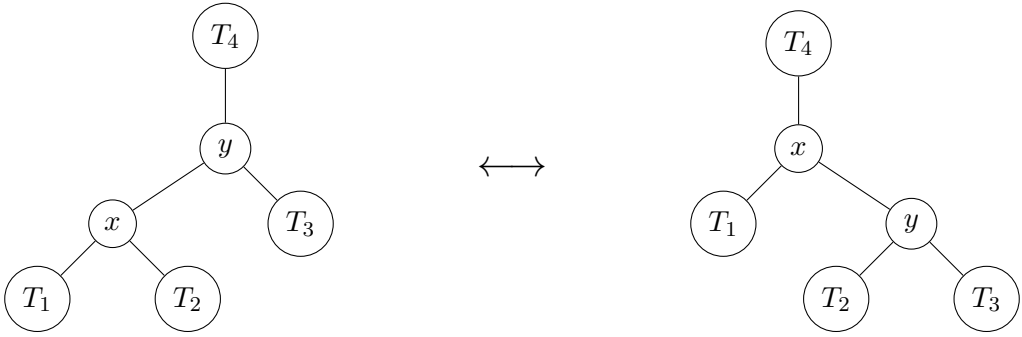

\begin{remark}  The fact that~$B'$ is binary search tree is well known, see for example \cite[\S13.2]{Algorithms}.
\end{remark}

 The move for a right edge is the reverse of this process.

\begin{proposition}\label{prop: cover_relation_path}
Let~$\cx,\cy\in\MTub(P_n)$, with~$G$-trees~$G_\cx, G_\cy$ related by a tree move on a left edge in~$G_\cx$ as in Definition~\ref{def: tree move}. Then~$\cx \lessdot \cy$ is a cover relation in~$\MTub(P_n)$.
\end{proposition}
\begin{proof}
    The principal lower order ideals of $G_\cx$ and $G_\cy$ differ by precisely one: the set~$\{x\}\cup T_1 \cup T_2$ is exchanged for~$\{y\} \cup T_2 \cup T_3$.\par 
    This represents the exchange of the tube~${x}\cup T_1 \cup T_2$ with the tube~$\{y\} \cup T_2 \cup T_3$, and so it is a cover relation in $\MTub(P_n)$. Specifically, a tree move along an edge~$(x <_{\cx} y)$ corresponds to the flip at~$\downX(x)$. \par 
    There is precisely one way to flip a particular tube in a maximal tubing, and since there are~$n-1$ of these tree moves and a maximal tubing of~$P_n$ contains~$n-1$ nontrivial tubes, these tree moves give \emph{every} cover relation. Furthermore, it is straightforward to determine which element covers the other by observing whether~$x<y$ or~$x>y$, as~$x$ and~$y$ are the tops of the exchanged tubes. 
\end{proof}
The global relations on $\MTub(P_n)$ are well known.
\begin{proposition}[cf. \cite{BjornerWachs_shellableII}, \S9]\label{prop:path_order}
    Let $\cx$ and $\cy$ be maximal tubings in $\MTub(P_n)$. The following are equivalent.
    \begin{itemize}
        \item $\cx \leq \cy$,
        \item $\inv(\cx) \subseteq \inv(\cy)$,
        \item $\inv(\cx) \subseteq \inv(\cy) \cup \inc(\cy)$,
        \item $\inv(\cx) \cup \inc(\cx) \subseteq \inv(\cy) \cup \inc(\cy)$, and
        \item $\coinv(\cx) \supseteq \coinv(\cy)$. 
    \end{itemize}
\end{proposition}
\begin{remark}
    Given a tubing $\cx \in \MTub(P_n)$, there are permutations in the symmetric group $\Sfrak_n$ whose inversion sets are $\inv(\cx)$ and $\inv(\cx) \cup\inc(\cx)$. They are often referred to as $\pidown$ and $\piup$ respectively.
\end{remark}

\section{\texorpdfstring{$G$}{G}-trees for maximal tubings of \texorpdfstring{$C_n$}{Cn}}\label{sec:Gtrees for cycle}

This section characterizes the $G$-trees for maximal tubings of~$C_n$, and gives an operation on those $G$-trees that corresponds to cover relations in~$\MTub(C_n)$. We will use $\cj,\ck,\cl,\cm$ for maximal tubings of $C_n$.  \par

\begin{definition}
Let~$m\in[n]$. Define a \emph{cyclic order}~$<_m$ on~$[n]\setminus m$ by
\[
m+1 <_m \cdots <_m n <_m 1 <_m \cdots <_m m-1.
\]
\end{definition}
\begin{definition}
A \emph{cyclic binary tree} is a rooted binary tree on~$[n]$ with root~$m$, where  the unique child subtree of~$m$ is a binary search tree with respect to the \emph{cyclic order}~$<_m$. 
\end{definition}

\begin{remark}
    Note if~$m=n$, then the cyclic order~$\leq_m$ is the same as the standard order~$\leq$ on~$[n-1]$.
\end{remark}

\begin{proposition}\label{prop:cyclic_BST}Let~$\cj\in\MTub(C_n)$ with~$G$-tree~$G_\cj$. Then~$G_\cj$ is a cyclic binary tree. Every cyclic binary tree is the~$G$-tree for a maximal tubing of the cycle graph.
\end{proposition}
\begin{proof}
In any maximal tubing of~$C_n$, the largest nontrivial tube~$H$ is a connected subgraph of~$C_n$ of size~$n-1$.
This subgraph is isomorphic to the path graph~$P_{n-1}$ via a cyclic shift on~$[n]$, and is uniquely characterized by the single element $m$ in~$C_n \setminus H$.\par 
In any cyclic binary tree, if one removes the maximal element $m$ and cyclically shifts the remaining elements in the binary tree to be the set $[n-1]$, the result is a binary search tree. \par 
These cyclic shift operations are the same, and so by Proposition \ref{prop: BST} we have the claim.
\end{proof}

Proposition~\ref{prop:cyclic_BST} allows us to show the following (already known) fact.
\begin{corollary}
The number of~$G$-trees for~$C_{n+1}$ is~$\binom{2n}{n}$.
\end{corollary}
\begin{proof}
    The number of binary search trees on~$[n]$ is~$\frac{1}{n+1}\binom{2n}{n}$. Every~$G$-tree is a choice of a number in~$[n+1]$ paired with any possible binary search tree. 
\end{proof}

\begin{example}\label{ex: G-trees for C4}
The following are all the~$G$-trees for maximal tubings of~$C_4$ such that the least nested element is 4. Note that for each when the tube~$[4]$ is removed, we get a tubing of~$P_3$ (see Example~\ref{ex: G-Trees for P3}).
\begin{center}
    \begin{tikzpicture}
        \begin{scope}
            \draw (0,0) node[draw, circle] (t1) {$1$};
            \draw (1,-1) node[draw,circle] (t2) {$2$};
            \draw (2,-2) node[draw,circle] (t3) {$3$};
            \draw (t1)--(t2)--(t3);
            \draw (0,1) node[draw,circle] (c) {$4$};
            \draw (c)--(t1);
        \end{scope}
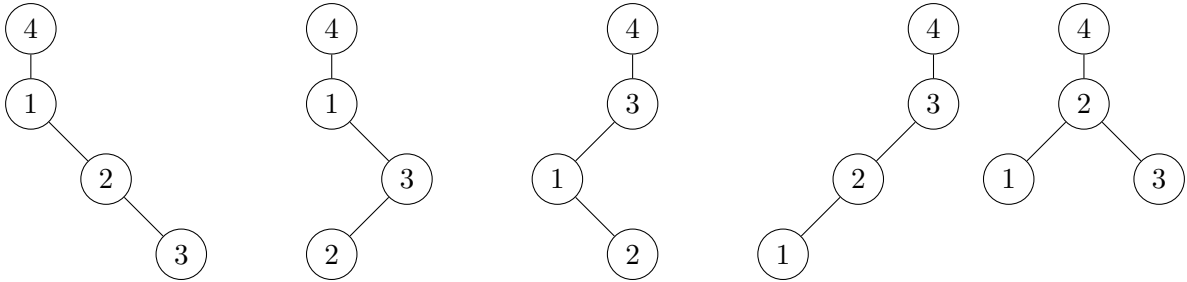
        \begin{scope}[xshift = 4cm]
            \draw (0,0) node[draw, circle] (t1) {$1$};
            \draw (1,-1) node[draw,circle] (t2) {$3$};
            \draw (0,-2) node[draw,circle] (t3) {$2$};
            \draw (t1)--(t2)--(t3);
            \draw (0,1) node[draw,circle] (c) {$4$};
            \draw (c)--(t1);
        \end{scope}
        \begin{scope}[xshift = 8cm]
            \draw (0,0) node[draw, circle] (t1) {$3$};
            \draw (-1,-1) node[draw,circle] (t2) {$1$};
            \draw (0,-2) node[draw,circle] (t3) {$2$};
            \draw (t1)--(t2)--(t3);
            \draw (0,1) node[draw,circle] (c) {$4$};
            \draw (c)--(t1);
        \end{scope}
        \begin{scope}[xshift = 12cm]
            \draw (0,0) node[draw, circle] (t1) {$3$};
            \draw (-1,-1) node[draw,circle] (t2) {$2$};
            \draw (-2,-2) node[draw,circle] (t3) {$1$};
            \draw (t1)--(t2)--(t3);
            \draw (0,1) node[draw,circle] (c) {$4$};
            \draw (c)--(t1);
        \end{scope}
        \begin{scope}[xshift = 14cm]
            \draw (0,0) node[draw, circle] (t1) {$2$};
            \draw (1,-1) node[draw,circle] (t2) {$3$};
            \draw (-1,-1) node[draw,circle] (t3) {$1$};
            \draw (t1)--(t2);
            \draw (t1)--(t3);
            \draw (0,1) node[draw,circle] (c) {$4$};
            \draw (c)--(t1);
        \end{scope}
    \end{tikzpicture}
\end{center}
Applying all possible cyclic shifts gives the remaining~$G$-trees for maximal tubings of~$C_4$, drawn below.  Trees that are related by cyclic shifts are displayed in vertical columns.
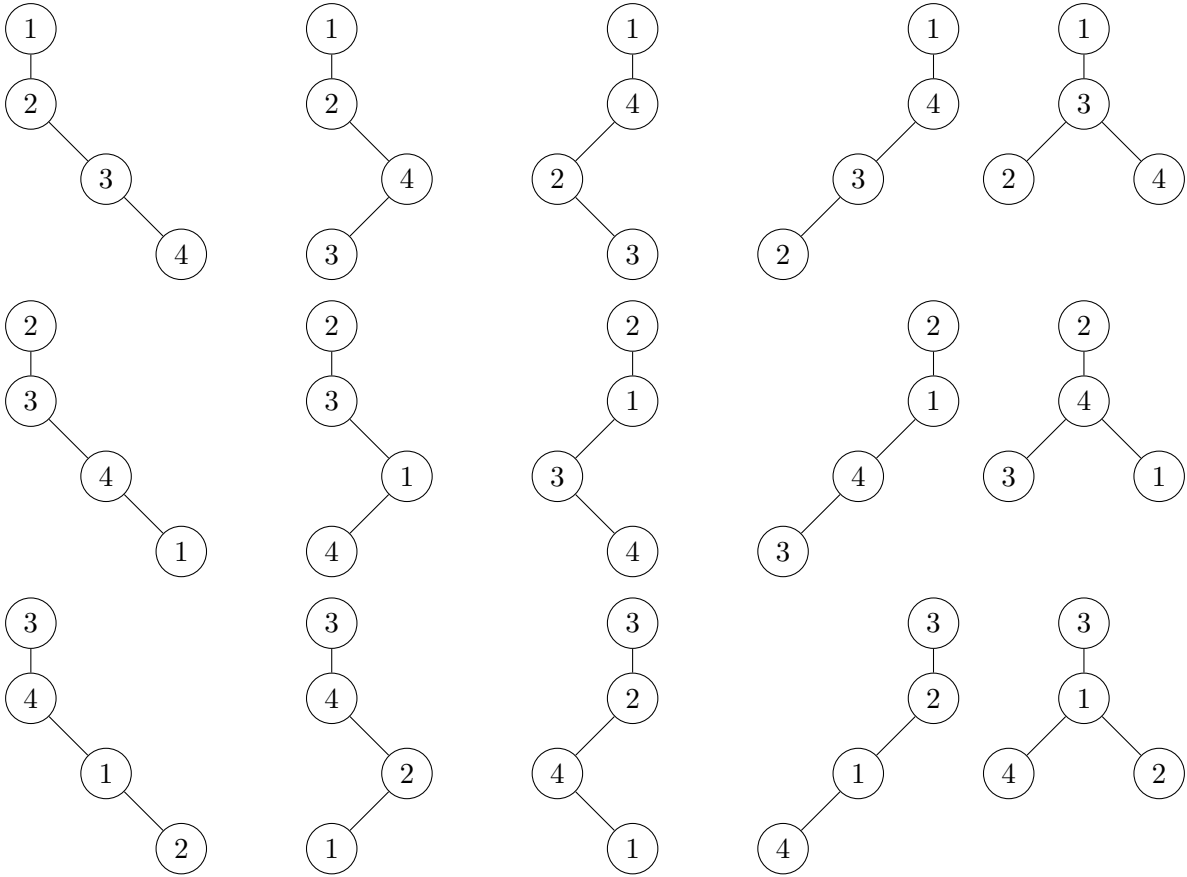
\begin{center}
    \begin{tikzpicture}
        \begin{scope}
            \draw (0,0) node[draw, circle] (t1) {$2$};
            \draw (1,-1) node[draw,circle] (t2) {$3$};
            \draw (2,-2) node[draw,circle] (t3) {$4$};
            \draw (t1)--(t2)--(t3);
            \draw (0,1) node[draw,circle] (c) {$1$};
            \draw (c)--(t1);
        \end{scope}
        \begin{scope}[xshift = 4cm]
            \draw (0,0) node[draw, circle] (t1) {$2$};
            \draw (1,-1) node[draw,circle] (t2) {$4$};
            \draw (0,-2) node[draw,circle] (t3) {$3$};
            \draw (t1)--(t2)--(t3);
            \draw (0,1) node[draw,circle] (c) {$1$};
            \draw (c)--(t1);
        \end{scope}
        \begin{scope}[xshift = 8cm]
            \draw (0,0) node[draw, circle] (t1) {$4$};
            \draw (-1,-1) node[draw,circle] (t2) {$2$};
            \draw (0,-2) node[draw,circle] (t3) {$3$};
            \draw (t1)--(t2)--(t3);
            \draw (0,1) node[draw,circle] (c) {$1$};
            \draw (c)--(t1);
        \end{scope}
        \begin{scope}[xshift = 12cm]
            \draw (0,0) node[draw, circle] (t1) {$4$};
            \draw (-1,-1) node[draw,circle] (t2) {$3$};
            \draw (-2,-2) node[draw,circle] (t3) {$2$};
            \draw (t1)--(t2)--(t3);
            \draw (0,1) node[draw,circle] (c) {$1$};
            \draw (c)--(t1);
        \end{scope}
        \begin{scope}[xshift = 14cm]
            \draw (0,0) node[draw, circle] (t1) {$3$};
            \draw (1,-1) node[draw,circle] (t2) {$4$};
            \draw (-1,-1) node[draw,circle] (t3) {$2$};
            \draw (t1)--(t2);
            \draw (t1)--(t3);
            \draw (0,1) node[draw,circle] (c) {$1$};
            \draw (c)--(t1);
        \end{scope}
    \end{tikzpicture}
\end{center}
\begin{center}
    \begin{tikzpicture}
        \begin{scope}
            \draw (0,0) node[draw, circle] (t1) {$3$};
            \draw (1,-1) node[draw,circle] (t2) {$4$};
            \draw (2,-2) node[draw,circle] (t3) {$1$};
            \draw (t1)--(t2)--(t3);
            \draw (0,1) node[draw,circle] (c) {$2$};
            \draw (c)--(t1);
        \end{scope}
        \begin{scope}[xshift = 4cm]
            \draw (0,0) node[draw, circle] (t1) {$3$};
            \draw (1,-1) node[draw,circle] (t2) {$1$};
            \draw (0,-2) node[draw,circle] (t3) {$4$};
            \draw (t1)--(t2)--(t3);
            \draw (0,1) node[draw,circle] (c) {$2$};
            \draw (c)--(t1);
        \end{scope}
        \begin{scope}[xshift = 8cm]
            \draw (0,0) node[draw, circle] (t1) {$1$};
            \draw (-1,-1) node[draw,circle] (t2) {$3$};
            \draw (0,-2) node[draw,circle] (t3) {$4$};
            \draw (t1)--(t2)--(t3);
            \draw (0,1) node[draw,circle] (c) {$2$};
            \draw (c)--(t1);
        \end{scope}
        \begin{scope}[xshift = 12cm]
            \draw (0,0) node[draw, circle] (t1) {$1$};
            \draw (-1,-1) node[draw,circle] (t2) {$4$};
            \draw (-2,-2) node[draw,circle] (t3) {$3$};
            \draw (t1)--(t2)--(t3);
            \draw (0,1) node[draw,circle] (c) {$2$};
            \draw (c)--(t1);
        \end{scope}
        \begin{scope}[xshift = 14cm]
            \draw (0,0) node[draw, circle] (t1) {$4$};
            \draw (1,-1) node[draw,circle] (t2) {$1$};
            \draw (-1,-1) node[draw,circle] (t3) {$3$};
            \draw (t1)--(t2);
            \draw (t1)--(t3);
            \draw (0,1) node[draw,circle] (c) {$2$};
            \draw (c)--(t1);
        \end{scope}
    \end{tikzpicture}
\end{center}
\begin{center}
    \begin{tikzpicture}
        \begin{scope}
            \draw (0,0) node[draw, circle] (t1) {$4$};
            \draw (1,-1) node[draw,circle] (t2) {$1$};
            \draw (2,-2) node[draw,circle] (t3) {$2$};
            \draw (t1)--(t2)--(t3);
            \draw (0,1) node[draw,circle] (c) {$3$};
            \draw (c)--(t1);
        \end{scope}
        \begin{scope}[xshift = 4cm]
            \draw (0,0) node[draw, circle] (t1) {$4$};
            \draw (1,-1) node[draw,circle] (t2) {$2$};
            \draw (0,-2) node[draw,circle] (t3) {$1$};
            \draw (t1)--(t2)--(t3);
            \draw (0,1) node[draw,circle] (c) {$3$};
            \draw (c)--(t1);
        \end{scope}
        \begin{scope}[xshift = 8cm]
            \draw (0,0) node[draw, circle] (t1) {$2$};
            \draw (-1,-1) node[draw,circle] (t2) {$4$};
            \draw (0,-2) node[draw,circle] (t3) {$1$};
            \draw (t1)--(t2)--(t3);
            \draw (0,1) node[draw,circle] (c) {$3$};
            \draw (c)--(t1);
        \end{scope}
        \begin{scope}[xshift = 12cm]
            \draw (0,0) node[draw, circle] (t1) {$2$};
            \draw (-1,-1) node[draw,circle] (t2) {$1$};
            \draw (-2,-2) node[draw,circle] (t3) {$4$};
            \draw (t1)--(t2)--(t3);
            \draw (0,1) node[draw,circle] (c) {$3$};
            \draw (c)--(t1);
        \end{scope}
        \begin{scope}[xshift = 14cm]
            \draw (0,0) node[draw, circle] (t1) {$1$};
            \draw (1,-1) node[draw,circle] (t2) {$2$};
            \draw (-1,-1) node[draw,circle] (t3) {$4$};
            \draw (t1)--(t2);
            \draw (t1)--(t3);
            \draw (0,1) node[draw,circle] (c) {$3$};
            \draw (c)--(t1);
        \end{scope}
    \end{tikzpicture}
\end{center}
\end{example}

\begin{example}\label{ex:cyclic order}
Say that~$m=5$ is the least nested element in a maximal tubing on~$C_9$. The following is one of the possible~$G$-trees for such a tubing, as well as a visual for the cyclic order.
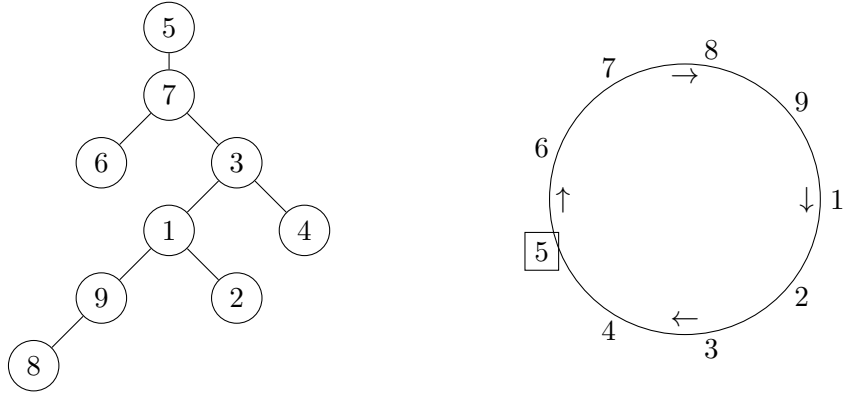
\begin{center}
\begin{tikzpicture}[scale=0.9]
\begin{scope}
    \draw (0,0) node[draw,circle] (v5) {$5$};
    \draw (0,-1) node[draw,circle] (v7) {$7$};
    \draw (1,-2) node[draw,circle] (v3) {$3$};
    \draw (0,-3) node[draw,circle] (v1) {$1$};
    \draw (-1,-4) node[draw,circle] (v9) {$9$};
    \draw (v9)--(v1)--(v3)--(v7)--(v5);

    \draw (-1,-2) node[draw,circle] (v6) {$6$};
    \draw (2,-3) node[draw,circle] (v4) {$4$};
    \draw (1,-4) node[draw,circle] (v2) {$2$};
    \draw (-2,-5) node[draw,circle] (v8) {$8$};
    \draw (v1)--(v2);
    \draw (v4)--(v3);
    \draw (v6)--(v7);
    \draw (v8)--(v9);
    \end{scope}
\begin{scope}[xshift=3in,yshift=-1in]

    \draw (0,0) circle (2cm);
    
    \draw (0:1.8cm) node (ar1) {$\downarrow$};
    \draw (90:1.8cm) node (ar2) {$\rightarrow$};
    \draw (180:1.8cm) node (ar3) {$\uparrow$};    
    \draw (270:1.8cm) node (ar4) {$\leftarrow$};
    
    \draw (0:2.25cm) node    (v1) {$1$};
    \draw (40:2.25cm) node   (v9) {$9$};
    \draw (80:2.25cm) node   (v8) {$8$};
    \draw (120:2.25cm) node  (v7) {$7$};
    \draw (160:2.25cm) node  (v6) {$6$};
    \draw (200:2.25cm) node[draw]  (v5) {$5$};
    \draw (240:2.25cm) node  (v4) {$4$};
    \draw (280:2.25cm) node  (v3) {$3$};
    \draw (320:2.25cm) node  (v2) {$2$};

\end{scope} 

\end{tikzpicture}
\end{center}
Here, the cyclic order on~$[9] \setminus \{5\}$ is~$6 <_5 7 <_5 8 <_5 9 <_5 1 <_5 2 <_5 3 <_5 4$. The~$G$-tree above is a cyclic binary tree with respect to this order, for example the left sub-tree of~$1$ contains~$8$ and~$9$, both of which are less than~$1$ in this cyclic order. \par 
We encourage the reader to use the visual aid of the circle to depict the cyclic order, as it can illustrate containment conditions. For example, in any maximal tubing of~$C_9$ where~$m=5$, if a tube contains~$7$ and~$2$ it must also contain~$9$, since~$7 <_5 9 <_5 2$.
\end{example}

\begin{definition}\label{def:tree_move_cyclic}
    Let~$B$ be a cyclic binary tree on~$[n]$. A \emph{tree move} on $B$ along an $(x,y)$ that is not the top edge takes precisely the same form as one for a binary search tree (Definition \ref{def: tree move}). The tree move along the top edge $(x,m)$ swaps $x$ with $m$, and swaps the left-subtree of $x$ with the right.
\end{definition}
\begin{figure}[htbp!]
    \centering
     \begin{tikzpicture}
      \begin{scope}
         \draw (0,0.5) node[draw, circle] (m)     {$m$};
         \draw (0,-1) node[draw, circle] (x)     {$x$};
         \draw (-1,-2) node[draw, circle] (T1)  {$T_1$};
         \draw (1,-2) node[draw, circle] (T2) {$T_2$};

         \draw (m)--(x)--(T1);
         \draw (x)--(T2);
     \end{scope} 
     \begin{scope}[xshift = 1.5in, yshift = -0.5in]
         \draw (0,0) node {\Large${ \longleftrightarrow}$};
     \end{scope}
     \begin{scope}[xshift = 3in]
        \draw (0,0.5) node[draw, circle] (m)     {$x$};
         \draw (0,-1) node[draw, circle] (x)     {$m$};
         \draw (-1,-2) node[draw, circle] (T1)  {$T_2$};
         \draw (1,-2) node[draw, circle] (T2) {$T_1$};

         \draw (m)--(x)--(T1);
         \draw (x)--(T2);
     \end{scope}
 \end{tikzpicture}
    \caption{The tree move on $(x,m)$ in a cyclic binary tree.}
    \label{fig:tree_move_cyclic_top}
\end{figure}
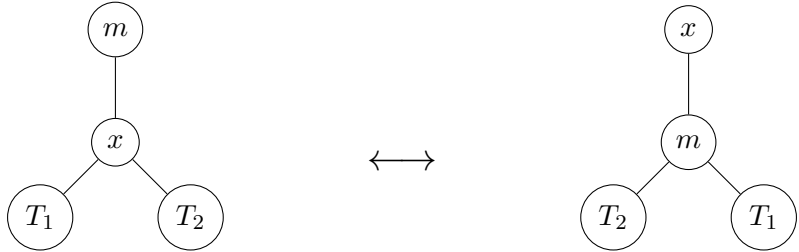
\begin{remark}
    The fact that a tree move in the non-top-edge case results in a cyclic binary tree is precisely the same as for binary search trees. The top-edge move can be checked with Figure \ref{fig:tree_move_cyclic_top}, as $t_1 <_m x <_m t_2$ and $t_2 <_x m <_x t_1$ for all elements $t_1 \in T_1$ and $t_2 \in T_2$, and the cyclic orders within $T_1$ and $T_2$ are identical.
\end{remark}
\begin{proposition}\label{prop:cover_relation_cycle}
Let~$\cj,\ck\in\MTub(C_n)$, with~$G$-trees~$B, B'$ related by a tree move on an edge $(x<y)$ in $B$ where $x <_B y$ as in Definition~\ref{def:tree_move_cyclic}. Then~$\cj \lessdot \ck$ is a cover relation in~$\MTub(C_n)$.
\end{proposition}
\begin{proof}
    In the non-top-edge case this is identical to Proposition \ref{prop: cover_relation_path}. \par 
    In the top edge case, the tubes exchanged are $[n] \setminus\{m\}$ and $n \setminus\{x\}$, and the logic again the same as in Proposition \ref{prop: cover_relation_path}.\par 
    Since there are~$n-1$ of these tree moves and a maximal tubing of~$C_n$ contains~$n-1$ nontrivial tubes, these tree moves give \emph{every} cover relation. Furthermore, it is straightforward to determine which element covers the other by observing whether~$x<y$ or~$x>y$, as~$x$ and~$y$ are the tops of the exchanged tubes. 
\end{proof}

\begin{example}
    Consider the following~$G$-tree for~$\cj$, a maximal tubing of~$C_5$. 
    
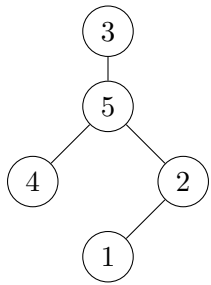
\begin{center}
    \begin{tikzpicture}
    \draw (0,0) node[draw,circle] (v3) {$3$};
    \draw (0,-1) node[draw,circle] (v5) {$5$};
    \draw (1,-2) node[draw,circle] (v2) {$2$};
    \draw (0,-3) node[draw,circle] (v1) {$1$};
    \draw (-1,-2) node[draw,circle] (v4) {$4$};
    
    \draw (v3)--(v5)--(v2)--(v1);
    \draw (v4)--(v5);
    \end{tikzpicture}
    \end{center}
    There are~$4$ edges in this tree, and so~$4$ available tree moves. Below are each of the corresponding new~$G$-trees, the edge along which a tree move was performed to create them, as well as whether the resulting maximal tubing is less than or greater than~$\cj$ in~$\MTub(C_5)$ (i.e. whether the edge is an ascent or descent in~$\cj$).
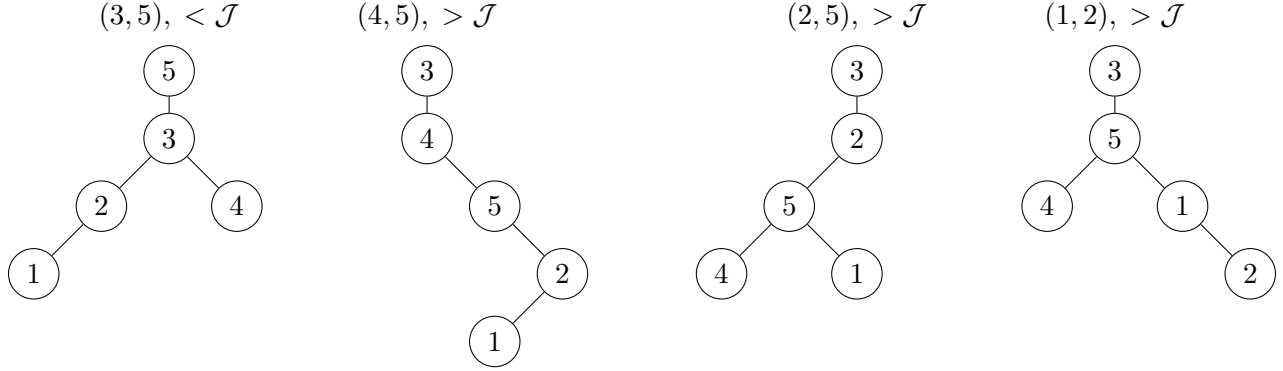
    \begin{center}
\begin{tikzpicture}[scale=0.9]
\begin{scope}[xshift=0in,yshift=0in]
    \draw (0,0.8) node (name) {$(3,5), \; < \cj$};
    \draw (0,0) node[draw,circle] (v3) {$5$};
    \draw (0,-1) node[draw,circle] (v5) {$3$};
    \draw (-1,-2) node[draw,circle] (v2) {$2$};
    \draw (-2,-3) node[draw,circle] (v1) {$1$};
    \draw (1,-2) node[draw,circle] (v4) {$4$};
    
    \draw (v3)--(v5)--(v2)--(v1);
    \draw (v4)--(v5);
\end{scope}

\begin{scope}[xshift=1.5in,yshift=0in]
    \draw (0,0.8) node (name) {$(4,5), \; > \cj$};
    \draw (0,0) node[draw,circle] (v3) {$3$};
    \draw (0,-1) node[draw,circle] (v4) {$4$};
    \draw (1,-2) node[draw,circle] (v5) {$5$};
    \draw (2,-3) node[draw,circle] (v2) {$2$};
    \draw (1,-4) node[draw,circle] (v1) {$1$};
    
    \draw (v3)--(v4)--(v5)--(v2)--(v1);
\end{scope}

\begin{scope}[xshift=4in,yshift=0in]
    \draw (0,0.8) node (name) {$(2,5), \; > \cj$};
    \draw (0,0) node[draw,circle] (v3) {$3$};
    \draw (0,-1) node[draw,circle] (v2) {$2$};
    \draw (-2,-3) node[draw,circle] (v4) {$4$};
    \draw (0,-3) node[draw,circle] (v1) {$1$};
    \draw (-1,-2) node[draw,circle] (v5) {$5$};
    
    \draw (v3)--(v2)--(v5)--(v1);
    \draw (v4)--(v5);
\end{scope}

\begin{scope}[xshift=5.5in,yshift=0in]
    \draw (0,0.8) node (name) {$(1,2), \; > \cj$};
    \draw (0,0) node[draw,circle] (v3) {$3$};
    \draw (0,-1) node[draw,circle] (v5) {$5$};
    \draw (1,-2) node[draw,circle] (v2) {$1$};
    \draw (2,-3) node[draw,circle] (v1) {$2$};
    \draw (-1,-2) node[draw,circle] (v4) {$4$};
    
    \draw (v3)--(v5)--(v2)--(v1);
    \draw (v4)--(v5);
\end{scope}

\end{tikzpicture}    
\end{center}

\end{example}

\begin{remark}\label{rem:minmax_gtree}
    The~$G$-trees for the unique minimal and maximal tubings in~$\MTub(C_n)$ are the total orders~$(1,...,n)$ and~$(n,...,1)$. 
\end{remark}

\section{Cut Map}\label{sec:Cut Map}

In this section we define two useful maps and describe some of their properties. The first map~$\Cut$, given in Definition~\ref{def:cutmap}, is a  surjective map  from~$\MTub(C_n)$ to~$\MTub(P_n)$. The~$\Cut$ map is not injective, but our second map~$\Sew_\cx$, given in Definition~\ref{def:Psi}, takes a tubing~$\cx\in\MTub(P_n)$ and an in-order shuffling of the left and right zipper of~$\cx$ and constructs an element in the fiber of~$\cx$ under the~$\Cut$ map.  \par 
\begin{definition}\label{def:cutmap}
    Let~$\Cut$ be the map
    \begin{align*}
        \Cut \colon \MTub(C_n) &\to \MTub(P_n)\\
        \cj = \{X_1,...,X_n\} &\mapsto \{\cut_\cj(X_1),...,\cut_\cj(X_n)\}
    \end{align*}
where if~$m$ is the root of the~$G$-tree of~$\cj$, then 
\[
\cut_\cj(\downJ(x)) \coloneqq \begin{cases}
    \downJ(x) &\text{if } x=m,\\
     \downJ(x) \cap [1,m-1] &\text{if }x < m, \\
    \downJ(x) \cap [m+1,n]  &\text{if } x > m,\\
\end{cases}
\]
where $<$ and $>$ indicate the usual ordering on the integers.
\end{definition}
\begin{remark}
    Let~$\cj\in\MTub(C_n)$. To visualize~$\cut(\cj)$, we imagine taking a pair of scissors and ``snipping'' the edge of~$C_n$ between~$1$ and~$n$.  The snipping turns~$C_n$ into~$P_n$, and also splits any nontrivial tubes that traverse this edge into two pieces, one piece that contains~$1$ and one that contains~$n$.  
    
    For proper downsets of~$\cj$ that contain both~$1$ and~$n$, this snipping creates two new tubes; the map~$\cut_\cj$ applied to $\downJ(x)$ selects the tube containing $x$. 
    If~$X \in \cj$ is a nontrivial tube that does \emph{not} contain both~$1$ and~$n$, then either~$X \subset [1,m-1]$ or~$X \subset [m+1,n]$. Thus,~$\cut_\cj(X) = X$. In other words, the only tubes that get ``cut'' are those that contain both~$1$ and~$n$.
    Hence, the image of~$\Cut$ is a new collection of tubes---one for each tube in~$\cj$.
    
\end{remark}

The following lemma will show that the image of the~$\Cut$ map truly lies in~$\MTub(P_n)$.
\begin{lemma}\label{lem:cutmap}
    Let~$\cj\in\MTub(C_n)$. Then~$\Cut(\cj)\in\MTub(P_n)$.
\end{lemma}
\begin{proof}
    Note that an element~$Y\in\Cut(\cj)$ is the result of restricting a tube to the subgraph~$[1,m-1]$ or~$[m+1,n]$, and therefore~$Y$ is also a tube of~$P_n$.

    Let~$x_1,x_2\in[n]$ with~$x_1\ne x_2$.  Let~$X_1=\downJ(x_1), X_2=\downJ(x_2)$.
    Let~${Y_1 = \cut_\cj(\downJ(x_1))}, {Y_2 = \cut_\cj(\downJ(x_2))}$. We will first argue that~$Y_1$ and~$Y_2$ are compatible tubes. 
    
    First assume~$X_1\cup X_2$ is disconnected in~$C_n$. Since~$Y_1=\cut_\cj(X_1) \subseteq X_1$ and~$Y_2=\cut_\cj(X_2) \subseteq X_2$, the union~$Y_1\cup Y_2$ is disconnected in~$P_n$, making the tubes~$Y_1, Y_2$ compatible. \par 
    Now assume~$X_1$ and~$X_2$ are nested,  without loss of generality~$X_1 \subseteq X_2$. Since~$[1,m-1]\cup[m+1,n]$ is disconnected in~$P_n$, either
    \begin{itemize}
        \item $X_1 \cap [1,m-1] \subset X_2 \cap[1,m-1]$,
        \item $X_1 \cap [m+1,n]\subset X_2 \cap[m+1,n]$, 
        \item $(X_1 \cap [1,m-1])\cup (X_2 \cap[m+1,n])$ is disconnected in $P_n$, or 
        \item $(X_1 \cap [m+1,n]) \cup (X_2 \cap[1,m-1])$ is disconnected in $P_n$.
    \end{itemize}
   In all cases~$Y_1$ and~$Y_2$ are compatible, and thus~$\Cut(\cj)$ is a tubing of~$P_n$.\par 
    Next we will argue that~$\Cut(\cj)$ is a maximal tubing by proving~$|\Cut(\cj)|=n$. In particular, we will show that~$\cut_\cj$ is injective as a map on~$\cj$. If~$\cut_\cj(X_1) = \cut_\cj(X_2)$ then~$X_1 \cap X_2 \neq \emptyset$. Since
    \[x_1 \in \cut_\cj(X_1) = \cut_\cj(X_2) \subseteq X_2\, ,\]
    we can infer that~$X_1=\downJ(x_1)\subseteq X_2$.
    By similar reasoning~$X_2\subseteq X_1$. 
     As a consequence, $\cut_\cj$ is injective, so~${\abs{\Cut(\cj)} = \abs{\cj} = n}$.  Thus,~$\Cut(\cj)\in\MTub(P_n)$.
\end{proof}

\begin{figure}[htbp!]
\begin{subfigure}[b]{0.25\textwidth}
        \centering
\begin{tikzpicture}[scale=.8]
        \node (p1) at (0,0) {$1$};
        \node (p2) at (-1,1) {$2$};
        \node (p3) at (0,2) {$3$};
        \node (p4) at (1.5,2) {$4$};
        \node (p5) at (2.5,1) {$5$};
        \node (p6) at (1.5,0) {$6$};

        \draw (p1) to (p2) to (p3) to (p4) to (p5) to (p6) to (p1);

        \draw[rounded corners] (-1.25,1.25) to (-1,1.5) to (-.5,1) to (-1,.5) to (-1.5,1) to (-1.25,1.25); % tube for 2

         \draw[rounded corners] 
        (-.75,2) to (0,2.75) to (.75,2) to (-1,.25) to (-1.75,1) to (-.75,2); % tube for 2,3

        \draw[rounded corners] (0,.5) to (-.25,.75) to (1,2) to (0,3) to (-2,1) to (0,-.75) to (.5,0) to (0,.5);% tube for 123

         \draw[rounded corners] (2.25,1.25) to (2.5,1.5) to (3,1) to (2.5,.5) to (2,1) to (2.25,1.25); % tube for 6

         \draw[rounded corners] (-1,2.25) to (0,3.25) to (1.25,2) to (0,.75) to (.5,.25) to (1,.25) to (2.5,1.75) to (3.25,1) to 
         (1.5,-1) to (0,-1) to (-2.25,1) to (-1,2.25); % tube for 12356

\end{tikzpicture}
        
        \caption{$\cj\in\MTub(C_n)$}
        \label{subfig: graph tubing example}
    \end{subfigure}
    \begin{subfigure}[b]{0.39\textwidth}
    \centering
\begin{tikzpicture}[scale=0.5,smallnode/.style={outer sep = 0pt}]
        \node[smallnode] (p1) at (0,0) {$1$};
        \node[smallnode] (p2) at (-1,1) {$2$};
        \node[smallnode] (p3) at (0,2) {$3$};
        \node[smallnode] (p4) at (1,2) {$4$};
        \node[smallnode] (p5) at (2,1) {$5$};
        \node[smallnode] (p6) at (1,0) {$6$};

        \draw (p1) to (p2) to (p3) to (p4) to (p5) to (p6) to (p1);

        \draw[rounded corners] (-1.25,1.25) to (-1,1.5) to (-.5,1) to (-1,.5) to (-1.5,1) to (-1.25,1.25);

        \draw[->] (3,1) to node[above] {$\cut_\cj$} (5,1);

        \node[smallnode] (q1) at (6,1) {$1$};
        \node[smallnode] (q2) at (7,1) {$2$};
        \node[smallnode] (q3) at (8,1) {$3$};
        \node[smallnode] (q4) at (9,1) {$4$};
        \node[smallnode] (q5) at (10,1) {$5$};
        \node[smallnode] (q6) at (11,1) {$6$};

        \draw (q1) to (q2) to (q3) to (q4) to (q5) to (q6);

        \draw[rounded corners] (6.75,1.25) to (7,1.5) to (7.5,1) to (7,.5) to (6.5,1) to (6.75,1.25);

    \end{tikzpicture}
    
    \vspace{.2in}
    
    \begin{tikzpicture}[scale=0.5,smallnode/.style={outer sep = 0pt}]
        \node[smallnode] (p1) at (0,0) {$1$};
        \node[smallnode] (p2) at (-1,1) {$2$};
        \node[smallnode] (p3) at (0,2) {$3$};
        \node[smallnode] (p4) at (1,2) {$4$};
        \node[smallnode] (p5) at (2,1) {$5$};
        \node[smallnode] (p6) at (1,0) {$6$};

        \draw (p1) to (p2) to (p3) to (p4) to (p5) to (p6) to (p1);

        \draw[rounded corners] 
        (-.5,2) to (0,2.5) to (.5,2) to (-1,.5) to (-1.5,1) to (-.5,2);
        
        \draw[->] (3,1) to node[above] {$\cut_\cj$} (5,1);

        \node[smallnode] (q1) at (6,1) {$1$};
        \node[smallnode] (q2) at (7,1) {$2$};
        \node[smallnode] (q3) at (8,1) {$3$};
        \node[smallnode] (q4) at (9,1) {$4$};
        \node[smallnode] (q5) at (10,1) {$5$};
        \node[smallnode] (q6) at (11,1) {$6$};

        \draw (q1) to (q2) to (q3) to (q4) to (q5) to (q6);

        \draw[rounded corners] (7.5,.5) to (7,.5) to (6.5,1) to (7,1.5) to (8,1.5) to (8.5,1) to (8,.5) to (7.5,.5);

    \end{tikzpicture}

    \vspace{.2in}
    
    \begin{tikzpicture}[scale=0.5,smallnode/.style={outer sep = 0pt}]
        \node[smallnode] (p1) at (0,0) {$1$};
        \node[smallnode] (p2) at (-1,1) {$2$};
        \node[smallnode] (p3) at (0,2) {$3$};
        \node[smallnode] (p4) at (1,2) {$4$};
        \node[smallnode] (p5) at (2,1) {$5$};
        \node[smallnode] (p6) at (1,0) {$6$};

        \draw (p1) to (p2) to (p3) to (p4) to (p5) to (p6) to (p1);

        \draw[rounded corners] (-.25,.75) to (-.5,1) to (.5,2) to (0,2.5) to (-1.5,1) to (0,-.5) to (.5,0) to (-.25,.75);

        \draw[->] (3,1) to node[above] {$\cut_\cj$} (5,1);

        \node[smallnode] (q1) at (6,1) {$1$};
        \node[smallnode] (q2) at (7,1) {$2$};
        \node[smallnode] (q3) at (8,1) {$3$};
        \node[smallnode] (q4) at (9,1) {$4$};
        \node[smallnode] (q5) at (10,1) {$5$};
        \node[smallnode] (q6) at (11,1) {$6$};

        \draw (q1) to (q2) to (q3) to (q4) to (q5) to (q6);

        \draw[rounded corners] (7,.5) to (6,.5) to (5.5,1) to (6,1.5) to (8,1.5) to (8.5,1) to (8,.5) to (7,.5);

    \end{tikzpicture}

    \vspace{.2in}

    \begin{tikzpicture}[scale=.5,smallnode/.style={outer sep = 0pt}]
        \node[smallnode] (p1) at (0,0) {$1$};
        \node[smallnode] (p2) at (-1,1) {$2$};
        \node[smallnode] (p3) at (0,2) {$3$};
        \node[smallnode] (p4) at (1,2) {$4$};
        \node[smallnode] (p5) at (2,1) {$5$};
        \node[smallnode] (p6) at (1,0) {$6$};

        \draw (p1) to (p2) to (p3) to (p4) to (p5) to (p6) to (p1);
       
        \draw[rounded corners] (1.75,1.25) to (2,1.5) to (2.5,1) to (2,.5) to (1.5,1) to (1.75,1.25);

        \draw[->] (3,1) to node[above] {$\cut_\cj$} (5,1);

        \node[smallnode] (q1) at (6,1) {$1$};
        \node[smallnode] (q2) at (7,1) {$2$};
        \node[smallnode] (q3) at (8,1) {$3$};
        \node[smallnode] (q4) at (9,1) {$4$};
        \node[smallnode] (q5) at (10,1) {$5$};
        \node[smallnode] (q6) at (11,1) {$6$};

        \draw (q1) to (q2) to (q3) to (q4) to (q5) to (q6);

        \draw[rounded corners] (9.75,1.25) to (10,1.5) to (10.5,1) to (10,.5) to (9.5,1) to (9.75,1.25);

    \end{tikzpicture}
    
    \vspace{.2in}
    
    \begin{tikzpicture}[scale=.5]
        \node (p1) at (0,0) {$1$};
        \node (p2) at (-1,1) {$2$};
        \node (p3) at (0,2) {$3$};
        \node (p4) at (1,2) {$4$};
        \node (p5) at (2,1) {$5$};
        \node (p6) at (1,0) {$6$};

        \draw (p1) to (p2) to (p3) to (p4) to (p5) to (p6) to (p1);
        \draw[rounded corners] (-.5,2) to (0,2.5) to (.5,2) to (0,1.5) to (-.5,1) to (0,.5) to (1,.5) to (1.5,1) to (2,1.5) to (2.5,1) to (1,-.5) to (0,-.5) to (-1.5,1) to (-.5,2);

        \draw[->] (3,1) to node[above]{$\cut_\cj$} (5,1);

        \node (q1) at (6,1) {$1$};
        \node (q2) at (7,1) {$2$};
        \node (q3) at (8,1) {$3$};
        \node (q4) at (9,1) {$4$};
        \node (q5) at (10,1) {$5$};
        \node (q6) at (11,1) {$6$};

        \draw (q1) to (q2) to (q3) to (q4) to (q5) to (q6);

        \draw[rounded corners] (10.5,.5) to (10,.5) to (9.5,1) to (10,1.5) to (11,1.5) to (11.5,1) to (11,.5) to (10.5,.5);
    \end{tikzpicture}
        
        \caption{Applying~$\cut_\cj$ to each tube in~$\cj$.}
        \label{subfig: cut of indiv tubes}
    \end{subfigure}
    \begin{subfigure}[b]{0.25\textwidth}
        \centering
        \begin{tikzpicture}[scale=0.7,smallnode/.style={outer sep = 0pt}]

        \node[smallnode] (q1) at (5,1) {$1$};
        \node[smallnode] (q2) at (6,1) {$2$};
        \node[smallnode] (q3) at (7,1) {$3$};
        \node[smallnode] (q4) at (8,1) {$4$};
        \node[smallnode] (q5) at (9,1) {$5$};
        \node[smallnode] (q6) at (10,1) {$6$};

        \draw (q1) to (q2) to (q3) to (q4) to (q5) to (q6);

        \draw[rounded corners] (5.75,1.25) to (6,1.5) to (6.5,1) to (6,.5) to (5.5,1) to (5.75,1.25); %tube for 2

        \draw[rounded corners] (6.5,.25) to (6,.25) to (5.25,1) to (6,1.75) to (7,1.75) to (7.5,1) to (7,.25) to (6.5,.25); %tube for 23

        \draw[rounded corners] (6,0) to (5,0) to (4.5,1) to (5,2) to (7,2) to (7.75,1) to (7,0) to (6,0); %tube for 123 

        \draw[rounded corners] (8.75,1.25) to (9,1.5) to (9.5,1) to (9,.5) to (8.5,1) to (8.75,1.25); %tube for 5

        \draw[rounded corners] (9.5,1.75) to (9,1.75) to (8.25,1) to (9,.25) to (10,.25) to (10.5,1) to (10,1.75) to (9.5,1.75); % tube for 56
        
    \end{tikzpicture}
    \caption{$\Cut(\cj)\in\MTub(P_n)$}
    \label{subfig: result of cut}
    \end{subfigure}
    \caption{An example application of the~$\Cut$ map to a maximal tubing~$\cj\in\MTub(C_n)$.}
    \label{fig: cut example}
\end{figure}
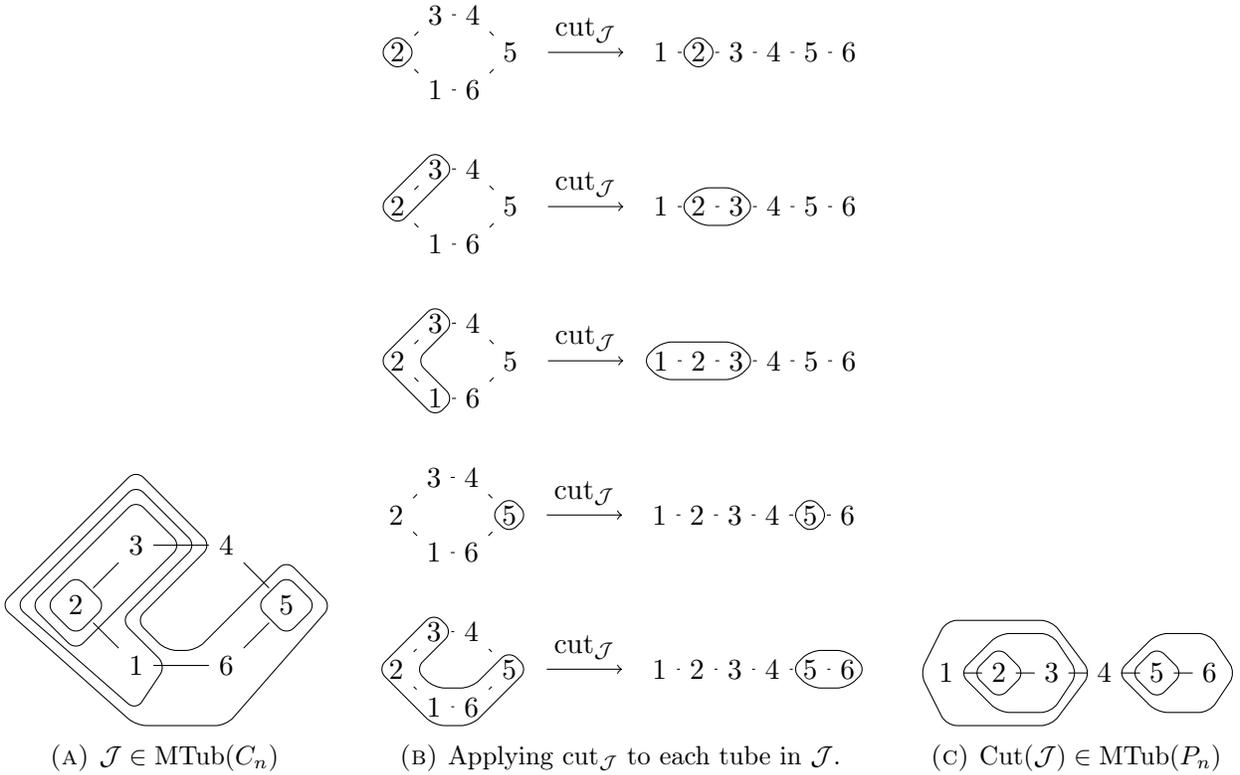

\begin{lemma}\label{lem:cut_tops}
Let~$\cj \in \MTub(C_n)$. Then for all vertices~$x\in C_n$,
    \[\cut_\cj(\downJ(x)) = \Cut(\cj)_\downarrow(x).\] In words, the~$\cut_\cj$ map takes the downset of the element~$x$ in $G_\cj$ to the downset of~$x$ in the~$G$-tree of the image of~$\cj$ under the~$\Cut$ map.
\end{lemma}
\begin{proof}
    Let~$x$ be a vertex of~$C_n$. If~$x = m$ is the least-nested element in~$\cj$, then it is the least-nested element in~$\Cut(\cj)$ as well and the claim holds.\par 
    Assume~$x\ne m$.  Since~$x \in \cut_\cj(\downJ(x))$ and~$\cut_\cj(\downJ(x))$ is a tube in~$\Cut(\cj)$, we have by definition that~$\cut_\cj(\downJ(x)) \supseteq \Cut(\cj)_\downarrow(x)$. On the other hand if~$y \in \cut_\cj(\downJ(x))$ then~$y \in \downJ(x)$ and all tubes~$X \in \cj$ that contain~$x$ also contain~$y$. In addition, if~$y \in \cut_\cj(\downJ(x))$ then~$y < m$ if~$x < m$ and~$y > m$ if~$x>m$. \par
    So if~$X \in \cj$ is any tube in~$\cj$ with~$x \in X$, then~$y \in X$ as well and either
    \begin{itemize}
        \item $x,y \in X \cap [1,m-1]$ or 
        \item $x,y \in X \cap [m+1,n]$. 
    \end{itemize}
    In particular, if~$X \in \cj$ is any tube in~$\cj$ with~$x \in \cut_\cj(X)$, then~$y \in \cut_\cj(X)$. So ~$y \in \Cut(\cj)_\downarrow(x)$, and we have the other direction of containment. 
\end{proof}
Now we will describe the result of the~$\Cut$ map on~$G$-trees for maximal tubings of~$C_n$. If~$G_\cj$ is the~$G$-tree of~$\cj \in \MTub(C_n)$ with maximal element~$m$, the \emph{unzipping}~$T$ of~$G_\cj$ is the poset on~$[n]$ with the following relations:
\begin{itemize}
    \item $a <_{T} b$ if $a <_\cj b$ and $a,b < m$,
    \item $a <_{T} b$ if $a <_\cj b$ and $a,b > m$, and 
    \item $a <_{T} m$ if $a \neq m$.
\end{itemize}
Less formally, the unzipping of~$P$ separates the two induced subposets of~$P$ on~$[1,m-1]$ and~$[m+1,n]$,  keeping~$m$ as the maximal element (see Example~\ref{ex:Cut map fiber}).

\begin{lemma}\label{lem:Cut_Gtrees_forward}
    If~$\cj \in \MTub(C_n)$, then the~$G$-tree of~$\Cut(\cj)$ is the unzipping of $G_\cj$.
\end{lemma}
\begin{proof}
    Let~$T$ be the unzipping of~$G_\cj$. We will prove that the lower order ideal of an element~$x$ in~$T$ is precisely~$\Cut(\cj)_\downarrow(x)$. This will heavily leverage Lemma \ref{lem:cut_tops}.\par 
    If~$x=m$ is the maximal element in $G_\cj$, then the lower order ideal of~$m$ in the unzipping of the~$G$-tree is~$[n]$. It is also the case that~$\Cut(\cj)_\downarrow(m) =\cut_\cj(\downJ(m)) =   \downJ(m) = [n]$ since removing the edge~$(1,n)$ does not disconnect~$C_n$. \par 
    If~$x < m$ then~$\Cut(\cj)_\downarrow(x) = \downJ(x) \cap [1,m-1]$. Since~$x < m$, an element~$a \in \Cut(\cj)_\downarrow(x)$ if and only if~$a < m$ and~$a \in \downJ(x)$ (i.e.~$a \leq_\cj x$). This is precisely the condition for~$a \leq_T x$. \par 
    If~$x > m$, then the same logic applies and~$a \in \Cut(\cj)_\downarrow(x) = \downJ(x) \cap [m+1,n]$ if and only if~$a > m$ and~$a <_\cj x$, which is the condition for~$a <_T x$.\par 
    So for any~$x \in [n]$, the lower order ideal~$\Cut(\cj)_\downarrow(x)$ is equal to the lower order ideal~$T_\downarrow(x)$. So~$T$ is the~$G$-tree of~$\Cut(\cj)$.
\end{proof}

\begin{example}\label{ex:Cut map fiber}
    Figure~\ref{fig: cut and sew map example} shows a maximal tubing 
    \[\cj = \left\{ \{2\},\, \{4\},\, \{6\},\, \{8\},\, \{8,9  \} ,\, \{1,2,8,9  \} ,\, \{1,2,3,4,8,9  \} ,\, \{1,2,3,4,6,7,8,9  \}  ,\, [9]\right\}\in\MTub(C_9)\, ,\]
    and its image
    \[
    \Cut(\cj)=\left\{\{2\},\{4\},\{6\},\{8\},\{8,9\},\{1,2\},\{1,2,3,4\},\{6,7,8,9\},[9] \right\}\in\MTub(P_n)
    \]
    under the~$\Cut$ map.
    The~$G$-tree of~$\Cut(\cj)$ is the unzipping of~$G_\cj$. The linear ordering~$9<1<3<7$ becomes two separate linear orders (the left and right zippers of~$\Cut(\cj)$):~$1<3$ and~$7<9$.

\end{example}

 \begin{figure}[htbp!]
        \centering
        
    \begin{tikzpicture}[scale=0.9,zipnode/.style={fill=black!30}]
    \begin{scope}
        \draw (0,0) node[draw,circle] (v5) {$5$};
    \draw (0,-1) node[draw,circle,zipnode] (v7) {$7$};
    \draw (1,-2) node[draw,circle,zipnode] (v3) {$3$};
    \draw (0,-3) node[draw,circle,zipnode] (v1) {$1$};
    \draw (-1,-4) node[draw,circle,zipnode] (v9) {$9$};
    \draw (v9)--(v1)--(v3)--(v7)--(v5);

    \draw (-1,-2) node[draw,circle] (v6) {$6$};
    \draw (2,-3) node[draw,circle] (v4) {$4$};
    \draw (1,-4) node[draw,circle] (v2) {$2$};
    \draw (-2,-5) node[draw,circle] (v8) {$8$};
    \draw (v1)--(v2);
    \draw (v4)--(v3);
    \draw (v6)--(v7);
    \draw (v8)--(v9);

    \draw (0,-6) node (name) {The~$G$-tree for~$\cj$.};

    \end{scope}
    \begin{scope}[xshift=1.5in]
        \draw[-stealth, very thick] (0,0) to node[below]{$\Cut$} (1,0);
        \draw[-stealth,very thick] (1,-2) to node[below]{$\Sew_{\Cut(\cj)}(9137)$} (0,-2);
    \end{scope}
    \begin{scope}[xshift=3.5in]
         \draw (0,0) node[draw,circle] (m) {$5$};
        
        \draw (-2,-2) node[draw,circle,zipnode] (al) {$3$};
        \draw (-4,-4) node[draw,circle,zipnode] (a2) {$1$};

        \draw (2,-2) node[draw,circle,zipnode] (br) {$7$};
        \draw (4,-4) node[draw,circle,zipnode] (b2) {$9$};

        \draw (a2)--(al)--(m)--(br)--(b2);

        \draw (-1,-3) node[draw,circle] (Tl)  {$4$};
        \draw (-3,-5) node[draw,circle] (T2)  {$2$};
        \draw (1,-3) node[draw,circle] (Jr)  {$6$};
        \draw (3,-5) node[draw,circle] (J2)  {$8$};

        \draw (a2)--(T2);
        \draw (al)--(Tl);
        \draw (b2)--(J2);
        \draw (br)--(Jr);
    \draw (0,-6) node (name) {The~$G$-tree for~$\Cut(\cj)$ is the unzipping of~$\cj$.};
        \end{scope}
    \end{tikzpicture}
 \caption{A~$G$-tree for~$\cj\in\MTub(C_n)$ and a~$G$-tree for~$\Cut(\cj)\in\MTub(P_n)$.  Note that~$\Sew_{\Cut(\cj)}(9137)=\cj$ (see Example~\ref{ex:Cut_map_inverse}.)}
        \label{fig: cut and sew map example}
    \end{figure}

Now we will define the map~$\Sew_\cx$. 
\begin{definition}\label{def:Psi}
    Let~$\cx=(Y_1,\dots,Y_n) \in \MTub(P_n)$. Recalling Definition \ref{def:zipper}, identify the left and right zippers of~$\cx$. Let~$W_\cx$ be the set of all in-order shufflings of the left and right zipper of~$\cx$. An element~$w\in W_\cx$ is a permutation that maintains the order relations of the left and right zippers. 
Define 
\begin{align*}
    \Sew_\cx \colon W_\cx &\to \MTub(C_n)\\
    w&\mapsto \{\sew_\cx(Y_1,w),...,\sew_\cx(Y_n,w)\}
\end{align*} 
where
\[
\sew_\cx(\downX(y),w) \coloneqq \begin{cases}
    \downX(y) &\text{if } y \text{ is not in the zipper, and}\\[5pt]
    \ds \bigcup_{i=1}^{w^{-1}(y)}\downX(w(i)) &\text{if } y \text{ is in the zipper.}
\end{cases}
\]
\end{definition}
\begin{remark}
    Let~$\cx\in\MTub(P_n)$. To visualize~$\Sew_\cx$, we imagine stitching the left and right zippers together in the order given by~$w$, i.e. taking a stitch on the left, then on the right, on the right again, etc. as though sewing a seam. Anything hanging off of the zipper is unaffected, in the image these elements hang off a single new seam (see Figure~\ref{fig: cut and sew map example}).
\end{remark}
\begin{lemma}
    Let~$\cx\in\MTub(P_n)$ and~$w\in W_\cx$. Then~$\Sew_\cx(w)\in\MTub(C_n)$.
\end{lemma}
\begin{proof}
    We will argue that~$\Sew_\cx(w)$ is a maximal tubing in much the same fashion as Lemma \ref{lem:cutmap}. \par 
    Let~$y \in [n]$. 
    Since~$\cx$ is a tubing of~$P_n$,~$\downX(y)$ is connected in~$P_n$ and thus~$\downX(y)$ is connected in~$C_n$.   All sets~$\downX(w(i))$ for~$w(i)$ in the zipper contain~$1$ or~$n$. In particular, every individual set in the union~$\ds \bigcup_{i=1}^{w^{-1}(x)}\downX(w(i))$ contains~$1$ or~$n$. Since~$(1,n)$ is an edge in~$C_n$ and each individual~$\downX(w(i))$ is connected, their union is also connected. This shows that all elements of~$\Sew_\cx(w)$ are tubes.\par 
    Let~$X_1 = \Sew_\cx(\downX(y_1))$  and~$X_2= \Sew_\cx(\downX(y_2))$ be tubes in~$\Sew_\cx(w)$ where~$y_1 \neq y_2$. We will show that~$X_1$ and~$X_2$ are compatible as tubes of~$C_n$ (i.e. the image is a tubing) and that~$X_1 \neq X_2$ (i.e.~$\sew_\cx$ is injective) by cases. Note that~$y_1=m$ if and only if~$\sew_\cx(\downX(y_1),w) = [n]$, so then~$X_2 \subsetneq X_1$, and vice versa, so we let~$y_1,y_2 \neq m$.\par 
    (Case 1) If neither~$y_1$ nor~$y_2$ is in the zipper then~$\sew_\cx(Y_1,w) = Y_1 \neq Y_2 = \sew_\cx(Y_2,w)$. If one is contained in the other that will not change. Similarly, if~$Y_1 ,Y_2 \neq [n]$ and~$y_1,y_2$ are not in the zipper, both~$Y_1$ and~$Y_2$ contain neither~$1$ nor~$n$. In particular, if~$Y_1$ and~$Y_2$ are disconnected in~$P_n$ then they are disconnected in~$C_n$.\par 
    (Case 2) Assume~$y_1 = w(j)$ is in the zipper (so~$X_1 = \bigcup_{i=1}^{w^{-1}(y_1)}\downX(w(i))$ but~$y_2$ is not (so~$X_2 = \downX(y_2)$). Let~$k$ be the minimal element of~$[r+l]$ so that~$y_2 \in \downX(w(k))$ (i.e. the smallest tube in~$\cx$ that contains both~$y_2$ and either~$1$ or~$n$). If~$k = j$ then~$\downX(y_2) \subset \downX(y_1)$, and so~$X_1 \subset \downX(y_1) \subseteq X_2$. If~$k < j$ then 
    \[
    X_2 = \downX(y_2) \subset \downX(w(k)) \subset \bigcup_{i=1}^{w^{-1}(y_1)}\downX(w(i)) = X_1.
    \]
    In particular,~$X_2 \subset X_1$. \par 
    If~$k > j$ then~$X_2 = \downX(y_2)$ is disconnected in~$P_n$ from all~$\downX(w(i))$ for~$i \leq j$, and so is disconnected from their union. So~$X_1$ and~$X_2$ are disconnected in~$P_n$. Since~$1,n \notin X_2$,~$X_1$ and~$X_2$ are disconnected in~$C_n$ as well. \par 
    (Case 3) If both~$y_1$ and~$y_2$ are in the zipper, then~$\sew_\cx(Y_1,w) \subset \sew_\cx(Y_2,w)$ if and only if~$w^{-1}(y_1) < w^{-1}(y_2)$ and vice versa. Since one of $X_1,X_2$ is always contained within the other, these tubes of~$C_n$ are compatible.\par 
    Since in any case~$X_1$ and~$X_2$ are compatible in~$C_n$ and~$X_1 \neq X_2$, the set~$\Sew_\cx(w)$ is a maximal tubing of~$C_n$.
\end{proof}

\begin{lemma}\label{lem:cutinv_downsets}
    Let~$\cx \in \MTub(P_n)$,~$w \in W_\cx$ and~$y \in [n]$. Then 
    \[\sew_\cx(\downX(y),w) = \Sew_\cx(w)_\downarrow(y).\]
    In words, the downset of~$y$ in~$\Sew_\cx(w)$ is the same as the image under~$\sew_\cx$ of the downset of~$\downX(y)$ and~$w$.  Loosely speaking, the~$\Sew_\cx$ map takes downsets to downsets.
\end{lemma}
\begin{proof}
    Define the map 
    \begin{align*}   
    \phi \colon [n] &\to [n] \\
    y &\mapsto \topT_{\Sew_\cx(w)}(\sew_\cx(\downX(y),w)),
        \end{align*}
    so that~$\phi(y)$ is the element in~$C_n$ whose downset in~$\Sew_\cx(w)$ is~$\sew_\cx(\downX(y),w)$. We seek to prove that~$\phi$ is the identity map.\par 
    Since~$\Sew_\cx(w)$ is a maximal tubing, the map~$\phi$ must be surjective, which means it is also injective. Certainly~${y \in \sew_\cx(\downX(y),w)}$, so by definition
    \[\Sew_\cx(w)_\downarrow(y) \subseteq \sew_\cx(\downX(y),w) = \Sew_\cx(w)_\downarrow(\phi(y))\, ,\]
    and so~$y \leq_{\Sew_\cx(w)} \phi(y)$.  \par 
    Let~$y\in[n]$. Let~$K =\{k \in [n] \mid y \in \Sew_\cx(w)_\downarrow(k)\}$. Since each of the tubes~$\Sew_\cx(w)_\downarrow(k)$ for~$k \in K$ intersect nontrivially,~$K$ is totally ordered by~$\leq_{\Sew_\cx(w)}$. Since~$k\leq_{\Sew_\cx(w)}\phi(k)$, for all~$k \in K$,~$\phi$ is a map from~$K$ to itself. Since~$\phi$ is injective and~$K$ is totally ordered, it must be that~$\phi$ is the identity on~$K$. Since~$y$ was arbitrary, we have that~$\phi$ is the identity and we have the claim.
\end{proof}

We give a more precise description of the~$G$-tree of~$\Sew_\cx(w)$ in Proposition \ref{prop:gtree_zipper}.
\begin{proposition}\label{prop:gtree_zipper}
    Let~$\cx\in\MTub(P_n)$, with~$G$-tree~$G_\cx$. Let~$w\in W_\cx$. The~$G$-tree of~$\Sew_\cx(w)$ is the transitive closure of~$G_\cx \cup (w(1) < w(2) < \cdots < w(r+l))$.  In other words, the~$G$-tree of~$\Sew_\cx(w)$ has all the relations of $G_\cx$, and the left and right zippers are ``zipped up'' in the order given by~$w$, introducing additional relations.
\end{proposition}

\begin{proof}
    Let~$T$ be the transitive closure of~$G_\cx \cup (w(1) < w(2) < \cdots < w(r+l))$, and for notation let $\cj = \Sew_\cx(w)$. Since the left and right zippers of~$\cx$ are disjoint chains and~$w$ respects their orders, the transitive closure of these relations is in fact a partially ordered set. Specifically, we will prove that the set of order relations for~$T$ is contained within the order relations in~$G_\cj$, and that incomparable elements in~$T$ are also incomparable in~$G_\cj$. \par 
    First we will show that the order relations for~$G_\cx$ are contained within those of the~$G_\cj$. If~$i <_\cx j$ (so~$i \in \downX(j)$), then~$i \in \sew_\cx(\downX(j),w) = \cj_\downarrow(j)$, and thus~$i <_{\cj} j$. So the order relations of~$G_\cx$ are contained within the order relations of the~$G_\cj$. By Lemma \ref{lem:cutinv_downsets} and the definition of~$\sew_\cx$, we have that 
    \[
    \cj_\downarrow(w(j)) =  \bigcup_{i=1}^{j}\downX
    (w(i)) \subset  \bigcup_{i=1}^{k}\downX(w(i)) = \cj_\downarrow(w(k))
    \]
    whenever~$j < k$, so the order relations~$w(1) < \cdots  < w(r+l)$ are also contained within the order relations of~$G_\cj$. Thus, the order relations of the transitive closure~$T$ are also contained within the order relations of the~$G_\cj$. \par 
    Now we will prove that the set of incomparable elements in~$T$ is the same as the set of incomparable elements in~$G_\cj$.\par 
    Let~$x \not\sim_{T} y$. In particular,~$x$ and~$y$ cannot both be in the zipper. If neither~$x$ nor~$y$ are in the zipper of~$\cx$, then the intersection 
    \[
    \left(\cj_\downarrow(x) = \sew_\cx(\downX(x),w) = \downX(x)\right) \cap \left(\downX(y) = \sew_\cx(\downX(y),w) = \cj_\downarrow(y)\right)
    \]
    is empty, and so~$x \not\sim_{\cj} y$. \par 
    If, without loss of generality,~$x$ is in the zipper of~$\cx$ and~$y$ is not, then we have two cases. We will prove that~$x \not<_{\cj} y$ and~$y \not<_{\cj} x$, both by contradiction.\par 
    If~$x <_{\cj}y$ and~$y$ is not in the zipper, then~$y = m$. Then~$\downX(y) = [n]$ contains~$\downX(x)$ and so~$x <_{T} y$, a contradiction. \par 
    On the other hand, if~$y <_{\cj} x$ then~$y \in {\cj}_\downarrow(x) = \ds \bigcup_{i=1}^{w^{-1}(x)}\downX(w(i))$. In particular, for some~${i < w^{-1}(x)}$, we have~$y \in \downX(w(i))$. Namely,~$y <_\cx w(i)$ and~$w(i)$ appears (weakly) before~$x$ in~$w$, so~$y <_{T} w(i) \leq_{T} x$, which is also a contradiction. 
\end{proof}

We can see Proposition~\ref{prop:gtree_zipper} in action in Figure~\ref{fig: cut and sew map example}.  Note that the $G$-tree of $\cj$ has all the relations of $\Cut(\cj)$, and the left and right zippers ($(1,3)$ and $(9,7)$ resp.) are zipped up in the order given by $w=9137$.

\begin{corollary}
    If~$\cx \in \MTub(P_n)$ then~$\Sew_\cx$ is injective.
\end{corollary}
\begin{proof}
    Two elements~$w,w' \in W_\cx$ are equal if and only if they represent the same total order on~$w([l+r])$. By Proposition \ref{prop:gtree_zipper}, the~$G$-tree of~$\Sew_\cx(w)$ is equal to the~$G$-tree of~$\Sew_\cx(w')$ if and only if that total order is identical.
\end{proof}
\begin{proposition}\label{prop:cutinv}
    Let~$\cx \in \MTub(P_n)$. Let~$\Cut^{-1}(\cx)=\{\cj\in\MTub(C_n) \ | \ \Cut(\cj)=\cx\}$. Then \[\Cut^{-1}(\cx) = \{\Sew_\cx(w) \mid w \in W_\cx\}.\]
\end{proposition}
\begin{proof}
    Let~$\cj \in \MTub(C_n)$ and~$\cx = \Cut(\cj)$. First the forwards direction of containment. We will show that~$\cj = \Sew_{\cx}(w)$ for some~$w \in W_{\cx}$. In particular, we will show for a fixed~$w$ that
    \[\sew_\cx(\cut_\cj(\downJ(x)),w) = \downJ(x).\] 
    Let~$\cj$ have maximal element~$m$. Since~$1$ and~$n$ are adjacent in~$C_n$ they must have a relation in $G_\cj$. Without loss of generality assume that~$1 <_\cj n$. Let~$w$ be the saturated chain from~$1$ to~$m$ in$G_\cj$. \par 
    If~$x = m$ then by definition~$\sew_\cx(\cut_\cj(\downJ(x)),w) = [n] = \downJ(x)$ for all~$w \in W_{\cx}$.\par 
    If~$x < m$ then 
    \begin{align*}
        \sew_\cx(\cut_\cj(\downJ(x)),w) &= \sew_\cx(\downJ(x)\cap [1,m-1],w) \\
        &= \begin{cases}
             \downJ(x)\cap [1,m-1] &\text{if } x \text{ is not in the zipper of~$\cx$, and}\\[5pt]
            \ds \bigcup_{i=1}^{w^{-1}(x)}\downX(w(i)) &\text{if } x \text{ is in the zipper.}
            \end{cases} \\       
    \end{align*}
    Now if~$x$ is not in the zipper then~$\cx_\downarrow(x)$ contains neither~$1$ nor~$n$. In particular,~$\downJ(x)\cap [1,m-1] = \downJ(x)$. \par 
    If~$x$ is in the (left, since~$x<m$) zipper and~$y \in \downJ(x)\cap [m+1,n]$, let~$k$ be such that~$\cx_\downarrow(k) = \downJ(x)\cap [m+1,n]$. So~$k$ is also in the zipper, and~$w^{-1}(k) < w^{-1}(x)$. It follows that~$y \in \bigcup_{i=1}^{w^{-1}(x)}\cx_\downarrow(w(i))$.\par 
    The argument for~$x > m$ is the same, swapping~$1$ with~$n$, left with right, and~$[1,m-1]$ with~$[m+1,n]$. Thus, we have that $ \sew_\cx(\cut_\cj(\downJ(x)),w) = \downJ(x)$, and so~$\cj = \Sew_{\cx}(w)$. \par 
    For the backwards direction of containment, let~$\cx \in \MTub(P_n)$ and~$w \in W_\cx$ be arbitrary. For notation let~$\cj = \Sew_\cx(w)$. Then have that 
    \begin{align*}
        \cut_{\cj}(\sew_\cx(\downX(x),w)) 
        &= \begin{cases}
             \cut_{\cj}(\downX(x)) &\text{if } x \text{ is not in the zipper of~$\cx$, and}\\[5pt]
            \ds \cut_\cj\left( \bigcup_{i=1}^{w^{-1}(x)}\downX(w(i))\right) &\text{if } x \text{ is in the zipper.}
            \end{cases} \\ 
        &= \begin{cases}
             \downX(x) &\text{if } x \text{ is not in the zipper of~$\cx$,}\\[5pt]
            \ds [1,m-1] \cap \left( \bigcup_{i=1}^{w^{-1}(x)}\downX(w(i))\right) &\text{if } x \text{ is in the left zipper of~$\cx$, and }\\
            \ds [m+1,n] \cap \left( \bigcup_{i=1}^{w^{-1}(x)}\downX(w(i))\right) &\text{if } x \text{ is in the right zipper of~$\cx$. }
            \end{cases} 
    \end{align*}
    However~$\downX(w(i)) \subseteq [1,m-1]$ if and only if~$w(i)$ is in the left zipper and~$\downX(w(i)) \subseteq [m+1,n]$ if and only if~$w(i)$ is in the right zipper. So each case above is equal to~$\downX(x)$. So we have that~$\Cut(\Sew_\cx(w)) = \cx$, and we get both directions of containment.    
\end{proof}
\begin{example}\label{ex:Cut_map_inverse}
Figure~\ref{fig: cut and sew map example} depicts the~$G$-trees of~$\cj=\Sew_{\Cut(\cj)}(9137)$ and~$\Cut(\cj)$.
Note that~$a = (1,3)$ is the left zipper and~$b = (9,7)$ is the right zipper of~$\Cut(\cj)$. There is one~$G$-tree in the fiber~$\Cut^{-1}(\Cut(\cj))$ for each of the following permutations: 
\[1397,\;\;1937,\;\;1973,\;\;9137,\;\;9173,\;\;9713.\]
\end{example}
\begin{corollary}
         The map~$\Cut:\MTub(C_n)\to\MTub(P_n)$ is surjective.
\end{corollary}
\begin{proof}
    While it is possible for~$\cx \in \MTub(P_n)$ to not have a left or right zipper (if $m=1$ or $m=n$), it cannot lack both. In particular,~$W_\cx$ is never empty.
\end{proof}

\begin{remark}
     Let~$l,r$ be the sizes of the left and right zippers respectively of~$\cx \in \MTub(P_n)$. There are 
  ~$\frac{(l+r)!}{(l)!(r)!}$ elements in the fiber of~$\cx$. This is the size of~$W_\cx$.
\end{remark}

\begin{proposition}\label{prop: fiber is an interval}
Let~$\cx\in\MTub(P_n)$.  Then~$\Cut^{-1}(\cx)\in\MTub(C_n)$ is isomorphic to~$W_\cx$ as an interval in the weak order.
\end{proposition}
\begin{proof}
    Since~$W_\cx$ is the set of in-order shufflings, it is isomorphic to an interval in the weak order. The weak order on~$W_\cx$ naturally induces an order on~$\Cut^{-1}(\cx)$. We will prove the cover relations for~$W_\cx$ are cover relations in~$\Cut^{-1}(\cx)$, and vice versa. \par 
    Denote by $s_i$ the transposition of $i$ and $i+1$, with $ws_i$ the multiplication of $w$ and $s_i$ as permutations.  A cover relation in~$W_\cx$ is~$w \lessdot ws_i$ where~$w(i)$ is in the left zipper and~$w(i+1)$ is in the right zipper. Any other case for~$w(j)$ and~$w(j+1)$ will result in~$ws_j \notin W_\cx$.\par 
    Assume~$w(i)$ is covered by~$w(i+1)$ in the~$G$-tree of~$\Sew_\cx(w)$. The tree move swapping~$w(i)$ and~$w(i+1)$ is, visually, 
     
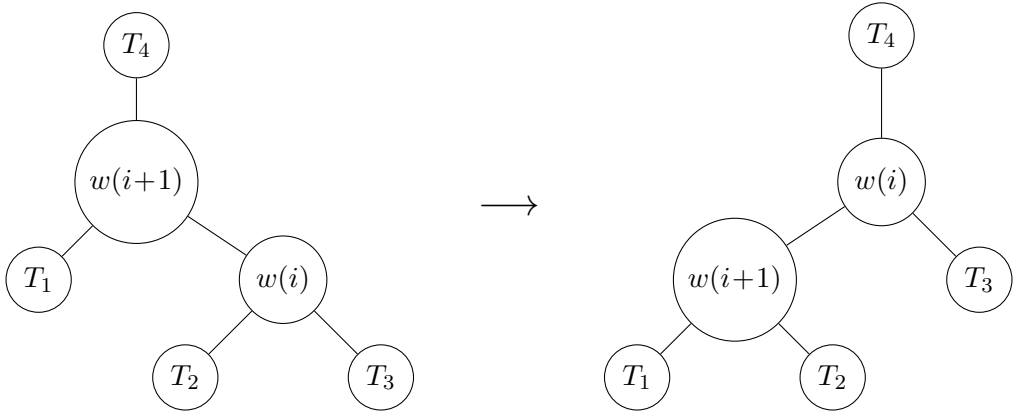
\begin{center}
 \begin{tikzpicture}[scale=1.3]
      \begin{scope}[xshift=3in]
         \draw (0,0.5) node[draw, circle] (T4)     {$T_4$};
         \draw (0,-1) node[draw, circle] (y)     {$w(i)$};
         \draw (-1.5,-2) node[draw, circle] (x)  {$w(i\!+\!1)$};
         \draw (-2.5,-3) node[draw, circle] (T1) {$T_1$};
         \draw (-0.5,-3) node[draw, circle] (T2) {$T_2$};
         \draw (1,-2) node[draw, circle] (T3) {$T_3$};

         \draw (T4)--(y)--(x)--(T1);
         \draw (x)--(T2);
         \draw (y)--(T3);
     \end{scope} 
     \begin{scope}[xshift = 1.5in, yshift = -0.5in]
         \draw (0,0) node {\Large${ \longrightarrow}$};
     \end{scope}
     \begin{scope}[xshift = 0in]
        \draw (0,0.4) node[draw, circle] (T4)     {$T_4$};
         \draw (1.5,-2) node[draw, circle] (y)     {$w(i)$};
         \draw (0,-1) node[draw, circle] (x)  {$w(i\!+\!1)$};
         \draw (-1,-2) node[draw, circle] (T1) {$T_1$};
         \draw (0.5,-3) node[draw, circle] (T2) {$T_2$};
         \draw (2.5,-3) node[draw, circle] (T3) {$T_3$};

         \draw (T4)--(x)--(y)--(T3);
         \draw (x)--(T1);
         \draw (y)--(T2);
     \end{scope}
 \end{tikzpicture}
 \end{center}
 where~$T_2 = \Sew_\cx(w)_\downarrow(w(i-1))$. Since this tree move only changes order relations for the zipper, the image under $\Cut$ of this new tubing is also~$\cx$. In particular, the~$G$-tree that corresponds to this cover relation is~$\Sew_\cx(ws_i)$. So the cover relations in~$W_\cx$ are contained within the cover relations for~$\Cut^{-1}(\cx)$. \par 
 On the other hand, any cover relation~$\Sew_\cx(w) \lessdot \Sew_\cx(w')$ must be a tree move of the above form. Any other tree move will result in a~$G$-tree that does not correspond to anything in~$\Cut^{-1}(\cx)$. We have already shown that these tree moves correspond to~$w \lessdot ws_i$ in~$W_\cx$ so the cover relations in~$\Cut^{-1}(\cx)$ are contained within the cover relations for~$W_\cx$. Since the cover relations are identical, these orders are equal.
\end{proof}
We end with an observation that the $\Cut$ map respects the map from Corollary \ref{cor:involution}, by which both $\MTub(C_n)$ and $\MTub(P_n)$ are self-dual.
\begin{proposition}\label{prop:cut_commuting}
    The map $\Cut \colon \MTub(C_n) \to \MTub(P_n)$ commutes with the involution induced from $w_0$ in Corollary \ref{cor:involution}. Formally, if $\cj \in \MTub(C_n)$ then 
    \[ \Cut(w_0(\cj)) = w_0(\Cut(\cj))\, . \]
\end{proposition}
\begin{proof}
    This follows from the definition of $\cut_\cj(\downJ(x))$.
\end{proof}

\section{Global Relations in \texorpdfstring{$\MTub(C_n)$}{MTub Cycle}}\label{sec:Global Relations}
The order on~$\MTub(G)$ is defined by its cover relations, so given two arbitrary maximal tubings of a graph~$G$ it is not clear at all whether they are related in~$\MTub(G)$. For~$\MTub(K_n)$ this issue is resolved by the characterization of the weak order as inversion-containment. For~$\MTub(P_n)$ this can be checked in the same fashion with Proposition \ref{prop:path_order}. 

This section will prove the following characterization of global relations in~$\MTub(C_n)$.
\begin{theorem}\label{thm:inversion_order}
    Let~$\cj$ and~$\ck$ be tubings of the cycle graph. Then~$\cj \leq \ck$ if and only if~$\inv(\cj) \subseteq \inv(\ck) \cup \inc(\ck)$.
\end{theorem}
Before we give the proof, which will consist of the remainder of this section, some examples.
\begin{example}\label{ex:Global Relations}
    Consider the following three cyclic binary trees., corresponding to the maximal tubings ${\cj,\ck,\cl \in \MTub(C_8)}$.
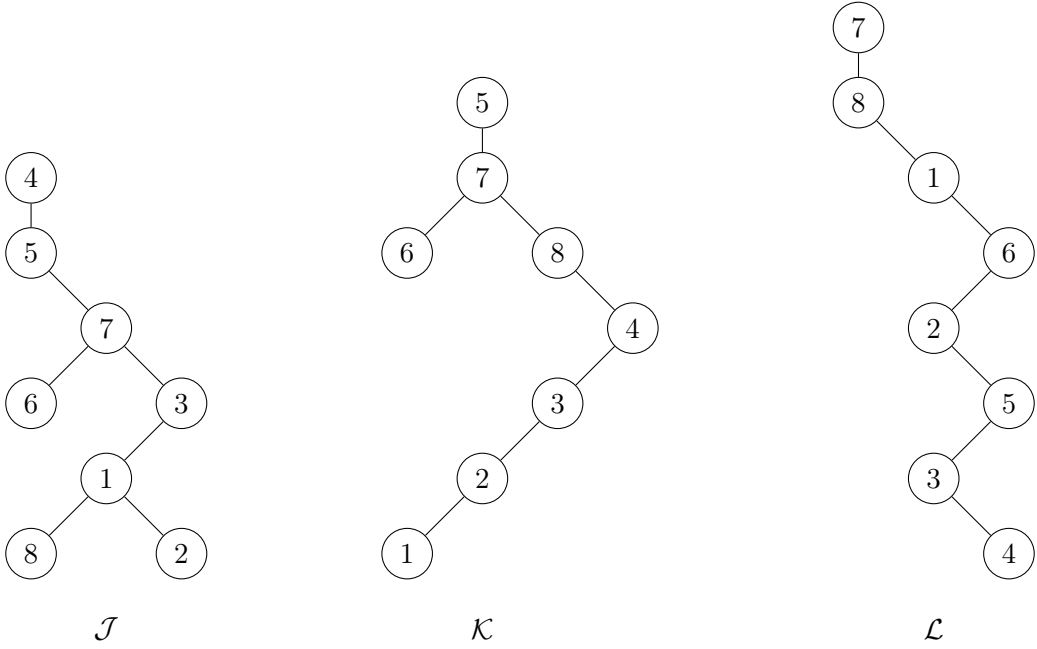
    \begin{center}
        
    \begin{tikzpicture}
  \node[draw,circle] (1) at (0,1) {1};
  \node[draw,circle] (2) at (1,0) {2};
  \node[draw,circle] (3) at (1,2) {3};
  \node[draw,circle] (4) at (-1,5) {4};
  \node[draw,circle] (5) at (-1,4) {5};
  \node[draw,circle] (6) at (-1,2) {6};
  \node[draw,circle] (7) at (0,3) {7};
  \node[draw,circle] (8) at (-1,0) {8};
  \node (name) at (0,-1) {$\cj$};

    \draw (4)--(5)--(7)--(3)--(1)--(2);
    \draw (1)--(8);
    \draw (7)--(6);
  
  \begin{scope}[shift = {(5,0)}]
  \node[draw,circle] (1) at (-1,0) {1};
  \node[draw,circle] (2) at (0,1) {2};
  \node[draw,circle] (3) at (1,2) {3};
  \node[draw,circle] (4) at (2,3) {4};
  \node[draw,circle] (5) at (0,6) {5};
  \node[draw,circle] (6) at (-1,4) {6};
  \node[draw,circle] (7) at (0,5) {7};
  \node[draw,circle] (8) at (1,4) {8};
  \node (name) at (0,-1) {$\ck$};
  
  \draw (5)--(7)--(8)--(4)--(3)--(2)--(1);
  \draw (7)--(6);
  \end{scope}
   \begin{scope}[shift = {(10,-1)}]
  \node[draw,circle] (1) at (1,6) {1};
  \node[draw,circle] (2) at (1,4) {2};
  \node[draw,circle] (3) at (1,2) {3};
  \node[draw,circle] (4) at (2,1) {4};
  \node[draw,circle] (5) at (2,3) {5};
  \node[draw,circle] (6) at (2,5) {6};
  \node[draw,circle] (7) at (0,8) {7};
  \node[draw,circle] (8) at (0,7) {8};
  \node (name) at (1,0) {$\cl$};
  \draw (7)--(8)--(1)--(6)--(2)--(5)--(3)--(4);
  \end{scope}
    \end{tikzpicture}
        \end{center}
    We can check that 
    \[
        \inv(\ck) = \{(5,6),(5,7),(5,8),(7,8)\} \subseteq \inv(\cj) \cup \inc(\cj),
    \]
    and so $\ck < \cj$. On the other hand, $(5,6)$ is an inversion in $\cj$ and $\ck$ but is a co-inversion in $\cl$, while $(3,4)$ is an inversion in $\cl$ but co-inversions in $\cj$ and $\ck$. So $\cl$ is not comparable to either $\cj$ or $\ck$.
    
\end{example}
\begin{example}
    The inclusion of $\inc(\ck)$ in Theorem \ref{thm:inversion_order} is necessary. For example, the following is a cover relation in $\MTub(C_4)$:
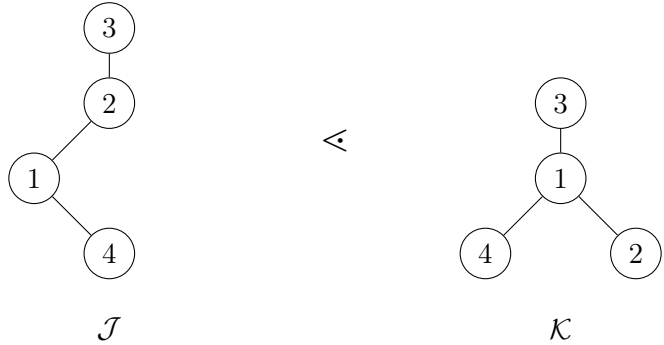
    \begin{center}
    \begin{tikzpicture}
    \begin{scope}
        \draw (0,0) node[draw,circle] (v1) {4};
        \draw (-1,1) node[draw,circle] (v2) {1};
        \draw (0,2) node[draw,circle] (v3) {2};
        \draw (0,3) node[draw,circle] (v4) {3};
        \draw (v1)--(v2)--(v3)--(v4);

        \draw (0,-1) node (name) {$\cj$};
    \end{scope}
    \begin{scope}[shift = {(3,1.5)}]
        \draw (0,0) node (l) {\Large $\lessdot$};
    \end{scope}
    \begin{scope}[shift = {(6,-1)}]
        \draw (-1,1) node[draw,circle] (v1) {4};
        \draw (0,2) node[draw,circle] (v2) {1};
        \draw (1,1) node[draw,circle] (v3) {2};
        \draw (0,3) node[draw,circle] (v4) {3};
        \draw (v3)--(v2)--(v4);
        \draw (v2)--(v1);

        \draw (0,0) node (name) {$\ck$};
    \end{scope}
        
    \end{tikzpicture}
    \end{center}
    However, $\inv(\cj) \not\subseteq \inv(\ck)$ because of $(2,4) \in \inv(\cj)$.
\end{example}

The forwards and backwards direction of Theorem \ref{thm:inversion_order} are Propositions \ref{prop: forward direction} and \ref{prop: Backwards}.\par 
First, we will show equality.

\begin{proposition}\label{prop:inv_equality}
    If~$\cj$ and~$\ck$ are tubings of the cycle graph then~$\inv(\cj) = \inv(\ck)$ if and only if~$\cj=\ck$.
\end{proposition}
\begin{proof}
    Assume for induction that this is true for all~$C_m$ for~$m < n$. The base case is easily verified for~$C_3$. \par 
    The backwards direction is trivial.\par 
    Say that~$\inv(\cj) = \inv(\ck)$. For any order relation~$a<_{\cj} b$, either~$a<b$ and so~$(a,b) \notin \inv(\cj)$ or~$a>b$ and~$(a,b) \in \inv(\cj)$. In particular, the maximal element in $G_\cj$ is precisely the element~$m$ such that~$(r,m) \in \inv(\cj)$ whenever~$r>m$ and~$(r,m) \notin \inv(\cj)$ whenever~$(r<m)$. So the maximal element is uniquely determined by~$\inv(\cj)$, so the maximal element of~$\cj$ and the maximal element of~$\ck$ are equal. \par 
    Next, since~$1$ is the smallest integer in~$C_n$, it is easily determined which elements are below~$1$ in the~$G$-trees of~$\cj$ and~$\ck$ from~$\inv(\cj)$ and~$\inv(\ck)$. In particular, 
    \[ \{i \mid i <_{\cj} 1\} = \{i \mid (1,i) \in \inv(\cj)\},\]
    and the same for~$\ck$. So the lower order ideals generated by~$1$ are identical in~$\cj$ and~$\ck$. Furthermore, the lower order ideal generated by~$1$ has two distinct sub-trees, the left and the right. The left sub-tree contains strictly those elements less than~$1$ in the circular order, in particular the left sub-tree has elements in~$\{m+1,...,n\}$. Similarly, the right sub-tree has elements in~$\{2,...,m-1\}$. Since the induced subgraphs on the vertices $[2,m-1]$ and $[m+1, n]$ are both path graphs, these sub-trees are determined by their inversion sets. It follows that the lower order ideals of~$1$ are uniquely determined by~$\inv(\cj)$, and so are equal in~$\cj$ and~$\ck$. \par  
    The subposet of $G_\cj$ induced by the set $[n]\setminus \downJ(1)$ is a cyclic binary tree under the order preserving map from $[n] \setminus \downJ(1)$ to $\left[ n-\abs{\downJ(1)}\right]$.
    Moreover, this subposet is itself a~$G$-tree for~$C_{n-\abs{\downJ(1)}}$. By the induction assumption, this tree is uniquely determined by its inversion set, which is simply~$\inv(\cj)$ with all of the pairs containing elements in~$\downJ(1)$ removed from it. \par 
    So the lower order ideal for~$1$ is uniquely determined, as is the tree with that lower order ideal removed. Since the location of~$1$ is uniquely determined by that tree, the entire structure is uniquely determined, and~$\cj=\ck$.
\end{proof}
Next, we can show the forwards direction.
\begin{proposition}\label{prop: forward direction}
    
If~$\cj < \ck$ in~$\MTub(C_n)$, then~$\inv(\cj)\subseteq \inv(\ck)\cup\inc(\ck)$.
\end{proposition} 
\begin{proof}

    Consider the set of all possible inversions that contradict the claim:
    \[
    \bigcup_{\substack{\cj,\ck\in\MTub(C_n) \\\mathcal{J}<\mathcal{K}}}\left\{ (i,j) \mid (i,j) \in \inv(\cj) \setminus (\inv(\ck)\cup\inc(\ck))\right\}.
    \]
      Fix an~$(i,k)$ in this set such that~$k-i$ is maximal (there may be many tuples with such a property, any will suffice). We will show that this element cannot exist, and so this set is empty. \par 
    Let~$\cj$ and~$\ck$ be maximal tubings such that~$(i,k) \in \inv(\cj)~$ and~$(i,k) \notin \inv(\ck)\cup\inc(\ck)$. Now~$(i,k) \in \inv(\cj)$ if and only if~$k <_{\cj} i$, where~$<_{\cj}$ is the order relation induced from $G_\cj$. On the other hand,~$(i,k) \notin \inv(\ck)\cup\inc(\ck)$ if and only if~$i <_{\ck}k$ in the order relation induced from~$G_\ck$. \par 
    Since~$\cj < \ck$, there is a sequence of cover relations~$\cj \lessdot \cj_1 \lessdot \cdots \lessdot \cj_{\ell} \lessdot \ck$  in~$\MTub(C_n)$. Since~$i \not<_{\cj} k$ and~$i <_{\ck}k$, there must exist a cover relation~$\cj_r < \cj_{r+1}$ in~$\MTub(C_n)$ such that~$i \not<_{\cj_r} k$ and~$i <_{\cj_{r+1}} k$ in the respective~$G$-trees. In particular, it suffices to assume that~$\cj$ is the \emph{last} tubing in this chain such that~$(i,k) \in  \inv(\cj)$, and that~$\ck$ is the first tubing in the chain thereafter such that~$(i,k) \notin \inv(\ck)\cup\inc(\ck)$.  \par 
    We will examine the cover relations~$\cj \lessdot \cj_1$ and~$\cj_{\ell} \lessdot \ck$ in order to find our contradiction. Note that neither of these cover relations may come from swapping the top element, since such a tree move changes only one order relation. There are two options for tree moves for a cover relation. Say that the following is the cover relation from~$\cj$ to~$\cj_1$, where~$k \in T_3$:
    
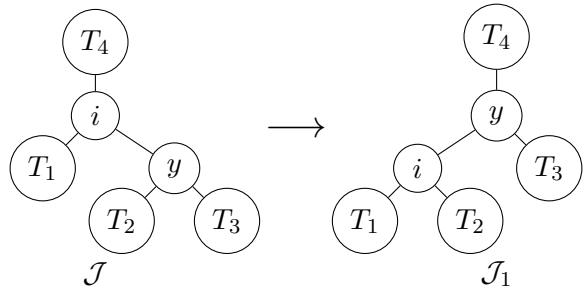
\begin{center}
 \begin{tikzpicture}[scale=0.7]
      \begin{scope}[xshift=3in]
         \draw (0,0.5) node[draw, circle] (T4)     {$T_4$};
         \draw (0,-1) node[draw, circle] (y)     {$y$};
         \draw (-1.5,-2) node[draw, circle] (x)  {$i$};
         \draw (-2.5,-3) node[draw, circle] (T1) {$T_1$};
         \draw (-0.5,-3) node[draw, circle] (T2) {$T_2$};
         \draw (1,-2) node[draw, circle] (T3) {$T_3$};

         \draw (T4)--(y)--(x)--(T1);
         \draw (x)--(T2);
         \draw (y)--(T3);
         \draw (0,-4) node (NM) {$\cj_1$};
     \end{scope} 
     \begin{scope}[xshift = 1.5in, yshift = -0.5in]
         \draw (0,0) node {\Large${ \longrightarrow}$};
     \end{scope}
     \begin{scope}[xshift = 0in]
        \draw (0,0.4) node[draw, circle] (T4)     {$T_4$};
         \draw (1.5,-2) node[draw, circle] (y)     {$y$};
         \draw (0,-1) node[draw, circle] (x)  {$i$};
         \draw (-1,-2) node[draw, circle] (T1) {$T_1$};
         \draw (0.5,-3) node[draw, circle] (T2) {$T_2$};
         \draw (2.5,-3) node[draw, circle] (T3) {$T_3$};

         \draw (T4)--(x)--(y)--(T3);
         \draw (x)--(T1);
         \draw (y)--(T2);
         \draw (0,-4) node (NM) {$\cj$};
     \end{scope}
 \end{tikzpicture}
 \end{center}
 Since~$i<_m y <_m k$ in the cyclic order,  the maximal element~$m$ is either greater than~$k$ or less than~$i$. In particular,~$i < y < k$ in the standard order on integers as well. For $\cj\lessdot\cj_1$ to be a cover relation it must be that~$i > y$, a contradiction. So we have that the following is the cover relation from~$\cj$ to~$\cj_1$,  where~$k \in T_1$:
    
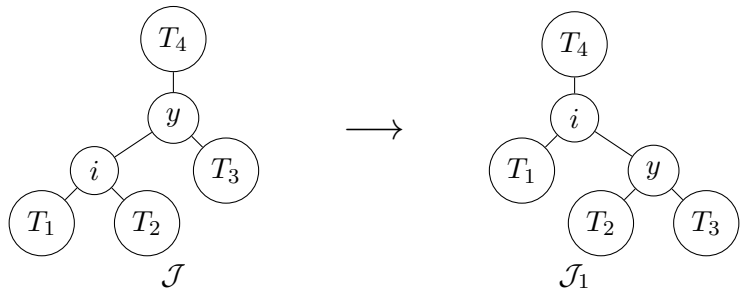
\begin{center}
 \begin{tikzpicture}[scale=0.7]
      \begin{scope}[xshift=0in]
         \draw (0,0.5) node[draw, circle] (T4)     {$T_4$};
         \draw (0,-1) node[draw, circle] (y)     {$y$};
         \draw (-1.5,-2) node[draw, circle] (x)  {$i$};
         \draw (-2.5,-3) node[draw, circle] (T1) {$T_1$};
         \draw (-0.5,-3) node[draw, circle] (T2) {$T_2$};
         \draw (1,-2) node[draw, circle] (T3) {$T_3$};

         \draw (T4)--(y)--(x)--(T1);
         \draw (x)--(T2);
         \draw (y)--(T3);
         \draw (0,-4) node (NM) {$\cj$};
     \end{scope} 
     \begin{scope}[xshift = 1.5in, yshift = -0.5in]
         \draw (0,0) node {\Large${ \longrightarrow}$};
     \end{scope}
     \begin{scope}[xshift = 3in]
        \draw (0,0.4) node[draw, circle] (T4)     {$T_4$};
         \draw (1.5,-2) node[draw, circle] (y)     {$y$};
         \draw (0,-1) node[draw, circle] (x)  {$i$};
         \draw (-1,-2) node[draw, circle] (T1) {$T_1$};
         \draw (0.5,-3) node[draw, circle] (T2) {$T_2$};
         \draw (2.5,-3) node[draw, circle] (T3) {$T_3$};

         \draw (T4)--(x)--(y)--(T3);
         \draw (x)--(T1);
         \draw (y)--(T2);
         \draw (0,-4) node (NM) {$\cj_1$};
     \end{scope}
 \end{tikzpicture}
 \end{center}
 Observe that $k <_m i <_m y$ in circular order and~$i<y$ in the standard order on integers. In particular, the maximal element for both~$\cj$ and~$\cj_1$ is some~$m$ such that~$i < m < k$. \par 
 Now we consider the cover relation~$\cj_\ell \lessdot \ck$. Again there are two options for tree moves, and only one is valid. Say that the following is the cover relation from~$\cj_\ell$ to~$\ck$, where~$i \in T_1$:
     
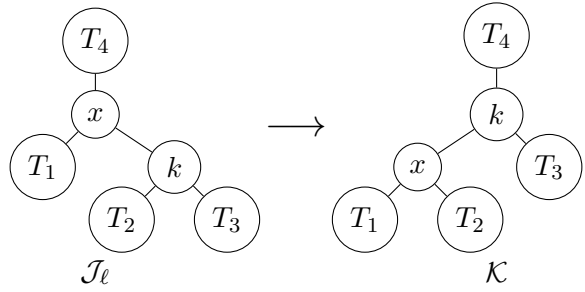
\begin{center}
 \begin{tikzpicture}[scale=0.7]
      \begin{scope}[xshift=3in]
         \draw (0,0.5) node[draw, circle] (T4)     {$T_4$};
         \draw (0,-1) node[draw, circle] (y)     {$k$};
         \draw (-1.5,-2) node[draw, circle] (x)  {$x$};
         \draw (-2.5,-3) node[draw, circle] (T1) {$T_1$};
         \draw (-0.5,-3) node[draw, circle] (T2) {$T_2$};
         \draw (1,-2) node[draw, circle] (T3) {$T_3$};

         \draw (T4)--(y)--(x)--(T1);
         \draw (x)--(T2);
         \draw (y)--(T3);
         \draw (0,-4) node (NM) {$\ck$};
     \end{scope} 
     \begin{scope}[xshift = 1.5in, yshift = -0.5in]
         \draw (0,0) node {\Large${ \longrightarrow}$};
     \end{scope}
     \begin{scope}[xshift = 0in]
        \draw (0,0.4) node[draw, circle] (T4)     {$T_4$};
         \draw (1.5,-2) node[draw, circle] (y)     {$k$};
         \draw (0,-1) node[draw, circle] (x)  {$x$};
         \draw (-1,-2) node[draw, circle] (T1) {$T_1$};
         \draw (0.5,-3) node[draw, circle] (T2) {$T_2$};
         \draw (2.5,-3) node[draw, circle] (T3) {$T_3$};

         \draw (T4)--(x)--(y)--(T3);
         \draw (x)--(T1);
         \draw (y)--(T2);
         \draw (0,-4) node (NM) {$\cj_\ell$};
     \end{scope}
 \end{tikzpicture}
 \end{center}
 Observe that~$i <_{m'} x <_{m'} k$ in circular order. Since $i<k$ in the integers, this implies that $i<x<k$ is true in the standard order on integers as well.
Since $\cj_\ell \lessdot \ck$, we have $k<x$ in the standard order, a contradiction.
So we have that the following is the cover relation from~$\cj_\ell$ to~$\ck$, where~$i \in T_3$:
    
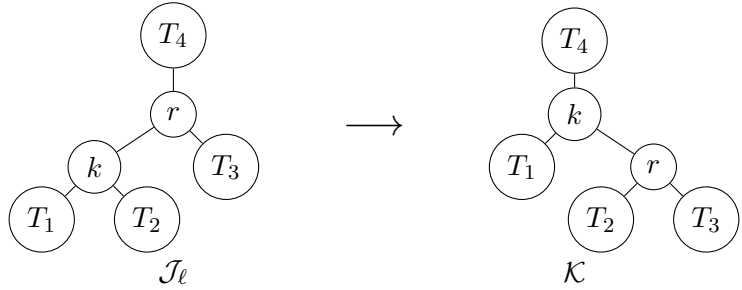
\begin{center}
 \begin{tikzpicture}[scale=0.7]
      \begin{scope}[xshift=0in]
         \draw (0,0.5) node[draw, circle] (T4)     {$T_4$};
         \draw (0,-1) node[draw, circle] (y)     {$r$};
         \draw (-1.5,-2) node[draw, circle] (x)  {$k$};
         \draw (-2.5,-3) node[draw, circle] (T1) {$T_1$};
         \draw (-0.5,-3) node[draw, circle] (T2) {$T_2$};
         \draw (1,-2) node[draw, circle] (T3) {$T_3$};

         \draw (T4)--(y)--(x)--(T1);
         \draw (x)--(T2);
         \draw (y)--(T3);
         \draw (0,-4) node (NM) {$\cj_\ell$};
     \end{scope} 
     \begin{scope}[xshift = 1.5in, yshift = -0.5in]
         \draw (0,0) node {\Large${ \longrightarrow}$};
     \end{scope}
     \begin{scope}[xshift = 3in]
        \draw (0,0.4) node[draw, circle] (T4)     {$T_4$};
         \draw (1.5,-2) node[draw, circle] (y)     {$r$};
         \draw (0,-1) node[draw, circle] (x)  {$k$};
         \draw (-1,-2) node[draw, circle] (T1) {$T_1$};
         \draw (0.5,-3) node[draw, circle] (T2) {$T_2$};
         \draw (2.5,-3) node[draw, circle] (T3) {$T_3$};

         \draw (T4)--(x)--(y)--(T3);
         \draw (x)--(T1);
         \draw (y)--(T2);
         \draw (0,-4) node (NM) {$\ck$};
     \end{scope}
 \end{tikzpicture}
 \end{center}
Observe that $k <_{m'} x <_{m'} i$ in circular order and~$k < x$ in the standard order. Note here that~$x >_{\ck}i$. On the other hand, there is a tube~$Y$ (the down-set of~$y$ in~$\cj$) in~$\cj$ that contains both~$i$ and~$k$. Since the maximal element~$m$ for~$\cj$ is between~$i$ and~$k$, this tube must contain all elements from~$k$ to~$i$ in that circular order for~$\cj$. That tube~$Y$ necessarily contains all elements greater than~$k$ and less than~$i$. In particular,~$x \in Y$. However, this is the same tube~$Y$ in~$\cj_1$ (the down-set of~$i$), so~$x <_{\cj_1} i$. But~$\cj_1 < \ck$ and so we have that~$(i,x)$ is in the set of problematic inversions, and~$x-i > k-i$, a contradiction. Thus no such pair~$(i,k)$ can exist, and so the set is empty, and we have our claim.
\end{proof}
The backwards direction of Theorem~\ref{thm:inversion_order} will prove more difficult, and will proceed in two main parts: First, we prove that~$\inv(\cj) \subseteq \inv(\ck) \cup \inc(\ck)$ implies the existence of a descent in~$\dec(\ck)$ that is \emph{not} an inversion in~$\inv(\cj)$. This will heavily leverage the $\Cut$ map. Second, we prove that the existence of such a descent implies the existence of a tubing~$\ck' \lessdot \ck$ such that~$\inv(\cj) \subseteq \inv(\ck') \cup \inc(\ck')$, and the full result will follow. We begin by deducing a relation on the images of~$\cj$ and~$\ck$ under the $\Cut$ map.
\begin{proposition} \label{prop: cut ineq}
    If~$\inv(\cj) \subseteq \inv(\ck) \cup \inc(\ck)$ then~$\Cut(\cj) \leq \Cut(\ck)$.
\end{proposition}
\begin{proof}

    Let~$[m] \coloneqq \{1,...,m\}$ and~$[m,n] \coloneqq \{m,m+1,...,n\}$.
    If~$m$ is the maximum of~$\cj$, it must be the case that~${\inv(\Cut(\cj)) = \inv(\cj) \cap \left( [m]^2 \cup [m,n]^2\right)}$. 
    Since the $\Cut$ map simply creates incomparable elements (never switching the order relations),~$\inv(\ck) \subseteq \inv(\Cut(\ck)) \cup \inc(\Cut(\ck))$. Furthermore,~$\inc(\ck) \subseteq \inc(\Cut(\ck))$. It follows that 
    \begin{align*}
        \inv(\Cut(\cj)) &= \inv(\cj) \cap \left( [m]^2 \cup [m,n]^2\right) \\
        &\subseteq \inv(\cj) \\
        &\subseteq \inv(\ck) \cup \inc(\ck) \\
        &\subseteq \left(\inv(\Cut(\ck)) \cup \inc(\Cut(\ck))\right) \cup \inc(\Cut(\ck)) \\
        &= \inv(\Cut(\ck)) \cup \inc(\Cut(\ck))\, .
    \end{align*}
    By Proposition \ref{prop:path_order}, the above containment implies that~$\Cut(\cj) \leq \Cut(\ck)$.
\end{proof}
\begin{corollary}\label{cor:cutmap_orderpreserving}
    If $\cj \leq \ck$ in $\MTub(C_n)$ then $\Cut(\cj) \leq \Cut(\ck)$ in $\MTub(P_n)$. In other words, the $\Cut$ map is order preserving.
\end{corollary}
\begin{proof}
    This is the direct application of Proposition \ref{prop: forward direction} then Proposition \ref{prop: cut ineq}.
\end{proof}
\begin{lemma} \label{lemma: transistive}
    If~$\cx$ and~$\cy$ are maximal tubings of the path graph and~$\cx < \cy$, then there exists a descent in~$\cy$ that is a co-inversion in~$\cx$.
\end{lemma}
\begin{proof}
    If~$\cx \lessdot \cy$ this is clear, as the tree move turns an ascent in~$\cx$ in to a descent in~$\cy$. \par 
    Say that~$\cx < \cx' \lessdot \cy$. So there exists an ascent~$(i<j)$ in~$\asc(\cx')$ that is turned in to a descent in~$\dec(\cy)$ in the following manner:
\begin{center}
 \begin{tikzpicture}[scale=0.8]
      \begin{scope}[xshift=0in]
         \draw (0,0.5) node[draw, circle] (T4)    {};
         \draw (0,-1) node[draw, circle] (y)    {$j$};
         \draw (-1.5,-2) node[draw, circle] (x) {$i$};
         \draw (-2.5,-3) node[draw, circle] (T1){};
         \draw (-0.5,-3) node[draw, circle] (T2) {};
         \draw (1,-2) node[draw, circle] (T3) {};

         \draw (T4)--(y)--(x)--(T1);
         \draw (x)--(T2);
         \draw (y)--(T3);
         \draw (0,-4) node (NM) {$\cx$};
     \end{scope} 
     \begin{scope}[xshift = 1.5in, yshift = -0.5in]
         \draw (0,0) node {\Large${ \longrightarrow}$};
     \end{scope}
     \begin{scope}[xshift = 3in]
        \draw (0,0.4) node[draw, circle] (T4)     {};
         \draw (1.5,-2) node[draw, circle] (y)    {$j$};
         \draw (0,-1) node[draw, circle] (x) {$i$};
         \draw (-1,-2) node[draw, circle] (T1) {};
         \draw (0.5,-3) node[draw, circle] (T2) {};
         \draw (2.5,-3) node[draw, circle] (T3) {};

         \draw (T4)--(x)--(y)--(T3);
         \draw (x)--(T1);
         \draw (y)--(T2);
         \draw (0,-4) node (NM) {$\cy$};
     \end{scope}
 \end{tikzpicture}
 \end{center}
    By Proposition~\ref{prop:path_order}, ~$\coinv(\cx') \subset \coinv(\cx)$. So~$(i,j) \in \dec(\cy)$ and~$(i,j) \in \asc(\cx') \subset \coinv(\cx') \subset \coinv(\cx)$.
\end{proof}

\begin{lemma} \label{lemma: Inversion}
    If~$\inv(\cj) \subseteq \inv(\ck) \cup \inc(\ck)$ and~$\Cut(\cj) < \Cut(\ck)$, then there exists a descent in~$\ck$ that is not an inversion of~$\cj$.
\end{lemma}
\begin{proof}
    By Lemma \ref{lemma: transistive}, we have that there is a descent in~$\Cut(\ck)$ that is a co-inversion in~$\Cut(\cj)$. \par 
    The $\Cut$ map only removes relations, so if two elements are related in~$\Cut(\cj)$ then they are related in~$\cj$. Let~$(b'<b)$ be the descent in~$\Cut(\ck)$ that is a co-inversion in~$\Cut(\cj)$. So~$b' <_{\cj} b$ and~$b \lessdot_{\Cut(\ck)} b'$. We know that~$b <_{\ck}b'$, but this may not be a descent. If~$b \lessdot_\ck b'$ then we are done. \par  
    There is only one way that a descent in~$\Cut(\ck)$ is turned in to a non-descent inversion in the shuffle operation that turns~$\Cut(\ck)$ back in to~$\ck$, and that is, $b$ and $b'$ are both elements of the right zipper of $\Cut(\ck)$. All other cover relations are either preserved or are ascents. In this case,~$b \lessdot_{\ck}a$ where~$a$ is some element in the left zipper of $\Cut(\ck)$. Note this ensures that~$a<b'$ as well. In particular, the shuffle permutation has contiguous substring~$ba \cdots b'$.\par 
     So~$b \lessdot_{\ck}a$. If~$b \not<_{\cj} a$, then we have the claim (because~$a<b$ wouldn't be an inversion in~$\cj$). Assume for contradiction that~$b <_{\cj} a$. So we have that~$b \lessdot_{\ck}a <_{\ck}b'$ and~$b' <_{\cj} b <_{\cj} a$. This means~$(a,b')$ is a co-inversion in~$\ck$ and an inversion in~$\cj$. This contradicts our assumption that~$\inv(\cj) \subseteq \inv(\ck) \cup \inc(\ck)$, so the descent~$(a,b)$ in~$\dec(\ck)$ cannot be an inversion in~$\cj$. 
\end{proof}

\begin{lemma}\label{lem: inv subset implies order rel}
    Let $\cj,\ck\in\MTub(C_n)$ with $\cj\ne\ck$. If~$\inv(\cj) \subseteq \inv(\ck) \cup \inc(\ck)$ and~$\Cut(\cj) = \Cut(\ck)$, then there exists a descent in~$\ck$ that is not an inversion of~$\cj$.
\end{lemma}
\begin{proof}
    The maximal tubings~$\cj$ and~$\ck$ are given by two different in-order shufflings of the left and right zippers~$a = (a_1,...,a_l)$ and~$b = (b_1,...,b_r)$ in~$\cx := \Cut(\cj) = \Cut(\ck)$. That is, $\cj = \Sew_\cx(w)$ and $\ck = \Sew_\cx(v)$ for $w,v \in W_\cx$. Since $\inv(\cj) \subseteq \inv(\ck) \cup \inc(\ck)$, we must have that $\inv(w) \subseteq \inv(v)$, and therefore $w<v$ in the weak order. Hence there is a descent in~$v$ that is a co-inversion in~$w$, and a corresponding descent in $\ck$ that is not an inversion of $\cj$. 
\end{proof}
We are now able to move to the second part of the overall proof of the backwards direction of Theorem~\ref{thm:inversion_order}.

\begin{proposition} \label{prop: backwards outline_R}
     If~$\inv(\cj)\subseteq \inv(\ck)\cup\inc(\ck)$ and there exists a descent in~$\dec(\ck)$ not in~$\inv(\cj)$ that is either the top-edge or a right-edge in $G_\ck$, then there exists a tubing~$\ck' \lessdot \ck$ such that~$\inv(\cj) \subseteq \inv(\ck') \cup \inc(\ck')$.
\end{proposition}
\begin{proof}
First, consider the maximal elements~$m_1$ and~$m_2$ in $G_\cj$ and $G_\ck$, respectively. If~$m_1 < m_2$, then we have~${(m_1,m_2) \in \inv(\cj)}$ and~${(m_1,m_2) \in \coinv(\ck)}$. So~$m_1 \geq m_{2}$.\par 
Consider a top- or right-edge descent~$(x<y) \in \dec(\ck) \setminus \inv(\cj)$. \par 
(Case 1: Top Edge) If~$(x,y)$ is the top edge in~$\ck$ (so~$x = m_2$), then the tubing~$\ck'$ corresponding to the tree move that swaps those elements easily satisfies~$\inv(\cj) \subseteq \inv(\ck') \cup \inc(\ck')$.\par 
(Case 2: Right Edge) Now consider the right-edge case, and let~$\ck'$ be the tubing for the cover relation~$\ck' \lessdot \ck$ pictured below.
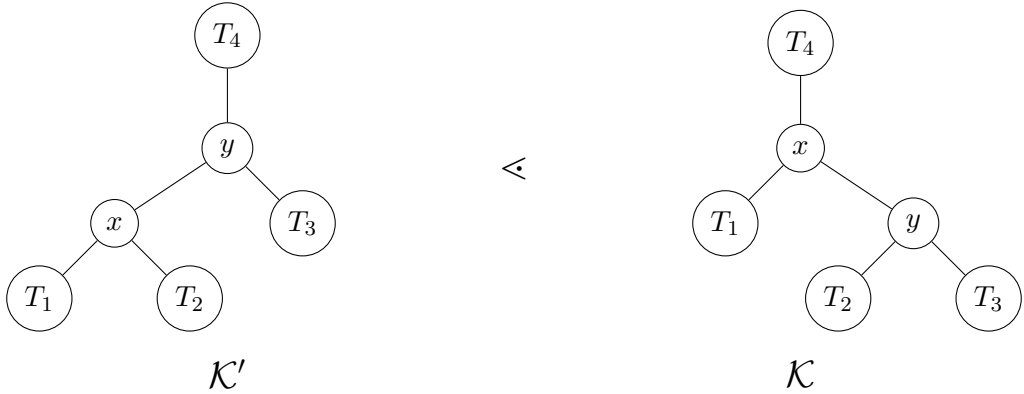
\begin{center}
 \begin{tikzpicture}
      \begin{scope}
         \draw (0,0.5) node[draw, circle] (T4)     {$T_4$};
         \draw (0,-1) node[draw, circle] (y)     {$y$};
         \draw (-1.5,-2) node[draw, circle] (x)  {$x$};
         \draw (-2.5,-3) node[draw, circle] (T1) {$T_1$};
         \draw (-0.5,-3) node[draw, circle] (T2) {$T_2$};
         \draw (1,-2) node[draw, circle] (T3) {$T_3$};
         \draw (0,-4) node (name) {\Large$\ck'$};

         \draw (T4)--(y)--(x)--(T1);
         \draw (x)--(T2);
         \draw (y)--(T3);
     \end{scope} 
     \begin{scope}[xshift = 1.5in, yshift = -0.5in]
         \draw (0,0) node {\Large${ \lessdot}$};
     \end{scope}
     \begin{scope}[xshift = 3in]
        \draw (0,0.4) node[draw, circle] (T4)     {$T_4$};
         \draw (1.5,-2) node[draw, circle] (y)     {$y$};
         \draw (0,-1) node[draw, circle] (x)  {$x$};
         \draw (-1,-2) node[draw, circle] (T1) {$T_1$};
         \draw (0.5,-3) node[draw, circle] (T2) {$T_2$};
         \draw (2.5,-3) node[draw, circle] (T3) {$T_3$};
         \draw (0,-4) node (name) {\Large$\ck$};

         \draw (T4)--(x)--(y)--(T3);
         \draw (x)--(T1);
         \draw (y)--(T2);
     \end{scope}
 \end{tikzpicture}
 \end{center}\par 
We claim that~$\inv(\cj) \subseteq \inv(\ck') \cup \inc(\ck')$. Assume for contradiction that there exists a coinversion~$(a<y)$ in~$\ck'$ that is an inversion in~$\inv(\cj)$. So $y \in \downJ(a)$ and $a \in \downK'(y)$. The new order relations in~$\ck'$ that are not in~$\ck$ are between elements in~$T_1$ and~$y$. So~$a \in T_1$ and~$a <_{m_2} x <_{m_2} y$ in the cyclic order. Since~$a < y$ in the standard order, we have that~$a < x < y$ in the standard order. \par
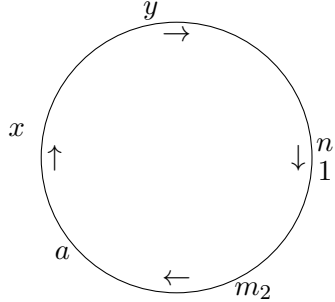
\begin{center}
\begin{tikzpicture}[scale=0.9]
    \draw (0,0) circle (2cm);
    \draw (0:1.8cm) node (ar1) {$\downarrow$};
    \draw (90:1.8cm) node (ar2) {$\rightarrow$};
    \draw (180:1.8cm) node (ar3) {$\uparrow$};    
    \draw (270:1.8cm) node (ar4) {$\leftarrow$};
    \draw (100:2.2cm) node (y) {$y$};
    \draw (220:2.2cm) node (a) {$a$};
    \draw (170:2.4cm) node (mJ) {$x$};
    \draw  (5:2.2cm) node (n) {$n$};
    \draw (-5:2.2cm) node (n1)  {$1$};
    \draw (-60:2.27) node (x) {$m_2$};
\end{tikzpicture}
\end{center}\par 
If~$x \in \downJ(a)$ then~$(a,x) \in \inv(\cj)$, but~$(a,x) \in \coinv(\ck)$, which contradicts our original assumption. If~${x \notin \downJ(a)}$, then~$m_1$ is strictly between~$a$ and~$y$ in the standard order (otherwise the subgraph induced by $\downJ(a)$ is not connected). For the same reason, we get that~$m_1 \in \downK(x)$. We conclude the proof by considering each of the following cases:\par 
\begin{itemize}
     \item If $m_1 = x$, then $(m_1,y) = (x,y) \in \inv(\cj)$.
\item If $m_1 = y$, then $y \notin \downJ(a)$, so $(a,y) \notin \inv(\cj)$.
\item If $m_1 \in T_1$, then $a < m_1 < x$ in the standard order. So $(m_1,x) \in \inv(\cj)$ and $(m_1,x) \in \coinv(\ck)$.   
\item If $m_1 \in T_2 \cup T_3$, then~$(m_1,y) \in \coinv(\ck)$, but~$(m_1,y) \in \inv(\cj)$. 
\end{itemize}

In any case, there is a contradiction with either the original assumption that $\inv(\cj) \subseteq \inv(\ck) \cup \inc(\ck)$ and $(x,y) \notin \inv(\cj)$, or with our assumption that $(a,y)$ was an inversion in $\cj$. It follows that no such~$a$ exists, so $\inv(\cj) \subseteq \inv(\ck') \cup \inc(\ck')$.
\end{proof}
\begin{proposition}\label{prop: backwards outline_L}
    If~$\inv(\cj)\subseteq \inv(\ck)\cup\inc(\ck)$ and there exists a descent in~$\dec(\ck)$ not in~$\inv(\cj)$ that is a left-edge in $G_\ck$, then there exists a tubing~$\ck' \lessdot \ck$ such that~$\inv(\cj) \subseteq \inv(\ck') \cup \inc(\ck')$.
\end{proposition}
\begin{proof}
As in Proposition \ref{prop: backwards outline_L}, consider the maximal elements~$m_1$ and~$m_2$ in the~$G$-trees of~$\cj$ and~$\ck$, respectively. If~$m_1 < m_2$ then~$(m_1,m_2) \in \inv(\cj)$ and~$(m_1,m_2) \in \coinv(\ck)$. So~$m_1 \geq m_{2}$.\par 
Consider a left-edge descent~$(x<y) \in \dec(\ck) \setminus \inv(\cj)$ (note this means~$m_1 \neq y$). Let~$\ck'$ be the tubing for the cover relation~$\ck' \lessdot \ck$ pictured below.
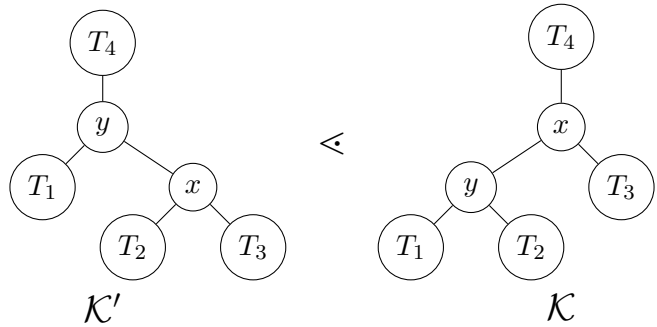
\begin{center}
 \begin{tikzpicture}[scale=0.8]
      \begin{scope}[xshift=3in]
         \draw (0,0.5) node[draw, circle] (T4)     {$T_4$};
         \draw (0,-1) node[draw, circle] (y)     {$x$};
         \draw (-1.5,-2) node[draw, circle] (x)  {$y$};
         \draw (-2.5,-3) node[draw, circle] (T1) {$T_1$};
         \draw (-0.5,-3) node[draw, circle] (T2) {$T_2$};
         \draw (1,-2) node[draw, circle] (T3) {$T_3$};
         \draw (0,-4) node (name) {\Large$\ck$};

         \draw (T4)--(y)--(x)--(T1);
         \draw (x)--(T2);
         \draw (y)--(T3);
     \end{scope} 
     \begin{scope}[xshift = 1.5in, yshift = -0.5in]
         \draw (0,0) node {\Large${ \lessdot}$};
     \end{scope}
     \begin{scope}
        \draw (0,0.4) node[draw, circle] (T4)     {$T_4$};
         \draw (1.5,-2) node[draw, circle] (y)     {$x$};
         \draw (0,-1) node[draw, circle] (x)  {$y$};
         \draw (-1,-2) node[draw, circle] (T1) {$T_1$};
         \draw (0.5,-3) node[draw, circle] (T2) {$T_2$};
         \draw (2.5,-3) node[draw, circle] (T3) {$T_3$};
         \draw (0,-4) node (name) {\Large$\ck'$};

         \draw (T4)--(x)--(y)--(T3);
         \draw (x)--(T1);
         \draw (y)--(T2);
     \end{scope}
 \end{tikzpicture}
 \end{center}
We claim that~$\inv(\cj) \subseteq \inv(\ck') \cup \inc(\ck')$. Say there exists a coinversion~$(a<y)$ in~$\ck'$ that is an inversion in~$\inv(\cj)$. If there is not, then we are done. In the other case, this proof will proceed by reducing to Proposition \ref{prop: backwards outline_R}. \par 
The order relations that are in~$\ck'$ but not in~$\ck$ are between elements in~$T_3$ and~$y$. So~$a \in T_3$ and~$y <_m x <_m a$ in the cyclic order for~$\ck$ and~$\ck'$. In particular, when moving clockwise around~$C_n$, the element~$m_2$ appears after~$a$ and before~$y$, and the element~$x$ appears after~$y$ and before~$a$. Since in addition~$x < y$ numerically, $x$ must appear after~$1$ and before~$a$. 
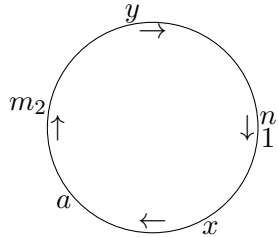
\par
\begin{center}
\begin{tikzpicture}[scale=0.7]
    \draw (0,0) circle (2cm);
    \draw (0:1.8cm) node (ar1) {$\downarrow$};
    \draw (90:1.8cm) node (ar2) {$\rightarrow$};
    \draw (180:1.8cm) node (ar3) {$\uparrow$};    
    \draw (270:1.8cm) node (ar4) {$\leftarrow$};
    \draw (100:2.2cm) node (y) {$y$};
    \draw (220:2.2cm) node (a) {$a$};
    \draw (170:2.4cm) node (mJ) {$m_2$};
    \draw  (5:2.2cm) node (n) {$n$};
    \draw (-5:2.2cm) node (n1)  {$1$};
    \draw (-60:2.2) node (x) {$x$};
\end{tikzpicture}
\end{center}
Since~$m_2 \leq m_1$, there are two possibilities for~$m_1$. Either,~$y<m_1 \leq n$, or~$m_2 \leq m_1 < y$. \par 
If~$y  < m_1 \leq n$, then any tube in~$\cj$ containing~$y$ and~$a$ contains~$m_2$. The tube~$\downJ(a)$ in~$\cj$ contains both~$a$ and~$y$. So~$m_2 \in \downJ(a)$, and so~$(a,m_2) \in \inv(\cj)$. However,~$m_{2} >_{\ck}a$ and~$m_2 > a$, so~$(a,m_2) \in \coinv(\ck)$. This contradicts our original assumption, and so cannot be the case. \par 
If~$m_2 \leq m_1 < y$ then~$x \in \downJ(a)$. In particular~$x <_{\cj} a$. Let~$(x,a_1,...,a_r)$ be the sequence of right-edges down from~$x$ in~$\ck$. We know~$a_1$ exists otherwise~$T_3$ is empty. We also know that each~$a_i$ appears after~$x$ and before~$m_
cj$ on the cycle, so~$x < a_1 < \cdots < a_r < m_2$ numerically. In particular, each~$(a_i,a_{i+1})$ is a descent. If any of the descents~$(a_i,a_{i+1}) \in \dec(\ck)$ are not inversions in~$\inv(\cj)$, then we can apply Proposition \ref{prop: backwards outline_R} to this descent and we are done. \par 
Say for contradiction that all the descents~$(a_i,a_{i+1})$ are inversions in~$\inv(\cj)$. Then 
\[a_r <_{\cj} a_{r-1} <_{\cj} \cdots <_{\cj} a_1 <_{\cj} x <_{\cj} a.\]
So~$a \neq a_i$ for any~$i$, and~$a \notin \downJ(x)$ but~$a_i \in \downJ(x)$ for~$i=1,...,r$. Since~$x < a < m_2 \leq m_1 < y$, each~$a_i$ must appear after~$x$ and before~$a$, so~$a > a_i$ for~$i=0,...,r$. In particular, $a>a_r$.
But~$a_r$ is the largest element (in both the normal and circular orderings for~$\ck$) in~$T_3$. So $a_r>a$, thus we have our contradiction. 
We conclude that one of these descents~$(a_i,a_{i+1})$ in~$\dec(\ck)$ must not be an inversion for~$\cj$, and we can apply Proposition \ref{prop: backwards outline_R}.
\end{proof}

\begin{proposition} \label{prop: Backwards}
    If~$\cj$ and~$\ck$ are tubings of the cycle graph, and~$\inv(\cj) \subseteq \inv(\ck) \cup \inc(\ck)$ then~$\cj\leq \ck$
\end{proposition}
\begin{proof}
    If~$\cj = \ck$ then we are done by Proposition \ref{prop:inv_equality}. So assume that~$\cj \neq \ck$. Proposition ~\ref{prop: cut ineq} gives that~$\Cut(\cj) \leq \Cut(\ck)$ in the poset of tubings of the path. Lemma ~\ref{lemma: transistive} states that if~$\Cut(\cj) < \Cut(\ck)$ then there exists a descent in~$\Cut(\ck)$ that is a co-inversion in~$\Cut(\cj)$. Together, Lemmas ~\ref{lemma: Inversion} and \ref{lem: inv subset implies order rel} state that there exists a descent in~$\ck$ that is not an inversion of~$\cj$. By Propositions \ref{prop: backwards outline_R} and \ref{prop: backwards outline_L}, there exists a~$\ck' \lessdot \ck$ such that~$\inv(\cj) \subseteq \inv(\ck') \cup \inc(\ck')$. \par 
    If~$\cj = \ck'$ then we have proven the claim. If not, we may repeat this process with~$\ck'$ to get a~$\ck'' \lessdot \ck'$, then~$\ck^{(3)} \lessdot \ck''$ and so on constructing~$\ck^{(i+1)} \lessdot \ck^{(i)}$ where~$\inv(\cj) \subseteq \inv(\ck^{(i+1)}) \cup \inc(\ck^{(i+1)})$ as long as~$\ck^{(i)} \neq \cj$. Since~$\MTub(C_n)$ is finite, this process must eventually terminate with some~$k$ such that~$\cj = \ck^{(k)}$, and we have the claim.
\end{proof}

\section{Lattice}\label{sec:Lattice}
This section will prove that~$\MTub(C_n)$ is a lattice, achieved by proving the existence of joins and meets. Again, the proof will heavily leverage the $\Cut$ map as well as the characterization of order in~$\MTub(C_n)$ given by Theorem \ref{thm:inversion_order}. \par 
First, given~$\cj \in \MTub(C_n)$ and~$\cx \in \MTub(P_n)$ such that~$\Cut(\cj) \lessdot \cx$, we will explicitly construct a specific element~$\cj^\cx \in \MTub(C_n)$ such that~$\Cut(\cj^\cx) = \cx$. Then, Lemmas \ref{lem:Tx_greater} and \ref{lem:key_lemma} will give properties of this~$\cj^\cx$. \par 
Let~$\cj \in \MTub(C_n)$. Then~$\Cut(\cj)$ is a maximal tubing of the path graph $P_n$.  Its $G$-tree takes the following form (cf. Figure \ref{fig:path_zipper}):
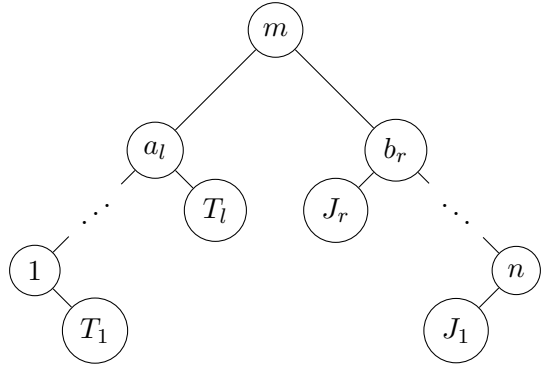
\begin{center}
    \begin{tikzpicture}[scale=0.8]
        \draw (0,0) node[draw,circle] (m) {$m$};
        
        \draw (-2,-2) node[draw,circle] (al) {$a_l$};
        \draw (-3,-3) node (ldts) {\raisebox{4pt}{$\iddots$}};
        \draw (-4,-4) node[draw,circle] (a2) {$1$};

        \draw (2,-2) node[draw,circle] (br) {$b_r$};
        \draw (3,-3) node (rdts) {\raisebox{4pt}{$\ddots$}};
        \draw (4,-4) node[draw,circle] (b2) {$n$};

        \draw (a2)--(ldts)--(al)--(m)--(br)--(rdts)--(b2);

        \draw (-1,-3) node[draw,circle] (Tl)  {$T_l$};
        \draw (-3,-5) node[draw,circle] (T2)  {$T_1$};
        \draw (1,-3) node[draw,circle] (Jr)  {$J_r$};
        \draw (3,-5) node[draw,circle] (J2)  {$J_1$};

        \draw (a2)--(T2);
        \draw (al)--(Tl);
        \draw (b2)--(J2);
        \draw (br)--(Jr);
        
    \end{tikzpicture}
\end{center}
And by the results in Section \ref{sec:Cut Map}, the tubing~$\cj = \Sew_{\Cut(\cj)}(w)$ is uniquely determined by~$w$: the in-order shuffling of the left~$a=(a_1,a_2,...,a_l)$ and right~$b=(b_1,...,b_r)$ sides of the zipper of~$\Cut(\cj)$.\par 
Using this notation, we get the following form for the~$G_\cj$:
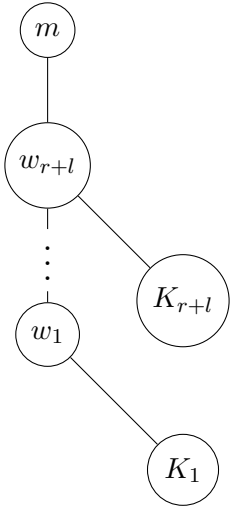
\begin{center}
\begin{tikzpicture}[scale=1.8]
    \draw (0,0) node[draw, circle] (m) {$m$};
    \draw (0,-1) node[draw, circle] (wrl) {$w_{r+l}$};
    \draw (0,-1.7) node[inner sep=-2pt] (dts) {\Large\raisebox{6pt}{$\vdots$}};
    \draw (0,-2.25) node[draw, circle] (w2) {$w_1$};

    \draw (1,-2) node[draw,circle] (Krl) {$K_{r+l}$};
    \draw (1,-3.25) node[draw,circle] (K2) {$K_{1}$};

    \draw (m)--(wrl)--(dts)--(w2)--(K2);
    \draw (wrl)--(Krl);

\end{tikzpicture}
\end{center}
with the caveat that if~$w_s=a_i$ is in the left zipper then~$K_s=T_i$ is a right subtree (as it appears above) and if~$w_s=b_i$ is in the right zipper then~$K_s = J_i$ actually is a left subtree. 

Let~$\cx$ be a tubing of the path graph~$P_n$ that covers~$\Cut(\cj)$, so~$\Cut(\cj)\lessdot \cx$. This means that~$\Cut(\cj)$ is related to~$\cx$ by a single tree move, which turns a left edge in~$\Cut(\cj)$ (i.e. an ascent) in to a right-edge in~$\cx$ (i.e. a descent). This left edge in~$\Cut(\cj)$ can be in one of the following four places:
\begin{enumerate}
    \item entirely contained within a~$T_i$ or~$J_i$,
    \item connect the top element of~$J_i$ to the corresponding~$b_i$ in the right part of the zipper,
    \item be an edge connecting~$a_i$ to~$a_{i+1}$ in the left zipper, or
    \item Connect~$a_l$ to the maximal element~$m$.
\end{enumerate}
We are going to construct a special element~$\cj^\cx$ such that~$\cj < \cj^\cx$ in the maximal tubings of the cycle graph and~$\Cut(\cj^\cx) = \cx$. This element will be slightly different for each type of tree move listed above, but in any case~$\cj^\cx$ can be given by designating a shuffling~$w^\cx$ of the zipper of~$\cx$.
\begin{enumerate}
    \item If the left edge for the tree move connecting~$\Cut(\cj)$ to~$\cx$ is contained entirely within a~$T_i$ or~$J_i$ subtree, then the zippers of~$\Cut(\cj)$ and~$\cx$ are identical, and we let~$w^\cx = w$.
    \item If the left edge for the tree move connecting~$\Cut(\cj)$ to~$\cx$ connects the top element~$b_0$ of~$J_i$ to some~$b_i$ in~$\Cut(\cj)$, then the left zipper of~$\cx$ is identical to the left zipper of~$\Cut(\cj)$ and the right zipper of~$\cx$ has~$b_0$ inserted between~$b_i$ and~$b_{i+1}$. We construct~$w^\cx$ from~$w$ by replacing the string~$\cdots b_i w_s \cdots$ in~$w$ with~$\cdots b_i b_0 w_s\cdots~$.
    \item If the left edge for the tree move connecting~$\Cut(\cj)$ to~$\cx$ connects~$a_i$ to~$a_{i+1}$, then the left zipper of~$\cx$ has~$a_{i+1}$ removed from the left zipper of~$\Cut(\cj)$ and the right zipper is identical. We construct~$w^\cx$ from~$w$ by replacing~$\cdots w_ta_iw_{t+2} \cdots w_sa_{i+1}w_{s+2}$ with~$\cdots w_tw_{t+2} \cdots w_sa_iw_{s+2}\cdots$. Note that every element of~$w$ between~$a_i$ and~$a_{i+1}$ must be an element of~$b$. 
    \item If the left edge for the tree move connecting~$\Cut(\cj)$ to~$\cx$ is the one connecting~$a_\ell$ with~$m$, then the left zipper of~$\cx$ loses~$a_l$ from~$\Cut(\cj)$ and the right zipper gains~$m$. We construct~$w^\cx$ from~$w$ by replacing~$\cdots w_ta_lw_{t+2} \cdots w_{s+l}$ with~$\cdots w_tw_{t+2} \cdots w_{s+l}m$.
\end{enumerate}

In each case we can get a picture of~$\cj^\cx$ as well.
\begin{enumerate}
    \item  If the left edge for the tree move connecting~$\Cut(\cj)$ to~$\cx$ is contained entirely within a~$T_i$ or~$J_i$ subtree, say within a~$K_s$. Let~$K_s'$ be the new sub-tree for~$X$, then~$\cj$ to~$\cj^\cx$ is: 
    
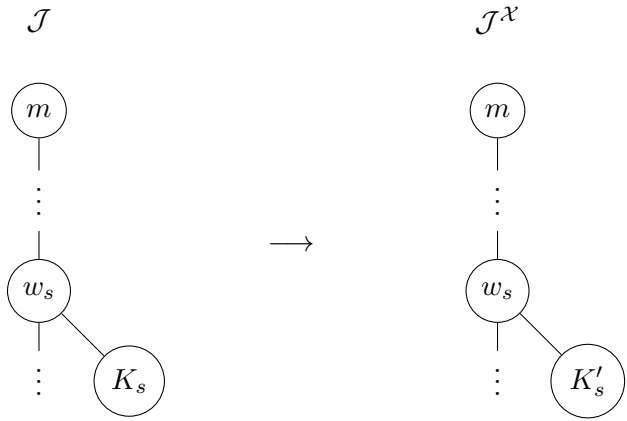
\begin{center}
    \begin{tikzpicture}[scale=1.2]
    \begin{scope}
        \draw (0,1) node (name) {$\cj^\cx$};
        \draw (0,0) node[draw,circle] (m) {$m$};
        \draw (0,-1) node[inner sep=0pt] (dts) {\raisebox{6pt}{$\vdots$}};
        \draw (0,-2) node[draw,circle] (w) {$w_s$};
        \draw (0,-3) node[inner sep=0pt] (dts2) {\raisebox{6pt}{$\vdots$}};
        \draw (1,-3) node[draw,circle] (K) {$K_s'$};

        \draw (m)--(dts)--(w)--(dts2);
        \draw (w)--(K);
    \end{scope}
    \begin{scope}[xshift = -0.9in,yshift=-1.5cm]
        \draw (0,0) node (arr) {$\longrightarrow$};
    \end{scope}
    \begin{scope}[xshift=-2in]
        \draw (0,1) node (name) {$\cj$};
        \draw (0,0) node[draw,circle] (m) {$m$};
        \draw (0,-1) node[inner sep=0pt] (dts) {\raisebox{6pt}{$\vdots$}};
        \draw (0,-2) node[draw,circle] (w) {$w_s$};
        \draw (0,-3) node[inner sep=0pt] (dts2) {\raisebox{6pt}{$\vdots$}};
        \draw (1,-3) node[draw,circle] (K) {$K_s$};

        \draw (m)--(dts)--(w)--(dts2);
        \draw (w)--(K);
    \end{scope}
    \end{tikzpicture}
    \end{center}
    
    \item If the left edge for the tree move connecting~$\Cut(\cj)$ to~$\cx$ connects the top element~$b_0$ of~$J_i$ to some~$b_i$ in~$\Cut(\cj)$, let~$J_i^1$ and~$J_i^2$ be the left and right subtrees of~$b_0$ respectively. Then~$\cj$ to~$\cj^\cx$ is:
    
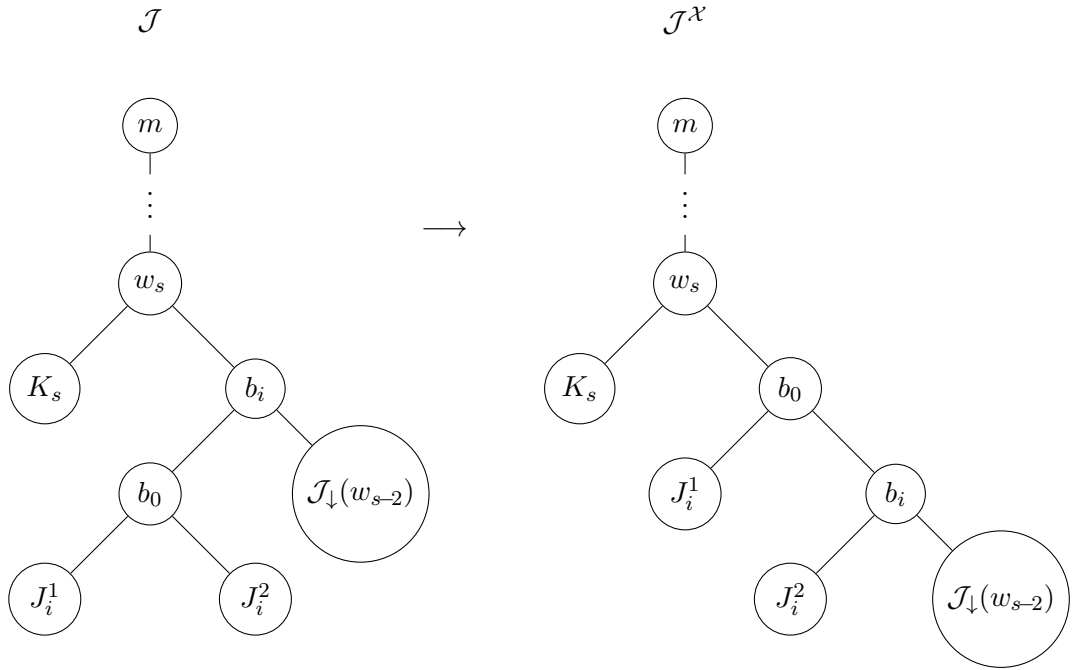
\begin{center}
    \begin{tikzpicture}[scale=1.4]
        \begin{scope}
        \draw (0,0.5) node (name) {$\cj$};
        \draw (0,-0.5) node[draw,circle] (m) {$m$};
        \draw (0,-1.25) node[inner sep=0pt] (dts) {\raisebox{6pt}{$\vdots$}};
        \draw (0,-2) node[draw,circle] (ws) {$w_s$};
        \draw (-1,-3) node[draw,circle] (Ks) {$K_s$};
        \draw (1,-3) node[draw,circle] (bi) {$b_i$};
        \draw (0,-4) node[draw,circle] (b0) {$b_0$};
        \draw (2,-4) node[draw,circle] (down)  {$\downJ(w_{s\!-\!2})$};
        \draw (-1,-5) node[draw,circle] (J1) {$J_i^1$};
        \draw (1,-5) node[draw,circle] (J2) {$J_i^2$};
        \draw (m)--(dts)--(ws)--(Ks);
        \draw (ws)--(bi)--(b0)--(J1);
        \draw (b0)--(J2);
        \draw (bi)--(down);
    \end{scope}
    \begin{scope}[xshift = 1.1in,yshift=-1.5cm]
        \draw (0,0) node (arr) {$\longrightarrow$};
    \end{scope}
    \begin{scope}[xshift=2in]
        \draw (0,0.5) node (name) {$\cj^\cx$};
        \draw (0,-0.5) node[draw,circle] (m) {$m$};
        \draw (0,-1.25) node[inner sep=0pt] (dts) {\raisebox{6pt}{$\vdots$}};
        \draw (0,-2) node[draw,circle] (ws) {$w_s$};
        \draw (-1,-3) node[draw,circle] (Ks) {$K_s$};
        \draw (1,-3) node[draw,circle] (b0) {$b_0$};
        \draw (0,-4) node[draw,circle] (J1) {$J_i^1$};
        \draw (2,-4) node[draw,circle] (bi) {$b_i$};
        \draw (1,-5) node[draw,circle] (J2) {$J_i^2$};
        \draw (3,-5) node[draw,circle] (down) {$\downJ(w_{s\!-\!2})$};
        \draw (m)--(dts)--(ws)--(Ks);
        \draw (ws)--(b0)--(J1);
        \draw (b0)--(bi)--(J2);
        \draw (bi)--(down);
    \end{scope}
    \end{tikzpicture}
    \end{center}

    \item If the left edge for the tree move connecting~$\Cut(\cj)$ to~$\cx$ connects~$a_i$ to~$a_{i+1}$, then every element in~$w$ between~$a_i$ and~$a_{i+1}$ is an element of~$b$. Say that string is~$b_p$ through~$b_q$ (so in our above notation for (3),~$b_p = w_{t+2}$ and~$b_q = w_s$). Then~$\cj$ to~$\cj^\cx$ is: 
    
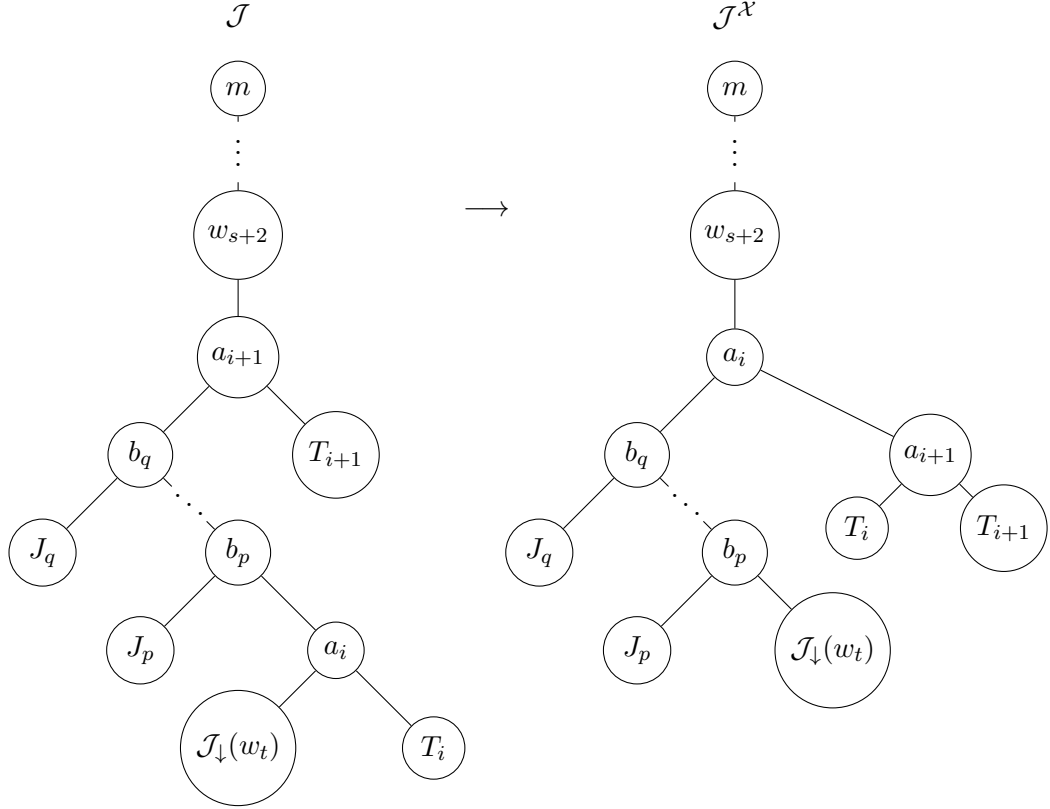
\begin{center}
    \begin{tikzpicture}[scale=1.3]
        \begin{scope}
        \draw (0,0.5) node (name) {$\cj$};
        \draw (0,-0.25) node[draw,circle] (m) {$m$};
        \draw (0,-0.9) node[inner sep=0pt] (dts) {\raisebox{6pt}{$\vdots$}};
        \draw (0,-1.75) node[draw,circle] (ws2) {$w_{s+2}$};
        \draw (0,-3) node[draw,circle] (ai1) {$a_{i+1}$};
        \draw (1,-4) node[draw,circle] (Ji1) {$T_{i+1}$};
        \draw (-1,-4) node[draw,circle] (bq) {$b_q$};
        \draw (-0.5,-4.5) node[inner sep=0pt] (dts2) {\raisebox{6pt}{$\ddots$}}; 
        \draw (0,-5) node[draw,circle] (bp) {$b_p$};
        \draw (1,-6) node[draw,circle] (ai) {$a_i$};
        \draw (2,-7) node[draw,circle] (Ji) {$T_i$};
        \draw (0,-7) node[draw,circle] (down) {$\downJ(w_t)$};
        \draw (-2,-5) node[draw,circle] (Jq) {$J_q$};
        \draw (-1,-6) node[draw,circle] (Jp) {$J_p$};

        \draw (m)--(dts)--(ws2)--(ai1)--(bq)--(dts2)--(bp)--(ai)--(down);
        \draw (ai1)--(Ji1);
        \draw (ai)--(Ji);
        \draw (bq)--(Jq);
        \draw (bp)--(Jp);
        
    \end{scope}
    \begin{scope}[xshift = 1in,yshift=-1.5cm]
        \draw (0,0) node (arr) {$\longrightarrow$};
    \end{scope}
    \begin{scope}[xshift=2in]
        \draw (0,0.5) node (name) {$\cj^\cx$};
        \draw (0,-0.25) node[draw,circle] (m) {$m$};
        \draw (0,-0.9) node[inner sep=0pt] (dts) {\raisebox{6pt}{$\vdots$}};
        \draw (0,-1.75) node[draw,circle] (ws2) {$w_{s+2}$};
        \draw (0,-3) node[draw,circle] (ai) {$a_{i}$};
        \draw (2,-4) node[draw,circle] (ai1) {$a_{i+1}$};
        \draw (1.25,-4.75) node[draw,circle] (Ji) {$T_i$};
        \draw (2.75,-4.75) node[draw,circle] (Ji1) {$T_{i+1}$};
        \draw (-1,-4) node[draw,circle] (bq) {$b_q$};
        \draw (-0.5,-4.5) node[inner sep=0pt] (dts2) {\raisebox{6pt}{$\ddots$}}; 
        \draw (0,-5) node[draw,circle] (bp) {$b_p$};
        \draw (1,-6) node[draw,circle] (down) {$\downJ(w_t)$};
        \draw (-2,-5) node[draw,circle] (Jq) {$J_q$};
        \draw (-1,-6) node[draw,circle] (Jp) {$J_p$};

        \draw (m)--(dts)--(ws2)--(ai)--(bq)--(dts2)--(bp)--(down);
        \draw (ai1)--(Ji1);
        \draw (ai)--(ai1)--(Ji);
        \draw (bq)--(Jq);
        \draw (bp)--(Jp);
    \end{scope}
    \end{tikzpicture}
    \end{center}

    \item If the left edge for the tree move connecting~$\Cut(\cj)$ to~$\cx$ connects~$a_l$ to~$m$, then every element in~$w$ after~$a_l$ is an element of~$b$. Say that string is~$b_p$ through~$b_r$ (so in our above notation for (4),~$b_p = w_{t+2}$ and~$b_r = w_{s+l}$). Then~$\cj$ to~$\cj^\cx$ is: 
    
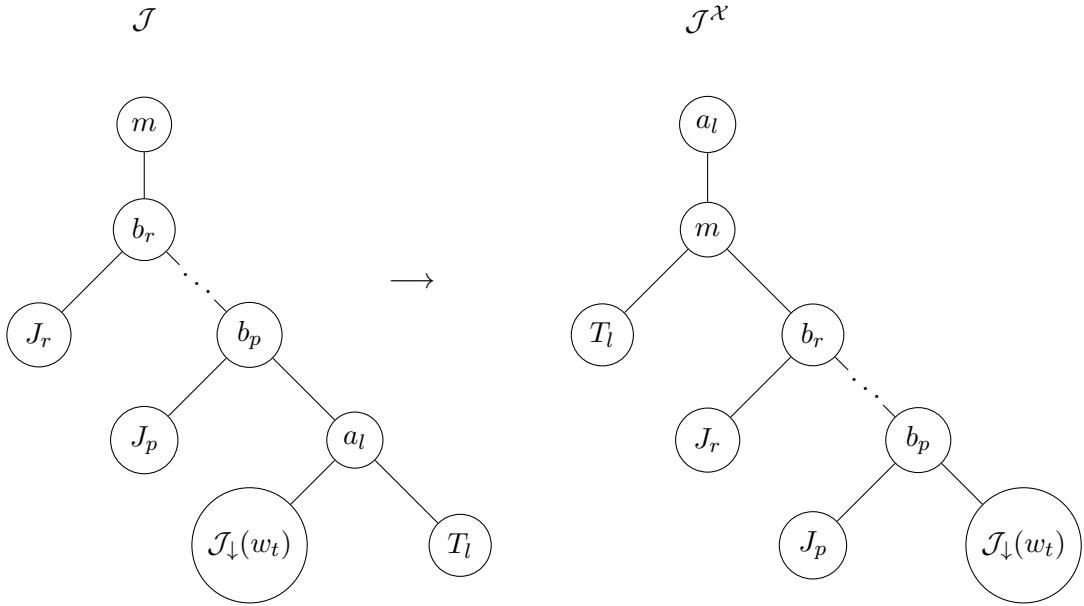
\hspace{-2cm}
    \begin{center}
    \begin{tikzpicture}[scale=1.4]
        \begin{scope}
        \draw (0,1) node (name) {$\cj$};
        \draw (0,0) node[draw,circle] (m) {$m$};
        \draw (0,-1) node[draw,circle] (br) {$b_r$};
        \draw (-1,-2) node[draw,circle] (Jr) {$J_r$};
        \draw (0.5,-1.5) node[inner sep=0pt] (dts) {\raisebox{6pt}{$\ddots$}};
        \draw (1,-2) node[draw,circle] (bp) {$b_p$};
        \draw (0,-3) node[draw,circle] (Jp) {$J_p$};
        \draw (2,-3) node[draw,circle] (ai) {$a_l$};
        \draw (1,-4) node[draw,circle] (down) {$\downJ(w_t)$};
        \draw (3,-4) node[draw,circle] (Ti) {$T_l$};

        \draw (m)--(br)--(dts)--(bp)--(ai)--(Ti);
        \draw (br)--(Jr);
        \draw (bp)--(Jp);
        \draw (ai)--(down);

    \end{scope}
    \begin{scope}[xshift = 1in,yshift=-1.5cm]
        \draw (0,0) node (arr) {$\longrightarrow$};
    \end{scope}
    \begin{scope}[xshift=2.5in,yshift=-1cm]
        \draw (-1,2) node (name) {$\cj^\cx$};
        \draw (-1,1) node[draw,circle] (ai) {$a_l$};
        \draw (-1,0) node[draw,circle] (m) {$m$};
        \draw (-2,-1) node[draw,circle] (Ti) {$T_l$};
        \draw (0,-1) node[draw,circle] (br) {$b_r$};
        \draw (-1,-2) node[draw,circle] (Jr) {$J_r$};
        \draw (0.5,-1.5) node[inner sep=0pt] (dts) {\raisebox{6pt}{$\ddots$}};
        \draw (1,-2) node[draw,circle] (bp) {$b_p$};
        \draw (0,-3) node[draw,circle] (Jp) {$J_p$};
        \draw (2,-3) node[draw,circle] (down) {$\downJ(w_t)$};

        \draw (ai)--(m)--(br)--(dts)--(bp)--(down);
        \draw(m)--(Ti);
        \draw (br)--(Jr);
        \draw (bp)--(Jp);
        
    \end{scope}
    \end{tikzpicture}
    \end{center}
\end{enumerate}
Before we prove the key lemma, we will check our original promise that~$\cj$ is related to~$\cj^\cx$.
\begin{lemma}\label{lem:Tx_greater}
    Given the above constructions,~$\cj < \cj^\cx$.
\end{lemma}
\begin{proof}
    For types (1) and (2), it is rather clear~$\cj$ and~$\cj^\cx$ are related by a single tree move, so~$\cj \lessdot \cj^\cx$.\par 
    For type (3), first doing the tree move in~$\cj$ on the ascent~$(a_i,b_p)$, then on~$(a_i,b_{p+1})$ and so on until~$(a_i,b_q)$ and then~$(a_i,a_{i+1})$ achieves~$\cj^\cx$.\par 
    For type (4), first doing the tree move in~$\cj$ on the ascent~$(a_i,b_p)$, then on~$(a_i,b_{p+1})$ and so on until~$(a_i,b_r)$ and then~$(a_i,m)$ achieves~$\cj^\cx$. 
\end{proof}
So~$\cj < \cj^
\cx$.  Moreover,~$\cj^\cx$ is the unique minimal element in the fiber~$\Cut^{-1}(\cx)$ with this property. Even stronger, as the following key lemma will show, for any element~$\ck$ in \emph{any} fiber~$\Cut^{-1}(\cy)$ such that~$\cx \leq \cy$, an order relation ~$\cj^\cx < \ck$ is equivalent to the order relation ~$\cj<\ck$.
\begin{lemma}\label{lem:key_lemma}
    Let~$\cj \in \MTub(C_n)$, let~$ \Cut(\cj)\lessdot \cx \leq \cy$ in~$\MTub(P_n)$, and let~$\Cut(\ck) = \cy$. Then~$\cj \leq \ck$ if and only if~$\cj^\cx \leq \ck$.

\end{lemma}
\begin{proof}
    Because of Lemma \ref{lem:Tx_greater}, if~$\cj^\cx \leq \ck$ then~$\cj \leq \ck$, so the backwards direction is straightforward.\par 
    Assume that~$\cj \leq \ck$, we seek to prove that~$\cj^\cx \leq \ck$. This argument will heavily depend on the characterization of order given in Proposition \ref{thm:inversion_order}, and proceed based on our cases from above.\par 
    \textbf{Case (1)}. All of the new inversions in~$\inv(\cj^\cx) \setminus \inv(\cj)$ must be contained within the inversions of~$K_s'$, as they are the only new order relations. However, every order relation in~$K_s'$ is also an order relation in~$\cx$. Since~$\cx<\cy$ in the path graph~$\inv(\cx) \subset \inv(\cy)$, and by definition~$\inv(\cy) \subseteq \inv(\ck)$. So 
    \[\inv(\cj^\cx) \setminus \inv(\cj) \subset \inv(\cx) \subset \inv(\cy) \subset \inv(\ck).\]
    Since~$\inv(\cj) \subseteq \inv(\ck) \cup \inc(\ck)$, we get that~$\inv(\cj^\cx) \subseteq \inv(\ck) \cup \inc(\ck)$ and so~$\cj^\cx \leq \ck$.\par 
    \textbf{Case (2)}. All of the new inversions in~$\inv(\cj^\cx) \setminus \inv(\cj)$ must be either~$(b_0,b_i)$ or contained within the the order relations~$(b_0,\downJ(w_{s-2}))$. For these to be new inversions, the particular elements in~$\downJ(w_{s-2})$ need to be \emph{greater} than~$b_0$. Moreover, they need to come from the \emph{right} zipper of~$\cx$, and are therefore already inversions in~$\cx$. More formally,~$\inv(\cj^\cx) \setminus \inv(\cj) \subseteq \inv(\cx)$. Now by the same logic as in (1) we get that~$\cj^\cx \leq \ck$.\par
    \textbf{Case (3)}. The new inversions in~$\inv(\cj^\cx) \setminus \inv(\cj)$ are (from the above characterization) precisely the following:
    \[
    \left(a_i,\{b_p,...,b_q\}\right) \cup \left(a_i,J_p \cup \cdots \cup J_q\right) \cup \left(a_i,\{a_{i+1}\} \cup T_{i+1}\right).
    \]
    The new inversions~$\left(a_i,\{a_{i+1}\} \cup T_{i+1}\right)$ are in~$\cx$ so by the same logic as in (1) and (2) they may be ignored. \par 
    Let~$\beta \in \{b_p,...,b_q\} \cup J_p \cup \cdots \cup J_q$ be arbitrary. Since we need for none of these inversions to be in~$\coinv(\cj)$, the claim is proven if it is shown that~$a_i \not<_{\ck} \beta$. Assume~$a_i <_{\ck} \beta$. By construction~$(a_{i+1},\beta)$ is an inversion in~$\cj$. Also,~$(a_i,a_{i+1})$ is a descent in~$\cx$, so it is an inversion in~$\cy$ and therefore an inversion in~$\ck$. So~$a_{i+1} <_{\ck} a_i$. But then~$a_{i+1} <_{\ck} a_i <_{\ck} \beta$, and we get that~$(a_{i+1},\beta)$ is a \emph{coinversion} in~$\ck$. So an inversion in~$\cj$ is a coinversion in~$\ck$, and we contradict the assumption that~$\cj \leq \ck$. Since no such~$\beta$ exists, we have that~$\cj^\cx \leq \ck$.\par 
    \textbf{Case (4)}. The new inversions in~$\inv(\cj^\cx) \setminus \inv(\cj)$ are (from the above characterization) precisely the following:
    \[
    \left(a_i,\{b_p,...,b_r\}\right) \cup \left(a_i,J_p \cup \cdots \cup J_r\right) \cup \left(a_i,\{m\}\right).
    \]
    Now~$(a_i,m)$ is an inversion in~$\cx$ and so follows the logic from (1), and the remaining logic is identical to (4) with~$a_{l+1}$ replaced with~$m$.
\end{proof}
The following statement is slightly stronger (yet less precise) and works for non-cover relations in the maximal tubings of~$P_n$.
\begin{corollary}\label{cor:key_corollary}
    Let~$\cj \in \MTub(C_n)$ and~$\cx \in \MTub(P_n)$ such that~$\Cut(\cj) \leq \cx$ in~$\MTub(P_n)$. Then there exists an element~$\cj^\cx \in \Cut^{-1}(\cx)$ such that for any tubings~$\cy \geq \cx$ and~$\ck \in \Cut^{-1}(\cy)$,
      ~$\cj \leq \ck$ if and only if~$\cj^{\cx} \leq \ck$.
    This element must be unique (for a particular~$\cj$ and~$\cx$).
\end{corollary}
\begin{proof}
    For the case of equality~$\Cut(\cj) = \cx$, we may simply let~$ \cj^\cx = \cj$ and the statement is trivial.\par 
    If~$\cx$ covers~$\Cut(\cj)$ this is Lemma \ref{lem:key_lemma}. This will serve as the base case for our proof by induction. Generally, let 
    \[ \Cut(\cj) \lessdot \cy_1 \lessdot \cdots \lessdot \cy_r = \cy \lessdot \cx\, ,\]
    and assume for induction that this claim holds for all chains of length~$\leq r$. Consider the elements~$\cj^{\cy}$ and~$\left(\cj^{\cy}\right)^\cx$. For any~$\ck$ such that~$\Cut(\ck) \geq \cx$ (note this implies that~$\Cut(\ck) > \cy$ as well)  
    \[\cj \leq \ck \text{ if and only if } \cj^{\cy} \leq \ck\]
    by induction and 
    \[
    \cj^{\cy} \leq \ck \text{ if and only if } (\cj^{\cy})^\cx \leq \ck
    \]
    by Lemma \ref{lem:key_lemma}.
    So~$\cj^\cx = (\cj^{\cy})^\cx$ and we have the claim.
\end{proof}

\begin{proposition}\label{prop:path_join}
Let~$\cj,\ck \in \MTub(C_n)$. 
All maximal lower bounds for~$\{\cj,\ck\}$ are contained within the fiber 
\[\Cut^{-1}\left( \Cut(\cj) \land \Cut(\ck) \right),\]
and all minimal upper bounds for~$\{\cj,\ck\}$ are contained within the fiber 
\[\Cut^{-1}\left( \Cut(\cj) \lor \Cut(\ck) \right).\]
\end{proposition}
\begin{proof}
    We will show that all maximal lower bounds are in the fiber of the meet, and the full claim will follow by Corollary \ref{cor:involution}. \par 
    Say that~$\cl$ is less than both~$\cj$ and~$\ck$. Then by Corollary \ref{cor:cutmap_orderpreserving},~$\Cut(\cl)$ is less than both~$\Cut(\cj)$ and~$\Cut(\ck)$, and is therefore less than~$\cx = \Cut(\cj) \land \Cut(\ck)$. By Corollary \ref{cor:key_corollary} there exists an element~$\cl^\cx$ such that~$\cl < \cl^\cx \leq \cj$, and~$\cl < \cl^\cx \leq \ck$. So if~$\Cut(\cl) \neq \cx$ then~$\cl$ is not a maximal lower bound.
\end{proof}
\begin{theorem}\label{thm:lattice}
    The poset of maximal tubings of the cycle graph $\MTub(C_n)$ is a lattice-poset.
\end{theorem}
\begin{proof}
    We will show that elements have a join. Let~$\cj$ and~$\ck$ be maximal tubings of~$C_n$. \par 
    Let~$\cl$ be a minimal upper bound for~$\cj$ and~$\ck$, and let~$\cx = \Cut(\cj) \lor \Cut(\ck)$.
    By Proposition \ref{prop:path_join} we have~$\cl \in \cx$.  If $\Cut(\cl) < \cx$, then by Corollary \ref{cor:key_corollary} there exist unique elements~$\cj^\cx$ and~$\ck^\cx$ in~$\Cut^{-1}(\cx)$ such that~$\cj \leq \cl$ if and only if~$\cj^\cx \leq \cl$ and~$\ck \leq \cl$ if and only if~$\ck^\cx \leq \cl$.\par 
    By Proposition \ref{prop: fiber is an interval}, the fiber~$\Cut^{-1}(\cx)$ is a lattice. Since~$\cj^\cx,\ck^\cx$, and~$\cl$ are all elements of the fiber, the join within the fiber~$\cm \coloneqq \cj^\cx \lor \ck^\cx$ is, by definition, less than or equal to~$\cl$. So~$\cj \leq \cm \leq \cl$ and~$\ck \leq \cm \leq \cl$. Since~$\cl$ is a minimal upper bound, we have that~$\cl = \cm$. So the minimal upper bound is unique, and is thus the join of~$\cj$ and~$\ck$.\par 
    By symmetry, maximal tubings~$\cj$ and~$\ck$ of~$C_n$ also have meets, and so~$\MTub(C_n)$ is a lattice.
\end{proof}

Recall that a surjection $\phi:L\to L'$ from a lattice $L$ to $L'$ is a \emph{lattice quotient} map if $\phi$ respects the join and meet operations of $L$ (e.g. $\phi(x\join y)=\phi(x)\join\phi(y)$ and $\phi(x\meet y)=\phi(x)\meet\phi(y)$ ).
\begin{corollary}
\label{cut_lattice_map}
    The map~$\Cut \colon \MTub(C_n) \to \MTub(P_n)$ is a lattice quotient map.
\end{corollary}
\begin{proof}
    Since $\Cut$ is surjective, it is enough to show that the $\Cut$ preserves the join and meet operations.
    The join of~$\cj$ and~$\ck$ is in the fiber of the join under the $\Cut$ map, in other words~$\Cut(\cj \lor \ck) = \Cut(\cj) \lor \Cut(\ck)$. Since $\Cut$ also respects the usual symmetry for $C_n$ and $P_n$ (see Corollary \ref{cor:involution} and Proposition \ref{prop:cut_commuting}), the same is true for meets.
\end{proof}
The next Corollary now follows immediately from Corollary~2.21 in \cite{BM2021}.
\begin{corollary}
    The M\"{o}bius function on $\MTub(C_n)$ takes values in $\{\pm 1, 0\}$.
\end{corollary}
\section{Join and Meet Irreducible Elements of \texorpdfstring{$\MTub(C_n)$}{MTub(Cn)}}\label{sec:joinirr}

In this section, we characterize the structure of the~$G$-trees of the join and meet irreducible elements of~$\MTub(C_n)$. 
    The following is a slightly more specific version of  Proposition \ref{prop:cover_relation_cycle} which will be very useful for results in this section.
    \begin{remark}\label{rem:joinirred_descents}
            If~$\cj\in \MTub(C_n)$, then~$\cj$ covers exactly one element if and only if~$G_\cj$ has a unique descent. See Figure~\ref{fig:shifts} for a collection of $G$-trees that each have a unique descent.
    \end{remark}

\begin{proposition}\label{prop:descent_edge_numbers}
    If $\cj \in \MTub(C_n)$, then $G_\cj$ has at least as many descents as there are right edges in $G_\cj$.
\end{proposition}
\begin{proof}
    Let $m$ be the maximal element in $G_\cj$. Consider an element $a$ covering a right child $b$ and itself covered by $c$. Visually, we have that 
    
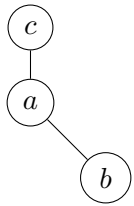
\begin{center}
        \begin{tikzpicture}
            \draw (0,0) node[draw,circle] (a) {$a$};
            \draw (1,-1) node[draw,circle] (b) {$b$};
            \draw (0,1) node[draw,circle] (c) {$c$};

            \draw (c)--(a)--(b);
        \end{tikzpicture}
    \end{center}
    The vertex $a$ may or may not have a left child, and could be a right or left child of $c$. \par 
    We first show that $(c,a)$ or $(a,b)$ are descents, then that each such $a$ may be assigned one of those descents injectively.\par

    If $a<b$ in the standard order on the integers then we have a descent. Otherwise, $b<a$ and since $b$ is a right child of $a$ we additionally get that $b < m < a$ in the standard order on the integers. Moreover, the tube $\cj_\downarrow(a)$ contains $a$ and $b$ but not $m$, and so based on the cyclic ordering it contains $1$ and $n$ (see the cyclic order below). Since $c\notin \cj_\downarrow(a)$ it must be that $b<c<a$ (i.e. $c$ lies in the same portion of the circular order pictured below as $m$, between $a$ and $b$), so $(c,a)$ is a descent. 
    
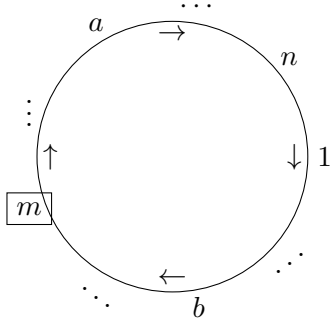
\begin{center}
\begin{tikzpicture}[scale=0.9]

    \draw (0,0) circle (2cm);
    
    \draw (0:1.8cm) node (ar1) {$\downarrow$};
    \draw (90:1.8cm) node (ar2) {$\rightarrow$};
    \draw (180:1.8cm) node (ar3) {$\uparrow$};    
    \draw (270:1.8cm) node (ar4) {$\leftarrow$};
    
    \draw (0:2.25cm) node    (v1) {$1$};
    \draw (40:2.25cm) node   (v9) {$n$};
    \draw (80:2.25cm) node   (v8) {$\cdots$};
    \draw (120:2.25cm) node  (v7) {$a$};
    \draw (160:2.25cm) node  (v6) {$\vdots$};
    \draw (200:2.25cm) node[draw]  (v5) {$m$};
    \draw (240:2.25cm) node  (v4) {$\ddots$};
    \draw (280:2.25cm) node  (v3) {$b$};
    \draw (320:2.25cm) node  (v2) {$\iddots$};

\end{tikzpicture}
\end{center}
 
    \par 
    We define an assignment $\phi$ of descents to elements with right children. Select a linear extension of $G_\cj$, then starting at the bottom to each $a$ with a right child, let $\phi(a)$ be the descent with its right child if that descent exists and is not already assigned, and the descent with its parent otherwise. The descent with its parent cannot have already been assigned, both since it is higher in the selected linear extension, and $a$ can have at most one right child. 
    
    To show $\phi$ is injective, we argue that if a descent with the right child of $a$ is already assigned by $\phi$, then there is a descent with the parent of $a$. 
    Let $a\lessdot_\cj c$ and let \[b \lessdot_\cj b_1 \lessdot_\cj \cdots \lessdot_\cj b_k \lessdot_\cj a\]
    be a saturated chain of right children so that $a < b_k < \cdots < b_1$ but $b < b_1$ in the standard order on the integers. In particular, this will assign the descent $(a,b_k)$ to $b_k$. We claim that $(c,a)$ is a descent, and so may be assigned to $a$. Since $b$ is a right child of $b_1$ and $b < b_1$, by the same logic as above, the tube $\cj_\downarrow(b_1)$ must contain $1$ and $n$. So $\cj_\downarrow(a)$ contains $b_k$,  $1$, and $n$, and moreover $b_k < m < a$. So all elements in $\{1,...,b_k\} \cup \{a,...,n\}$ are also in $\cj_\downarrow(a)$. Since $c \notin \cj_\downarrow(a)$ we know that $b_k < c < a$, and so $(c,a)$ is a descent.
\end{proof}
\begin{corollary}\label{cor:descent_branches}
    A $G$ tree has at least as many descents as there are branching points.
\end{corollary}
\begin{proof}
    A branching point has two children, so one of them is a right child.
\end{proof}

When combined with Proposition~\ref{prop:cover_relation_cycle} (see Remark~\ref{rem:joinirred_descents}), Corollary~\ref{cor:descent_branches} implies the following.
\begin{corollary}
    Join irreducible elements of~$\MTub(C_n)$ have at most one branch.
\end{corollary}

We now define a class of~$G$-trees which will turn out to be the $G$-trees of the join irreducible elements of $\MTub(C_n)$ (see Lemma \ref{lem:all_ji_gtrees}).

\begin{definition}\label{def:join_irr}
We define~$j_{i,k}$ for~$1\leq i,k\leq n-1$ to be the tree built in the following way: 

\begin{itemize}
    \item If~$k=1$ we have the tree in Figure \ref{fig:jia}.

 \item If~$1<k<n-i$, then we have the tree which takes~$j_{1,k}$ and moves the~$k-1$ elements above~$i$ to be left-children of~$i$. See  Figure \ref{fig:jib}.
 \item If~$k=n-i$, then we have the tree which takes~$j_{n-i-2,i}$ and swaps 
 locations of~$n$ and~$i$. See Figure \ref{fig:jic}.

  \item If~$n-1\geq k\geq n-i+1$, the we have the tree that takes~$j_{n-i,i}$ and moves the~$k-n+i$ right children below~$n$ above~$n$. See  Figure \ref{fig:jid}.
\end{itemize}

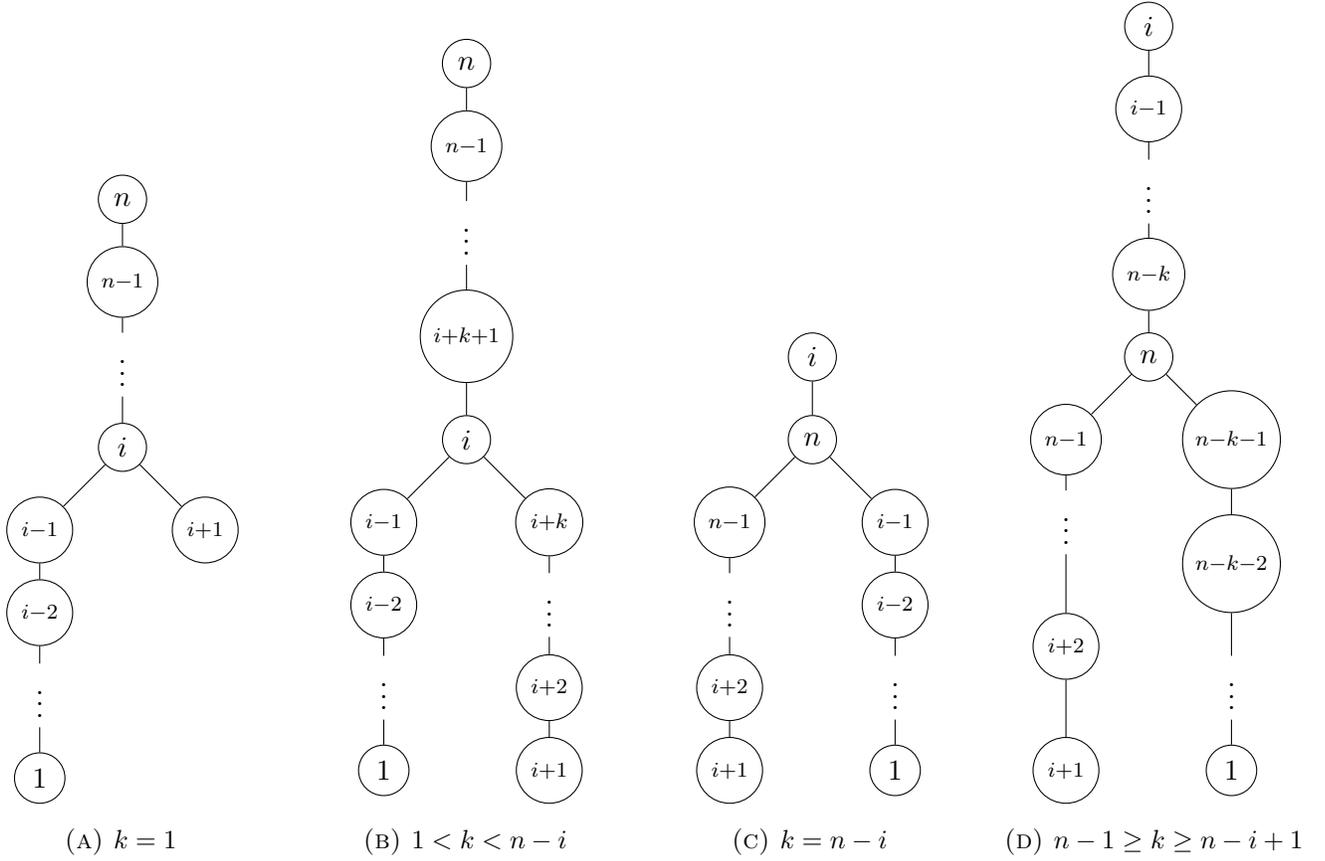
\begin{figure*}[t!]
    \centering
    \begin{subfigure}[t]{0.24\textwidth}
        \centering
    \begin{center}
    \begin{tikzpicture}[scale=1.1]
         \draw (0,3) node[draw,circle] (n)     {$n$};
         \draw (0,0) node[draw,circle] (k)     {$i$};
         \draw (0,2) node[draw, circle] (n1)     {$\scriptstyle n-1$};
         \draw (0,1) node[draw=none] (ldots)  {$\vdots$};
         \draw (1,-1) node[draw, circle] (k1)  {$\scriptstyle i+1$};
         \draw (-1,-1) node[draw, circle] (km1) {$\scriptstyle i-1$};
         \draw (-1,-2) node[draw, circle] (km2) {$\scriptstyle i-2$};
         \draw (-1,-3) node[draw=none] (rdots) {\raisebox{-1pt}{$\vdots$}};
         \draw (-1,-4) node[draw,circle] (1)     {$1$};

         \draw (n)--(n1)--(ldots)--(k)--(k1);
         \draw (k)--(km1)--(km2)--(rdots)--(1);
 \end{tikzpicture}
 \end{center} 
        \caption{$k=1$}\label{fig:jia}
    \end{subfigure}%
    ~ 
    \begin{subfigure}[t]{0.24\textwidth}
        \centering
        
    \begin{center}
    \begin{tikzpicture}[scale=1.1]
         \draw (0,4.55) node[draw,circle] (n)     {$n$};
         \draw (0,0) node[draw,circle] (k)     {$i$};
         \draw (0,3.55) node[draw, circle] (n1)     {$\scriptstyle n-1$};
         \draw (0,2.5) node[draw=none] (ldots)  {$\vdots$};
         \draw (0,1.25) node[draw, circle] (ki1)  {$\scriptstyle i+k+1$};
         \draw (1,-1) node[draw, circle] (ki)  {$\scriptstyle i+k$};
         \draw (1,-3) node[draw, circle] (k2)  {$\scriptstyle i+2$};
         \draw (1,-4) node[draw, circle] (k1)  {$\scriptstyle i+1$};
         \draw (-1,-1) node[draw, circle] (km1) {$\scriptstyle i-1$};
         \draw (-1,-2) node[draw, circle] (km2) {$\scriptstyle i-2$};
         \draw (-1,-3) node[draw=none] (rdots) {$\vdots$};
         \draw (1,-2) node[draw=none] (kdots) {$\vdots$};
         \draw (-1,-4) node[draw,circle] (1)     {$1$};

         \draw (n)--(n1)--(ldots)--(ki1)--(k)--(ki)--(kdots)--(k2)--(k1);
         \draw (k)--(km1)--(km2)--(rdots)--(1);
 \end{tikzpicture}
 \end{center}
        \caption{$1<k<n-i$}\label{fig:jib}
    \end{subfigure}%
    ~ 
    \begin{subfigure}[t]{0.24\textwidth}
        \centering
    \begin{center}
    \begin{tikzpicture}[scale=1.1]
         \draw (0,1) node[draw,circle] (n)     {$i$};
         \draw (0,0) node[draw,circle] (k)     {$n$};
         \draw (-1,-1) node[draw, circle] (ki)  {$\scriptstyle n-1$};
         \draw (-1,-3) node[draw, circle] (k2)  {$\scriptstyle i+2$};
         \draw (-1,-4) node[draw, circle] (k1)  {$\scriptstyle i+1$};
         \draw (1,-1) node[draw, circle] (km1) {$\scriptstyle i-1$};
         \draw (1,-2) node[draw, circle] (km2) {$\scriptstyle i-2$};
         \draw (1,-3) node[draw=none] (rdots) {$\vdots$};
         \draw (-1,-2) node[draw=none] (kdots) {$\vdots$};
         \draw (1,-4) node[draw,circle] (1)     {$1$};

         \draw (n)--(k)--(ki)--(kdots)--(k2)--(k1);
         \draw (k)--(km1)--(km2)--(rdots)--(1);
 \end{tikzpicture}
 \end{center} 
        \caption{$k=n-i$}\label{fig:jic}
    \end{subfigure}%
    ~ 
    \begin{subfigure}[t]{0.24\textwidth}
        \centering
        
    \begin{center}
    \begin{tikzpicture}[scale=1.1]
         \draw (0,4) node[draw,circle] (n)     {$i$};
         \draw (0,0) node[draw,circle] (k)     {$n$};
         \draw (0,3) node[draw, circle] (n1)     {$\scriptstyle i-1$};
         \draw (0,2) node[draw=none] (ldots)  {$\vdots$};
         \draw (0,1) node[draw, circle] (ki1)  {$\scriptstyle n-k$};
         \draw (-1,-1) node[draw, circle] (ki)  {$\scriptstyle n-1$};
         \draw (-1,-3.5) node[draw, circle] (k2)  {$\scriptstyle i+2$};
         \draw (-1,-5) node[draw, circle] (k1)  {$\scriptstyle i+1$};
         \draw (1,-1) node[draw, circle] (km1) {$\scriptstyle n-k-1$};
         \draw (1,-2.5) node[draw, circle] (km2) {$\scriptstyle n-k-2$};
         \draw (1,-4) node[draw=none] (rdots) {$\vdots$};
         \draw (-1,-2) node[draw=none] (kdots) {$\vdots$};
         \draw (1,-5) node[draw,circle] (1)     {$1$};

         \draw (n)--(n1)--(ldots)--(ki1)--(k)--(ki)--(kdots)--(k2)--(k1);
         \draw (k)--(km1)--(km2)--(rdots)--(1);
 \end{tikzpicture}
 \end{center} 
        \caption{$n-1\geq k\geq n-i+1$}\label{fig:jid}
    \end{subfigure}
    
    \caption{Diagrams describing~$j_{i,k}$ for various values. 
    }\label{fig:jik_diagrams}

\end{figure*}

\end{definition}

\begin{remark}
    The trees $j_{i,k}$ are constructed to be cyclic binary trees that have only one descent.  There are only a few possibilities, since such a tree can have at most one branch.  Trees where $n$ is the root correspond to $G$-trees for $P_{n-1}$, where the trees with one descent are well understood. In this case the branch occurs at $i$ and the descent will always be the pair consisting of $i$ and its right child.
    
    When the root is not $n$, we know that if the tree has a branch it will occur at $n$.  The descent will always be the pair consisting of $n$ and its parent.
\end{remark}

Figure \ref{fig:shifts} shows some specific examples. As one can verify in Figure \ref{fig:jid}, the trees $j_{i,k}$ appear to be $G$-trees satisfying the above necessary conditions for join irreducible elements for the corresponding tubing. In fact, not only is this always true, but this generates all of them.

\begin{lemma}\label{lem:all_ji_gtrees}

    The $G$-trees for the join irreducible elements in $\MTub(C_n)$ are given by the trees $j_{i,k}$ for all $1\leq i,k\leq n-1$. 
\label{lem:ji_cycles}
\end{lemma}

\begin{proof}
We first note each $j_{i,k}$ is the $G$-tree of maximal tubing of $C_n$ because in the partial order on $[n]$ induced by $j_{i,k}$ each principal order ideal is a cyclic interval and thus a connected subgraph of $C_n$.
Now let $\cj \in \MTub(C_n)$ be a join irreducible element and consider the types (A--D) of $j_{i,k}$ in Figure \ref{fig:jik_diagrams}.  
Suppose the root of $G_\cj$ is $n$. The number of descents in a binary search tree is exactly the number of right edges. The binary search trees on $[n-1]$ with a single right edge---with $n$ added to the top---take the form of (A) or (B). So $G_\cj = j_{i,k}$ for some $1 \leq k < n-i$. \par 
Now suppose the root of $G_\cj$ is not $n$. By Proposition \ref{prop:descent_edge_numbers}, all cyclic binary trees with at one descent are cyclic shifts of those $j_{i,k}$ where $1 \leq k < n-i$. However, many of these cyclic shifts will have more than one descent. Since any such shift will have a descent with $n$, that must be the descent that occurs at the right edge as in the proof of Proposition \ref{prop:descent_edge_numbers}. So the cyclic binary trees with precisely one descent, whose maximal element is \emph{not} $n$, take the form of (C) or (D). Thus $G_\cj = j_{i,k}$ where $n-i \leq k \leq n-1$.
 Consequently, all such trees are precisely described by $j_{i,k}$.
\end{proof}

For the remainder of this paper we will denote by $\cj_{i,k}$ the join irreducible element  $\cj\in\MTub(C_n)$ whose $G$-tree is $G_{\cj} =j_{i,k}$.

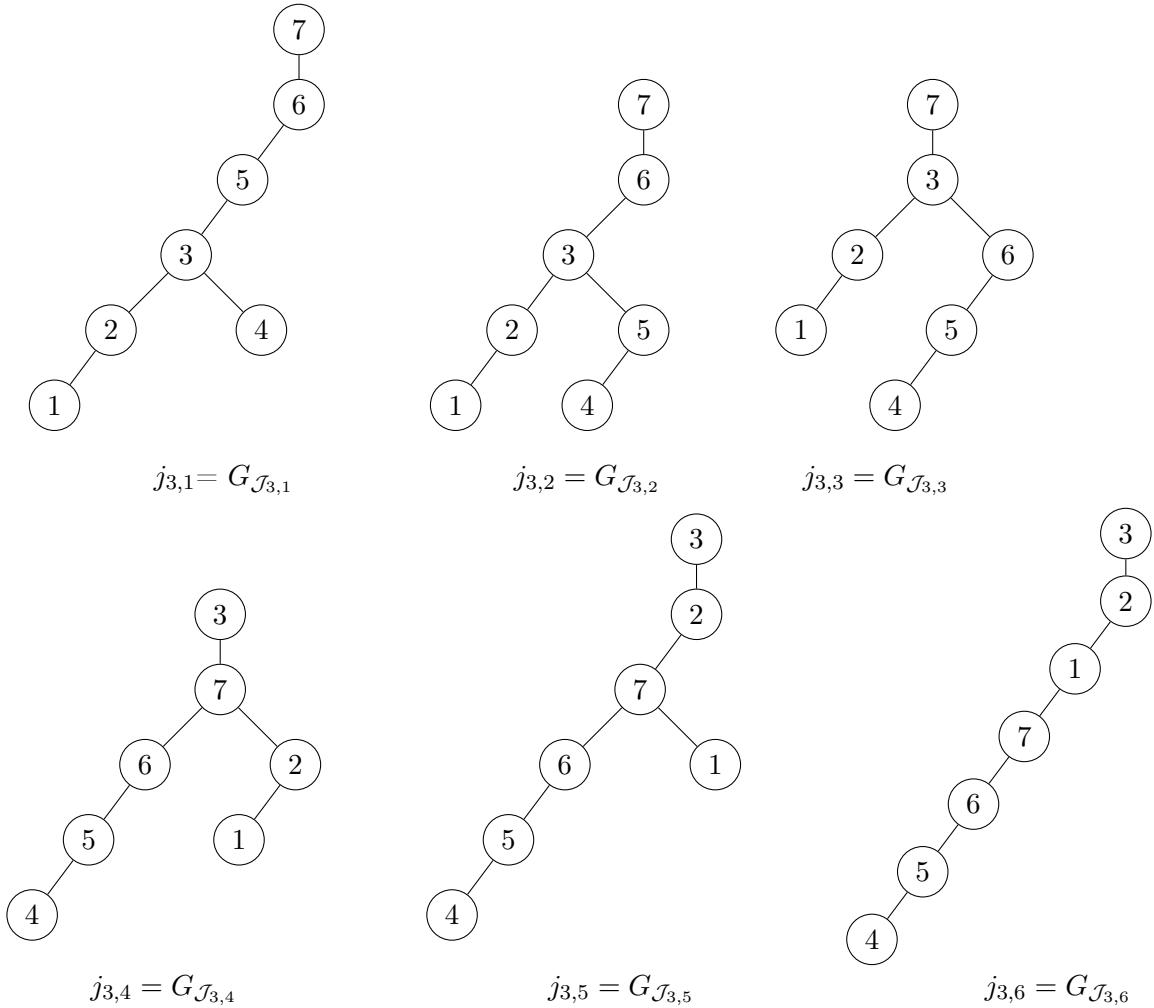
\begin{figure}[ht]
\begin{center}
    \begin{tikzpicture}
         \draw (0,3) node[draw,circle] (a)     {$7$};
         \draw (0,2) node[draw, circle] (c)     {$6$};
         \draw (-0.75,1) node[draw, circle] (d)     {$5$};
         \draw (-1.5,0) node[draw,circle] (b)     {$3$};
         \draw (-0.5,-1) node[draw, circle] (e)  {$4$};
         \draw (-2.5,-1) node[draw, circle] (f) {$2$};
         \draw (-3.25,-2) node[draw, circle] (g) {$1$};
         \draw (-1,-3) node[draw=none] (T)     {$j_{3,1}$= $G_{\cj_{3,1}}$};

         \draw (a)--(c)--(d)--(b)--(e);
         \draw (b)--(f)--(g);
 \end{tikzpicture}\hspace{0.5in}    
 \begin{tikzpicture}
         \draw (0,1) node[draw,circle] (a)     {$7$};
         \draw (0,0) node[draw,circle] (b)     {$6$};
         \draw (-1,-1) node[draw, circle] (c)     {$3$};
         \draw (-1.75,-2) node[draw, circle] (d)     {$2$};
         \draw (-2.5,-3) node[draw, circle] (e)  {$1$};
         \draw (0,-2) node[draw, circle] (f) {$5$};
         \draw (-0.75,-3) node[draw, circle] (g) {$4$};
         \draw (-0.75,-4) node[draw=none] (T)     {$j_{3,2}=G_{\cj_{3,2}}$};

         \draw (a)--(b)--(c)--(d)--(e);
         \draw (c)--(f)--(g);
    \end{tikzpicture}\hspace{0.5in}
    \begin{tikzpicture}
         \draw (0,1) node[draw,circle] (a)     {$7$};
         \draw (0,0) node[draw,circle] (b)     {$3$};
         \draw (-1,-1) node[draw, circle] (c)     {$2$};
         \draw (-1.75,-2) node[draw, circle] (d)     {$1$};
         \draw (-0.5,-3) node[draw, circle] (e)  {$4$};
         \draw (1,-1) node[draw, circle] (f) {$6$};
         \draw (0.25,-2) node[draw, circle] (g) {$5$};
         \draw (-0.75,-4) node[draw=none] (T)     {$j_{3,3}=G_{\cj_{3,3}}$};

         \draw (a)--(b)--(c)--(d);
         \draw (b)--(f)--(g)--(e);
    \end{tikzpicture}\hspace{0.5in}
    
    \begin{tikzpicture}
         \draw (0,1) node[draw,circle] (a)     {$3$};
         \draw (0,0) node[draw,circle] (b)     {$7$};
         \draw (-1,-1) node[draw, circle] (c)     {$6$};
         \draw (-1.75,-2) node[draw, circle] (d)     {$5$};
         \draw (-2.5,-3) node[draw, circle] (e)  {$4$};
         \draw (1,-1) node[draw, circle] (f) {$2$};
         \draw (0.25,-2) node[draw, circle] (g) {$1$};
         \draw (-0.75,-4) node[draw=none] (T)     {$j_{3,4}=G_{\cj_{3,4}}$};

         \draw (a)--(b)--(c)--(d)--(e);
         \draw (b)--(f)--(g);
    \end{tikzpicture}\hspace{0.5in}
    \begin{tikzpicture}
         \draw (0,2) node[draw,circle] (a)     {$3$};
         \draw (0,1) node[draw, circle] (f) {$2$};
         \draw (-0.75,0) node[draw,circle] (b)     {$7$};
         \draw (-1.75,-1) node[draw, circle] (c)     {$6$};
         \draw (-2.5,-2) node[draw, circle] (d)     {$5$};
         \draw (-3.25,-3) node[draw, circle] (e)  {$4$};
         \draw (0.25,-1) node[draw, circle] (g) {$1$};
         \draw (-1,-4) node[draw=none] (T)     {$j_{3,5}=G_{\cj_{3,5}}$};

         \draw (a)--(f)--(b)--(c)--(d)--(e);
         \draw (b)--(g);
 \end{tikzpicture}\hspace{0.5in}
 \begin{tikzpicture}[scale=0.9]
         \draw (0,2) node[draw,circle] (a)     {$3$};
         \draw (0,1) node[draw, circle] (f) {$2$};
         \draw (-0.75,0) node[draw,circle] (b)     {$1$};
         \draw (-1.5,-1) node[draw, circle] (c)     {$7$};
         \draw (-2.25,-2) node[draw, circle] (d)     {$6$};
         \draw (-3,-3) node[draw, circle] (e)  {$5$};
         \draw (-3.75,-4) node[draw, circle] (g) {$4$};
         \draw (-1,-4.75) node[draw=none] (T)     {$j_{3,6}=G_{\cj_{3,6}}$};

         \draw (a)--(f)--(b)--(c)--(d)--(e)--(g);
 \end{tikzpicture}
\end{center}
    \caption{The $G$-trees of a chain of join irreducibles in $\MTub(C_7)$: $\cj_{3,s}\le\cj_{3,t}$ provided $s\le t$. In fact this chain is saturated in $\MTub(C_7)$, see Proposition~\ref{prop:lrshifts}. In particular, referring to Definition~\ref{def:join_irr}, the tree $j_{3,1}$ is an example of case $(A)$, $j_{3,2},j_{3,3}$ are case $(B)$, $j_{3,4}$ is case $(C)$, and $j_{3,5},j_{3,6}$ are case $(D)$.}
    \label{fig:shifts}
\end{figure}

\begin{corollary}\label{cor:jienum}
    There are $(n-1)^2$ join irreducible elements of $\MTub(C_n)$.
\end{corollary}

\section{Semidistributive and Congruence Uniform}\label{sec:semidis}

Using the results from prior sections, the primary goal of this section is to demonstrate the following:

\begin{theorem}\label{thm:semidistributive}
   $\MTub(C_n)$ is a semidistributive lattice.
\end{theorem} 

We now review the basics of semidistributive lattices. Recall that if $L$ is a finite lattice then $Ji(L)$ and $Mi(L)$ are the sets of join and meet irreducible elements and if $j\in Ji(L)$ then $j_*$ is the unique element of $L$ such that $j_*\lessdot j$. 

A lattice $L$ is said to be \emph{meet-semidistributive} if for $x,y,z\in L$ such that if $x\meet y = x\meet z$ implies that $x\meet (y\join z) = (x\meet y)\join (x\meet z) = x\meet y$. Dually, a lattice $L$ is said to be \emph{join-semidistributive} if $x\join y = x\join z$ implies $x\join (y\meet z) = (x\join y)\meet (x\join z) = x\join y$. We say that $L$ is \emph{semidistributive} if $L$ is both meet-semidistributive and join-semidistributive. 

We further will  show that $\MTub(C_n)$ is in fact a \emph{congruence uniform} lattice by way of the characterization finite congruence uniform lattices of Reading, Speyer, and Thomas \cite{reading2021fundamental}.

We start by developing a deeper understanding of the structure of join irreducible elements in $\MTub(C_n)$, which will be critical in demonstrating Theorem \ref{thm:semidistributive}.

\begin{proposition}\label{prop:lrshifts}
     The collection of join irreducibles $\cj_{i,1},\cj_{i,2},\dots, \cj_{i,n-1}$ forms a saturated chain in $\MTub(C_n)$. 
\end{proposition}
\begin{proof}
We will demonstrate that the $G$-trees $j_{i,k}$ and $j_{i,k+1}$ are related by a a single $G$-tree move, thus implying a cover relation $\cj_{i,k}\lessdot\cj_{i,k+1}$.
    We proceed by three cases depending on the value of $k$. 

    \begin{enumerate}
        \item Suppose that $1\leq k<n-i-1$. Note that in this case, one can witness $j_{i,k+1}$ as starting with $j_{i,k}$ and moving $i+k+1$ so that $i+k\lessdot_{\cj_{i,k+1}} i+k+1\lessdot_{\cj_{i,k+1}} i\lessdot_{\cj_{i,k+1}} i+k+2$, that is, $i+k+1$ becomes a left-child of $i$ while otherwise all other relations are kept the same from the partial order of $j_{i,k}$.  Note in this case that in the schematic for $G$-tree moves in Figure~\ref{fig:tree_move}, $T_1$ is empty.
        \item If $k=n-i-1$, then one can witness $j_{i,n-i}$ from $j_{i,n-i-1}$ by switching the locations of $n$ and $i$. Note in this case that in the schematic for $G$-tree moves in Figure~\ref{fig:tree_move}, $T_4$ and $T_3$ are empty.
        \item If $k\geq n-i$, then one can witness $j_{i,k+1}$ as $j_{i,k}$ where $n-k-1$  is moved so that $n-k-2\lessdot_{\cj_{i,k+1}} n\lessdot_{\cj_{i,k+1}} n-k-1\lessdot_{\cj_{i,k+1}} n-k$, that is, $n-k-1$ covers $n$ while otherwise all other relations are kept the same from $j_{i,k}$.  Note in this case that in the schematic for $G$-tree moves in Figure~\ref{fig:tree_move}, $T_2$ is empty.\qedhere
    \end{enumerate}
\end{proof}

Recall that an \emph{atom} of a finite poset containing a minimum element is an element which covers the minimal element. Similarly a \emph{coatom} is an element of a finite poset containing a maximum element which is covered by the maximum element. Recall that in a lattice $L$ if $j$ is a join irreducible element, then $j_*$ is the unique element covered by $j$. Consequently for us if $\cj$ is a join irreducible element of $\MTub(C)_n$ we denote by $\cj_*$ the unique element such that $\cj_*\lessdot \cj$.
\begin{proposition}\label{prop:subposetjoins}
    In $\MTub(C_n)$, the subposet of join irreducible elements is $(n-1)$ disjoint chains of $(n-1)$ elements.
\end{proposition}
\begin{proof}
    Since we can partition the set of join irreducibles of $\MTub(C_n)$ into saturated chains of join irreducibles each of which contains an atom, it will suffice to show that for each join irreducible element $\cj$ there is a unique atom $\mathcal{A} \le \cj$. This is because if we have an order relation between elements of two disjoint saturated chains of join irreducibles each containing an atom, this relation would imply there is some element in one of the chains that is larger than multiple distinct atoms. 

    So suppose that $\mathcal{B}$ is an atom and $\mathcal{B}\le \cj$. We will show that $\mathcal{A} = \mathcal{B}$. Assume that $\cj$ is not an atom, as if so the result follows immediately.
   By Proposition~\ref{prop:lrshifts} there is a saturated chain $C$ of join irreducibles from $\mathcal{A}$ to $\cj$. Since $\mathcal{A}\le \cj, \mathcal{B}\le \cj$ and $\cj$ is join irreducible, then $\mathcal{A},\mathcal{B}\le \cj_*$. Otherwise, $\cj$ would not be join irreducible. Note that $\cj_*$ must also be a join irreducible element in $C$. Following the same logic, since $\cj_*\in C$ because $\cj$ is not an atom, then inductively it follows that $\mathcal{B}\le \ck$ for all $\ck\in C$ with $\ck\le \cj$. This implies that $\mathcal{B}\le \mathcal{A}$ implying $\mathcal{B}= \mathcal{A}$ since both $\mathcal{A}$ and $\mathcal{B}$ are atoms.  The result then follows by Corollary~\ref{cor:jienum} and Proposition~\ref{prop:lrshifts}.
\end{proof}

\begin{corollary} \label{cor:atoms} An element $\mathcal{A}$ is an atom in $\MTub(C_n)$ if and only if $\mathcal{A}=\cj_{i,1}$ for some $i$. 
\end{corollary}

\begin{lemma}\label{lem:invs_jis}
   For the join irreducible $\cj_{i,k}\in\MTub(C_n)$,
   \begin{align*}
   \inv(\cj_{i,k}) &= \begin{cases}
       \{(i,i+1), \dots, (i,i+k)\} & \text{if }i\le n-k\, ,\\
       \bigcup_{\ell = 0}^{i-(n-k)}\{(i-\ell,i+1), \dots, (i-\ell,n)\} & \text{if }i> n-k \, ,  \end{cases} \\
       &= \begin{cases}
       \{i\} \times \{i+1, \dots, i+k\} & \text{if }i\le n-k\, ,\\
       \{n-k,\ldots ,i  \} \times \{i+1,\ldots,n\} & \text{if }i> n-k \, .  \end{cases}
   \end{align*}
\end{lemma}
\begin{proof}
    This follows from Definition \ref{def:join_irr}.
\end{proof}

Recall that $w_0$ is the permutation that swaps $i$ and $n-i+1$.

\begin{lemma}\label{cor:coinvs_mis} 
    Let $\cm$ be the meet irreducible of $\MTub(C_n)$ obtained by applying the anti-isomorphism $w_0$ of Corollary~\ref{cor:involution} to the join irreducible $\cj$. Then $\coinv(\cm) = w_0(\inv (\cj)),$ 
    where $w_0$ applied to a set of pairs is applying $w_0$ to each term of each element.
\end{lemma}
\begin{proof}
    This follows from Lemma~\ref{lem:invs_jis} because $G_\cm$ is obtained by replacing each $i$ in $G_\cj$ with $n-i+1$. Consequently the only comparable non-inversion pairs of $G_\cm$ are the pairs $(n-k+1,n-i+1)$ where $(i,k)$ was an inversion of $G_\cj$.
\end{proof}

    Denote by $j_i$ the chain of join irreducibles whose inversions all contain $(i,i+1)$ and by $m_i$ the chain of meet irreducible elements of $w_0 j_i$, that is the meet irreducibles whose set of coinversions all do not contain $(n-i,n-i+1)$. 
    Similarly let $\cm_{i,k}$ be the meet irreducible element of height $k$ in the chain $m_i$ as a subposet of $\MTub(C_n)$. Note that $w_0 \cj_{i,k} = \cm_{i,n-k}$.
    \begin{corollary}\label{cor:coinv_mik}
        For the meet irreducible $\cm_{i,k}\in \MTub(C_n)$,
        \[
        \coinv(\cm_{i,k}) = \left(\begin{cases}
       \{k-i+1,\ldots,n-i\} \times \{n-i+1\} & \text{if }i\le k\\
        [n-i] \times \{n-i+1,\ldots n-k+1\} & \text{if }i > k\end{cases} \right)\, .
        \]
    \end{corollary}
    \begin{proof}
        This follows directly from Lemmas \ref{lem:invs_jis} and~\ref{cor:coinvs_mis}:
        \begin{align*}
        \coinv(\cm_{i,k}) &= w_0 \inv(\cj_{i,n-k})\\
        &=w_0 \left(\begin{cases}
       \{i\} \times \{i+1, \dots, i+(n-k)\} & \text{if }i\le n-(n-k)\\
       \{n-(n-k),\ldots ,i  \} \times \{i+1,\ldots,n\} & \text{if }i> n-(n-k)\end{cases} \right) \\
       &= \begin{cases}
       \{k-i+1,\ldots,n-i\} \times \{n-i+1\} & \text{if }i\le k\, ,\\
        [n-i] \times \{n-i+1,\ldots n-k+1\} & \text{if }i > k\, .\end{cases}
    \end{align*}
    \end{proof}

    We now introduce some more notation. For each $i\in [n-1]$ let $c_i$ be the permutation on $[n-1]$ given by
    \begin{equation*}
        c_i(k)=\left\{\begin{array}{ll}
            n-i+1-k  &  1\le k\le n-i\, ,\\
            k & k>n-i\, . 
        \end{array}
        \right.
    \end{equation*}

\begin{proposition}\label{prop:join_chains_relation_to_meets}
     For all $\ell > k$, it follows that
     $\cm_{c_i(\ell),1} > \cj_{i,k}$. Moreover, $\cj_{i,k}$ is less than every maximal tubing in $m_{c_i(\ell)}$.
\end{proposition}
\begin{proof} 
    We will leverage the characterization of order from Theorem \ref{thm:inversion_order}, and in particular prove the claim by show that $\inv(\cj_{i,k}) \cap \coinv(\cm_{c_i(\ell),1}) = \emptyset$. By Corollary \ref{cor:coinv_mik}, 
        \[\coinv(\cm_{c_i(\ell),1}) = [n-c_i(\ell)] \times \{n-c_i(\ell)+1,\ldots,n\}\, . \]
    The argument breaks into two cases based on the value of $k$. In each, we show that an arbitrary co-inversion $(p,q) \in \coinv(\coinv(\cm_{c_i(\ell),1})$ cannot be an inversion in $\inv(\cj_{i,k})$.
    
    (Case 1: $k\le n-i$) By Lemma~\ref{lem:invs_jis}, \[\inv(\cj_{i,k})=\{(i,i+1), (i,i+1), \dots, (i, i+k)\}\, .\] 
    The argument further breaks into two subcases based on the value of $\ell$.
    
    (Case 1A: $\ell \le n-i$) So $c_i(\ell) = n-i+1-\ell$, and $n-c_i(\ell) = \ell+i-1$.
    Since $\ell > k$, every coinversion $(p,q)$ in $\coinv(\cm_{c_i(\ell),1})$ has the larger term $q \geq i+\ell >  i+k$. So $(p,q) \notin \inv(\cj_{i,j})$, and the coinversions $\coinv(\cm_{c_i(\ell),1})$ are disjoint from the inversions $\inv(\cj_{i,k})$.
    
    (Case 1B: $\ell > n-i$) So $c_i(\ell) = \ell$, and $i > n-\ell$. Since the maximum first term of a coinversion $(p,q)$ in $\coinv(\cm_{\ell,1})$ is $p \leq n-\ell$, it must be that $p \neq i$.  Consequently, the coinversions $\coinv(\cm_{c_i(\ell),1})$ are disjoint from the inversions $\inv(\cj_{i,k})$.

(Case 2: $k>n-i$) So $\ell > n-i$ and $c_i(\ell) = \ell$. Again by Lemma~\ref{lem:invs_jis},  \[\inv(\cj_{i,k})=\bigcup_{r = 0}^{k+i-n}\{(i-r,i+1), \dots, (i-r,n)\} = \{n-k,\ldots ,i  \} \times \{i+1,\ldots,n\}\, .\] 

The largest that $p$ can be in a coinversion $(p,q) \in \coinv(\cm_{\ell,1})$ is $n-\ell$. On the other hand, the smallest first term in an element of $\inv(\cj_{i,k})$ is $n-k > n-\ell$. So the coinversions $\coinv(\cm_{c_i(\ell),1})$ are disjoint from the inversions $\inv(\cj_{i,k})$.
The ``moreover" part follows from $\cm_{c_i(\ell),1}$ being the minimal element in the chain~$m_{c_i(\ell)}$.
\end{proof}
\begin{proposition}\label{prop:max_not_jstar}
For $\ell \le k$: 
\begin{itemize}
    \item if $\ell \le n-i$, then $\cm_{c_i(\ell),n-\ell}$ is the maximal element of $m_{c_i(\ell)}$ not greater than $\cj_{i,k}$,
    \item if $\ell > n-i$, then $\cm_{c_i(\ell),n-i}$ is the maximal element of $m_{c_i(\ell)}$ not greater than $\cj_{i,k}$.
\end{itemize}
\end{proposition}

\begin{proof}
  By Corollary \ref{cor:coinv_mik},
 \[
 \coinv(\cm_{c_i(\ell),n-s}) = 
 \begin{cases}
       \{n-s-c_i(\ell)+1,\ldots,n-c_i(\ell)\} \times \{n-c_i(\ell)+1\} & \text{if }c_i(\ell)\le n-s\, ,\\
        [n-c_i(\ell)] \times \{n-c_i(\ell)+1,\ldots s+1\} & \text{if }c_i(\ell) > n-s\, .\end{cases}
 \]
The argument breaks into two cases, based on the value of $\ell$:
 
(Case 1: $\ell \le n-i$) 
So $c_i(\ell) = n-i+1-\ell$, and thus $n-c_i(\ell) = \ell+i-1$. Additionally, set $s = \ell$ above to get
 \[
 \coinv(\cm_{c_i(\ell),n-\ell}) = 
       \{i,\ldots,\ell+i-1\} \times \{\ell+i\}\, .
 \] since $n-c_i(\ell)+1-\ell= i+\ell-1+1-\ell = i$.
By Lemma \ref{lem:invs_jis}, $(i,i+\ell)$ is always an inversion $\inv(\cj_{i,k})$.
 Since $(i,i+\ell)$ is an inversion of $\cj_{i,k}$ then $\cm_{c_i(\ell),n-\ell}$ is not greater than $\cj_{i,k}$ by Theorem~\ref{thm:inversion_order}. On the other hand, ${\cm_{c_i(\ell),n-\ell+1}>\cj_{i,k}}$ again by Theorem~\ref{thm:inversion_order} because if we evaluate $s=\ell-1$ we have that 
 \[\coinv(\cm_{c_i(\ell),n-\ell+1})=\{(i+1,i+\ell),(i+2,i+\ell),\dots,(i+\ell-1,i+\ell)\}\] which is disjoint from $\inv(\cj_{i,k})$.

 (Case 2: $\ell > n-i$) So $c_i(\ell) = \ell$, and we set $s=i$ above to get 
 \[
 \coinv(\cm_{c_i(\ell),n-i})=[n-\ell]\times \{n-\ell+1, \dots, i+1\}.
 \]

Since $k \geq \ell$, this also means that $k > n-i$ (i.e. $i > n-k$), so 
\[\inv(\cj_{i,k}) = \{n-k,\ldots,i\} \times \{i+1,\ldots,n\}.\] In particular, $(n-\ell,i+1) \in \coinv(M_{c_i(\ell),n-s}) \cap \inv(J_{i,k})$ and so $\cm_{c_i(\ell),n-i} \not > \cj_{i,k}$.

On the other hand $\coinv(\cm_{c_i(\ell), n-i+1})=[n-\ell]\times \{n-\ell+1, \dots, i\}$ again by Lemma~\ref{cor:coinvs_mis}. Since every inversion of $\cj_{i,k}$ has the value of the second coordinate at least $i+1$ then no coinversion of $\cm_{c_i(\ell), n-i+1}$ is an inversion of $\cj_{i,k}$ implying that $\cm_{c_i(\ell), n-i+1} > \cj_{i,k}$ by Theorem~\ref{thm:inversion_order}. 
\end{proof}

\begin{lemma}\label{lem:max_meet}
    For $\cj_{i,k}$ there is a maximum meet irreducible element $\cm\in \MTub(C_n)$ such that $\cj_{i,k} \not < \cm$ and $\cj_{i,k}{}_* \le \cm$. 
\end{lemma}
\begin{proof}
    By Proposition~\ref{prop:join_chains_relation_to_meets} we know that the order relations between both of $\cj_{i,k}$ and $\cj_{i,k}{}_*$ and elements of $m_{c_i(\ell)}$ agree for all elements of $m_{c_i(\ell)}$ for $\ell > k$. Similarly by Proposition~\ref{prop:max_not_jstar} the order relations between elements of $m_{c_i(\ell)}$ and $\cj_{i,k}$ and $\cj_{i,k}{}_*$ agree for $\ell \le k-1$. Consequently the only chain $m_s$ of meet irreducible elements where the order relations between elements of that chain and $\cj_{i,k}$ and $\cj_{i,k}{}_*$ disagree is the only chain that can contain an element $\cm$ so that $\cm \ge \cj_{i,k}{}_*$ and $\cm\not > \cj_{i,k}$. By the preceding argument, this chain is $m_{c_i(k)}$. Additionally, there are elements of $m_{c_i(k)}$ that are not greater $\cj_{i,k}$ while all elements of $m_{c_i(k)}$ are greater than $\cj_{i,k}{}_* = \cj_{i,k-1}$ by Proposition~\ref{prop:join_chains_relation_to_meets}  and Proposition~\ref{prop:max_not_jstar}. The claim follows as there must be a maximal such meet irreducible in $m_{c_i(k)}$ as $m_{c_i(k)}$ is a saturated chain. 
\end{proof}
\begin{lemma}\label{lem:max_meet_is_max}
    The set $\{\cl \ | \ \cl\wedge \cj_{i,k} = \cj_{i,k}{}_*\}$ has a maximum element. Moreover, the maximal element is the meet irreducible element of Proposition~\ref{prop:max_not_jstar}. 
    
\end{lemma}
\begin{proof}
    Let $\cm\in\MTub(C_n)$ be the maximum meet irreducible element so that $\cm > \cj_{i,k}{}_*$ and $\cm\not > \cj_{i,k}$ as guaranteed by Lemma~\ref{lem:max_meet}. We claim that $\cm$ is maximal among those elements of $\MTub(C_n)$ that are greater than $\cj_{i,k}{}_*$ and not $\cj_{i,k}$. There are two cases, either $k=1$ or $k>1$. If $k=1$, then $\cm$ is a coatom.    Since each pair $(i,i+1)$ is the unique inversion of an atom for all $1\le i \le n-1$, there is a coatom whose only coinversion is $(i,i+1)$ which is maximal among elements not larger than $\cj_{i,1}$. The only element larger than $\cm$ is $\hat{1}$ so $\cm$ is maximal. Now suppose that $k>1$ and let $\ck$ be a maximal element of $\{\cl \ |  \ \cl\wedge \cj_{i,k} = \cj_{i,k}{}_*\}$. 
    
    We will prove that $\ck = \cm$. Recall that every element $p$ of a finite lattice is the meet of the meet irreducible elements larger than or equal to $p$. 
    Since $\cj_{i,k}{}_*= \cj_{i,k-1}$ is also a join irreducible element, then if $\ck$ is any maximum element of the set of elements of $\MTub(C_n)$ whose meet with $\cj_{i,k}$ is $\cj_{i,k}{}_*$, then $\ck$ must be less than a meet irreducible element which is larger than $\cj_{i,k}{}_*$ and not less than $\cj_{i,k}$. This is because the meet of any two elements of a lattice is an element of the lattice whose set of join irreducible elements below is the intersection of the sets of join irreducible elements~\cite[Theorem 4]{markowsky1975factorization}. This maximal element must be less than or equal to $\cm$, so $\cm$ is the unique maximum element. Since $\{\cl 
    \ | \ \cl\wedge \cj_{i,k} = \cj_{i,k}{}_*\}$ is precisely the subset of $\MTub(C_n)$ that are greater than $\cj_{i,k}{}_*$ and not $\cj_{i,k}$ we are done. 
\end{proof}

\begin{lemma}\label{lem:kappa_bij}
    The map of $\kappa: Ji(\MTub(C_n))\to Mi(\MTub(C_n))$ of $\kappa(\cj_{i,k}) = \max\{\cl \ | \ \cl\wedge \cj_{i,k} = \cj_{i,k}{}_*\}$ is bijective.  
\end{lemma}
\begin{proof}
   By Propositions ~\ref{prop:join_chains_relation_to_meets} and ~\ref{prop:max_not_jstar}, and Lemma \ref{lem:max_meet_is_max}, 
   \[
       \max\{\cl : \cl\wedge \cj_{i,k} = \cj_{i,k}{}_*\} = \begin{cases}\cm_{c_{i(k), n-k}} &\text{if } k \le n-i\text{, and }\\ \cm_{c_i(k), n-i} &\text{if }k > n-i.
       \end{cases}
   \]
    As a mapping on $[n-1]\times [n-1]$, $\kappa$ can be equivalently encoded by the function $\kappa(\cj_{i,k})=\cm_{f(i,k)}$ where $f(i,k) = (n+1-i-k,n-k)$ if $i+k\le n$ and $(k, n-i)$ otherwise. In particular $f$ is clearly surjective and thus a bijection, so $\kappa$ is as well.
\end{proof}

We now have laid out enough to prove the main result of this section: that $\MTub(C_n)$ is a semidistributive lattice. 
Let $L$ be a finite lattice. 
For each join irreducible $j\in L$, we define $\kappa(j)$ to the unique maximal element of the set $\{a \ | \ a\ge j_*, a\not\ge j\}$, when such an element exists. 

\begin{proposition}\cite[Theorem 2.56]{freese1995free}\label{prop:kappa_meet_semidistributive}
    A finite lattice $L$ is meet-semidistributive if and only if $\kappa(j)$ exists for all join irreducible elements $j$.
\end{proposition}

In the case where $L$ is a finite semidistributive lattice, then $\kappa(j)$ exists  and is  unique for all join irreducible elements $j$.  Viewed as a map, $\kappa$ is a bijection between join irreducibles and meet irreducibles of $L$.

We now prove that $\MTub(C_n)$ is a semidistributive lattice, our first main result of this section.
\begin{proof}[Proof of Theorem~\ref{thm:semidistributive}]
     The map 
     \begin{align*}
         \kappa: Ji(\MTub(C_n))&\to Mi(\MTub(C_n))\\
         \cj_{i,k}&\mapsto  \max\{\cl \ | \  \cl\wedge \cj_{i,k} = \cj_{i,k}{}_*\}
     \end{align*}
      is well defined by Lemma~\ref{lem:max_meet_is_max}. 
     So then by Proposition~\ref{prop:kappa_meet_semidistributive} $\MTub(C_n)$ is meet-semidistributive. Since $\MTub(C_n)$ is self dual by Corollary \ref{cor:involution} then $\MTub(C_n)$ is also join-semidistributive and thus semidistributive.
\end{proof}

Now that we know that $\MTub(C_n)$ is a semidistributive lattice, we will show that $\MTub(C_n)$ satisfies the stronger property of being a \emph{congruence uniform} lattice. To do so, we will make use of an alternative characterization of finite congruence uniform lattices given in~\cite{reading2021fundamental} that we will restate after the necessary preliminaries in Corollary~\ref{cor:RST_CongruenceUniform}.

To properly explain our proof, we review the characterization of finite semidistributive lattices of~\cite{reading2021fundamental}, which we refer to as the Fundamental Theorem of Finite Semidistributive Lattices, whose notation and terminology we follow for our exposition. 

We begin by defining some binary relations on an arbitrary finite set. These statements work similarly for the case of infinite sets and lattices, but we restrict to the case where everything is finite.
\begin{definition}[c.f. \cite{reading2021fundamental}]
    Let $\Sha$ be a finite set and $\rightarrow$ be a binary relation on $\Sha$, and let $X\subseteq \Sha$. Then define \[X^\perp \coloneqq \{y\in \Sha \ \mid \ x\not\rightarrow y \text{ }\forall x\in X\}\text{ and }{}^\perp X\coloneqq \{y\in \Sha \ \mid \  y\not\rightarrow x \text{ }\forall x\in X\}.\]
    Then a \emph{maximal orthogonal pair} is a pair $(X,Y)$ of subsets of $\Sha$ so that $X^\perp = Y$ and $X = {}^\perp Y$.
\end{definition}

If $(X',Y')$ and $(X,Y)$ are maximal orthogonal pairs of $\Sha$ then $X'\subseteq X$ if and only if $Y' \supseteq Y$. Importantly this allows us to discuss the lattice of maximal orthogonal pairs of a set $\Sha$ together with the relation $\rightarrow$ by partially ordering $(X',Y') \le (X,Y)$ if $X'\subseteq X$ or $Y'\supseteq Y$. We refer to this lattice as $\text{Pairs}(\to)$.
For the Fundamental Theorem of Finite Semidistributive Lattices, now recall two other relations on $\Sha$ induced by $\to$.

\begin{definition}
    The relation $\twoheadrightarrow$ (pronounced ``onto'') on $\Sha$ is defined by $x\twoheadrightarrow y$ if for all $z\in \Sha$ such that $y\to z$ then $x\to z$. 

    Similarly, the relation $\hookrightarrow$ (pronounced ``into'') is defined dually by $x\hookrightarrow y$ if for all $z$ such that $z\to x$ then $z\to y$.
    
 The pair $(\twoheadrightarrow, \hookrightarrow)$ is called the \emph{factorization} of $\to$ and denoted by $\text{Fact}(\to)$.
\end{definition}
In particular, the relations $\twoheadrightarrow$ and $\hookrightarrow$ are a pair of reflexive and transitive relations. A relation satisfying these properties are known as preorders. Dually, one can take an ordered pair of preorders $(\twoheadrightarrow,\hookrightarrow)$ on the same set $\Sha$ to a reflexive relation on $\Sha$ via the process of \emph{multiplication}, denoted by $\text{Mult}(\twoheadrightarrow,\hookrightarrow)$, where the relation is $x\to z$ if and only if there exists some $y\in \Sha$ such that $x\twoheadrightarrow y \hookrightarrow z$. 

\begin{definition}[\cite{reading2021fundamental}]
    A factorization system is a tuple $(\Sha, \to, \twoheadrightarrow, \hookrightarrow)$ such that $\to, \twoheadrightarrow,\hookrightarrow$ are relations on the set $\Sha$ satisfying $\text{Fact}(\to)= (\twoheadrightarrow, \hookrightarrow)$ and $\text{Mult}(\twoheadrightarrow, \hookrightarrow) = \to$. 
    \end{definition}
    \begin{definition}[\cite{reading2021fundamental}]
    A factorization system is said to be \emph{two--acyclic} if we do \textbf{not} have \begin{itemize}
        \item $x\hookrightarrow y \hookrightarrow x$ or $x\twoheadrightarrow y \twoheadrightarrow x$ for $x\neq y$, and
        \item $x\twoheadrightarrow y\hookrightarrow x$ for $x\neq y$.
    \end{itemize} 
The first condition is the \emph{order condition}, which is to say that the relations $\hookrightarrow, \twoheadrightarrow$ are partial orders. The second condition is the \emph{brick} condition (for historical reasons).
\end{definition}

Recall that for $j\in Ji(L)$ finite semidistributive lattice $L$ $\kappa(j) = \max\{x \mid x\meet j = j_*\}$. Associated to $L$ are the following relations on $Ji(L)$:
\begin{itemize}
    \item $x\to_L y$ if $x\not\leq \kappa(y)$,
    \item $x \twoheadrightarrow_L y$ if $x \ge y$, and
    \item $x \hookrightarrow_L y$ if $\kappa(x) \ge \kappa(y)$.
\end{itemize}

We can now state the Fundamental Theorem of Finite Semidistributive Lattices.

\begin{theorem}[c.f. \cite{reading2021fundamental}]
    A finite poset $L$ is a semidistributive lattice  if and only if it is isomorphic to $\text{Pairs}(\to)$ for a finite two-acyclic factorization system $(\Sha, \to, \twoheadrightarrow,\hookrightarrow)$. In this case $(\Sha, \to, \twoheadrightarrow,\hookrightarrow)$ and $(Ji(L),\to_L,\twoheadrightarrow_L,\hookrightarrow_L) $ are isomorphic.
\end{theorem}
\begin{example}
The following is the relations of the two--acyclic factorization system  $(Ji(L),\to_L,\twoheadrightarrow_L,\hookrightarrow_L) $ when $L = \MTub(C_3)$.
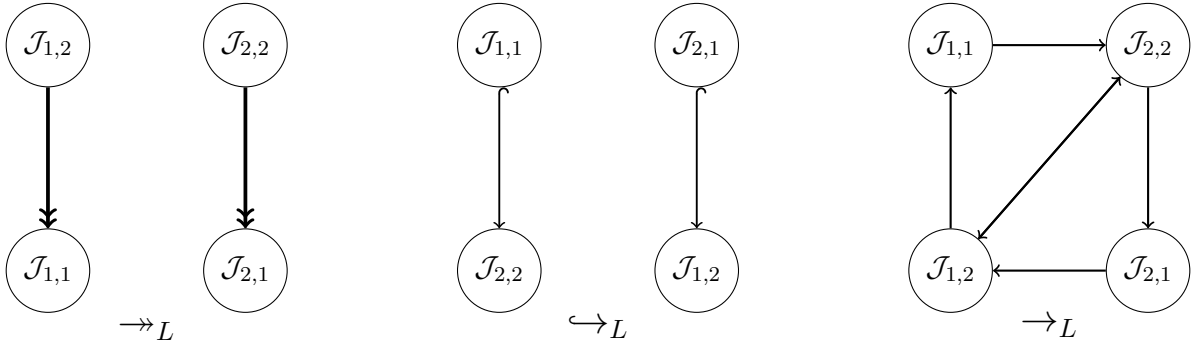
\begin{center}
    \begin{tikzpicture}[scale = .75]
    \def \x {2};
    \def \y {2};

 \node[circle,draw] (11) at (0*\x,0*\y) {$\cj_{1,1}$};
 \node[circle,draw] (12) at (0*\x,2*\y) {$\cj_{1,2}$};
 \node[circle,draw] (21) at (1.75*\x,0*\y) {$\cj_{2,1}$};
 \node[circle,draw] (22) at (1.75*\x,2*\y) {$\cj_{2,2}$};
 \draw[->>, line width=.5mm] (12) -- (11);
 \draw[->>, line width=.5mm] (22) -- (21);
 \draw (0.875*\x,-0.5*\y) node (dbl) {\Large$\twoheadrightarrow_{L}$};
 \node[circle,draw] (11m) at (4*\x,2*\y) {$\cj_{1,1}$};
 \node[circle,draw] (12m) at (5.75*\x,0*\y) {$\cj_{1,2}$};
 \node[circle,draw] (21m) at (5.75*\x,2*\y) {$\cj_{2,1}$};
 \node[circle,draw] (22m) at (4*\x,0*\y) {$\cj_{2,2}$};
 \draw[right hook->, line width=.25mm] (21m) -- (12m);
 \draw[right hook->, line width=.25mm] (11m) -- (22m);
 \draw (4.875*\x,-0.5*\y) node (hook) {\Large$\hookrightarrow_{L}$};
 \node[circle,draw] (11f) at (8*\x,2*\y) {$\cj_{1,1}$};
 \node[circle,draw] (12f) at (8*\x,0*\y) {$\cj_{1,2}$};
 \node[circle,draw] (21f) at (9.75*\x,0*\y) {$\cj_{2,1}$};
 \node[circle,draw] (22f) at (9.75*\x,2*\y) {$\cj_{2,2}$};
 \draw (8.875*\x,-0.5*\y) node (mlt) {\Large$\to_{L}$};  
 \draw[->, thick] (12f)--(11f);
 \draw[->, thick] (12f)--(22f);
 \draw[->, thick] (21f)--(12f);
 \draw[->, thick] (22f)--(21f);
 \draw[->, thick] (11f)--(22f);
 \draw[->, thick] (22f)--(12f);
    \end{tikzpicture}
    \end{center}
It is easy to check that $\to_L = \text{Mult}(\twoheadrightarrow_L, \hookrightarrow_L)$, and if one constructs the partial order $\text{Pairs}(\to_L)$, the result is isomorphic to $\MTub(C_3)$.
\end{example}
In particular, one can think of the relations of $\twoheadrightarrow,\hookrightarrow$ as the partial order of the semidsitributive lattice $L$ restricted to the dual order on $Ji(L)$ for $\twoheadrightarrow$ and $\kappa(Ji(L))$ for $\hookrightarrow$. It will be helpful to have a more specific description of these relations, for which we have Lemma \ref{lem:hook-double_rels} below.

\begin{lemma}\label{lem:hook-double_rels}
    Let $\cj_{i,k}$ and $\cj_{s,t}$ be two different join-irreducible elements in $\MTub(C_n)$. Then 
    \begin{center}
    \begingroup
    \renewcommand{\arraystretch}{1.5}
    \begin{tabular}{lll}
        $\cj_{i,k} \twoheadrightarrow_{\MTub(C_n)} \cj_{s,t}$ &\hspace{0.5cm} if and only if &\hspace{0.5cm} $s=i\text{ and }t < k$, and \\
    $\cj_{i,k} \hookrightarrow_{\MTub(C_n)} \cj_{s,t}$ &\hspace{0.5cm} if and only if & \hspace{0.5cm} $c_i(k) = c_s(t)\text{ and } (i+k,k) <_{\text{lex}} (s+t,t)$,
    \end{tabular}
    \endgroup
    \end{center}
    where $<_\text{lex}$ is lexicographic order.
\end{lemma}
\begin{proof}
    For the first tautology, we have by the definition of $\twoheadrightarrow_{\MTub(C_n)}$ and Proposition \ref{prop:lrshifts} that 
    \[
    \begin{matrix}
    \cj_{i,k} \twoheadrightarrow \cj_{s,t} & \text{ if and only if } & \cj_{i,k} > \cj_{s,t} & \text{ if and only if } & s=i,\;\;t < k.
\end{matrix}
\]
The second tautology requires slightly more. Recall from the proof of Lemma \ref{lem:kappa_bij} that, by Propositions ~\ref{prop:join_chains_relation_to_meets} and ~\ref{prop:max_not_jstar}, and Lemma \ref{lem:max_meet_is_max}, 
   \[
       \kappa(\cj_{i,k}) = \begin{cases}\cm_{c_{i(k), n-k}} &\text{if } k \le n-i\text{, and }\\ \cm_{c_i(k), n-i} &\text{if }k > n-i.
       \end{cases}
   \]
By the definition of $\hookrightarrow_{\MTub(C_n)}$  we have that $\cj_{i,k} \hookrightarrow_{\MTub(C_n)} \cj_{s,t}$ if and only if $\kappa(\cj_{i,k}) > \kappa(\cj_{s,t})$. By the above description of $\kappa$, we have that $\cj_{i,k} \hookrightarrow_{\MTub(C_n)} \cj_{s,t}$ if and only if
\[
c_i(k) = c_s(t) \;\;\text{ and }\;\; \begin{cases}
    k < t &\text{if }i+k \leq n \text{ and }s+t \leq n, \\
    i < t &\text{if }i+k > n \text{ and }s+t \leq n, \\
    k < s &\text{if }i+k \leq n \text{ and }s+t > n, \\
    i < s &\text{if }i+k > n \text{ and }s+t > n .\\
\end{cases}
\]
Now if $i=s$ and $c_i(k)=c_s(t)$, then  $k=t$, which is not true since $\cj_{i,k} \neq \cj_{s,t}$. So we get that $i \neq s$. Furthermore, if $t+s \leq n < i+k$, then $k = c_i(k) = c_s(t) = n-s+1-t$ and $i < t$. So $i+k < i+t = n-s+1 \leq n$, a contradiction. So this third case never occurs when $c_i(k) = c_s(t)$. To finish the claim, we will assume that $c_i(k) = c_s(t)$ and prove that 
\[
(i+k,k) <_\text{lex} (s+t,t) \;\; \text{ if and only if }\;\; \begin{cases}
    k < t &\text{if }i+k \leq n \text{ and }s+t \leq n ,\\
    i < t &\text{if }i+k > n \text{ and }s+t \leq n ,\\
    i < s &\text{if }i+k > n \text{ and }s+t > n. \\
\end{cases}
\]
The forward direction can be done in cases based on the values $i+k$ and $s+t$. 
\begin{itemize}
    \item If $i+k = s+t \leq n$ then $k < t$ and we are done.
    \item If $i+k < s+t \leq n$ then $n-i+1-k = c_i(k) = c_s(t)= n-s+1-t$ and so $i+k=s+t$, a contradiction. This case never occurs.
    \item If $i+k \leq n < s+t$ then $n-i+1-k = c_i(k) = c_s(t) = t$, and so $k+t =n-i+1 \leq n < s+t$. In particular, $k<s$ and we are done.
    \item If $n < i+k < s+t$ then $k = c_i(k) = c_s(t) = t$, and so $i < k$ and we are done.
    \item If $n < i+k = s+t$ then $k=c_i(k)=c_s(t)=t$ and so $i = s$, a contradiction. This case never occurs.
\end{itemize}
Thus we have the forwards implication. The backwards direction can also be done in the three cases given.
\begin{itemize}
    \item  If $i+k \leq n$ and $s+t \leq n$, then $n-i+1-k = c_i(k) = c_s(t) = n-s+1-t$ and so $i+k = s+t$. Since $k < t$, we are done.
    \item If $i+k \leq n$ and $s+t > n$, we are done automatically.
    \item  If $i+k > n$ and $s+t > n$, then $k = c_i(k) = c_s(t) = t$ and $i < s$, so $i+k < s+t$ and we are done.
\end{itemize}
So we have the claim in full.
\end{proof}
\begin{example}
    With Lemma \ref{lem:hook-double_rels}, we can compute the $\twoheadrightarrow_{\MTub(C_n)}$ and $\hookrightarrow_{\MTub(C_n)}$ relations for much larger $n$. Below is the relevant relations for $n = 5$:
    \begin{center}
            \begin{tikzpicture}[scale = .75]
    \def \x {2};
    \def \y {2};

 \node[circle,draw] (11) at (4*\x,1*\y) {11};
 \node[circle,draw] (12) at (4*\x,2*\y) {12};
 \node[circle,draw] (13) at (4*\x,3*\y) {13};
 \node[circle,draw] (14) at (4*\x,4*\y) {14};
 \node[circle,draw] (21) at (3*\x,1*\y) {21};
 \node[circle,draw] (22) at (3*\x,2*\y) {22};
 \node[circle,draw] (23) at (3*\x,3*\y) {23};
 \node[circle,draw] (24) at (3*\x,4*\y) {24};
 \node[circle,draw] (31) at (2*\x,1*\y) {31};
 \node[circle,draw] (32) at (2*\x,2*\y) {32};
 \node[circle,draw] (33) at (2*\x,3*\y) {33};
 \node[circle,draw] (34) at (2*\x,4*\y) {34};
 \node[circle,draw] (41) at (1*\x,1*\y) {41};
 \node[circle,draw] (42) at (1*\x,2*\y) {42};
 \node[circle,draw] (43) at (1*\x,3*\y) {43};
 \node[circle,draw] (44) at (1*\x,4*\y) {44};
 \draw[->>, line width=.5mm] (14) -- (13);
 \draw[->>, line width=.5mm] (13) -- (12);
 \draw[->>, line width=.5mm] (12) -- (11);
 \draw[->>, line width=.5mm] (24) -- (23);
 \draw[->>, line width=.5mm] (23) -- (22);
 \draw[->>, line width=.5mm] (22) -- (21);
 \draw[->>, line width=.5mm] (34) -- (33);
 \draw[->>, line width=.5mm] (33) -- (32);
 \draw[->>, line width=.5mm] (32) -- (31);
 \draw[->>, line width=.5mm] (44) -- (43);
 \draw[->>, line width=.5mm] (43) -- (42);
 \draw[->>, line width=.5mm] (42) -- (41);
 \node[circle,draw] (11m) at (10*\x,4*\y) {11};
 \node[circle,draw] (12m) at (9*\x,3*\y) {12};
 \node[circle,draw] (13m) at (8*\x,2*\y) {13};
 \node[circle,draw] (14m) at (7*\x,1*\y) {14};
 \node[circle,draw] (21m) at (9*\x,4*\y) {21};
 \node[circle,draw] (22m) at (8*\x,3*\y) {22};
 \node[circle,draw] (23m) at (7*\x,2*\y) {23};
 \node[circle,draw] (24m) at (10*\x,3*\y) {24};
 \node[circle,draw] (31m) at (8*\x,4*\y) {31};
 \node[circle,draw] (32m) at (7*\x,3*\y) {32};
 \node[circle,draw] (33m) at (9*\x,2*\y) {33};
 \node[circle,draw] (34m) at (10*\x,2*\y) {34};
 \node[circle,draw] (41m) at (7*\x,4*\y) {41};
 \node[circle,draw] (42m) at (8*\x,1*\y) {42};
 \node[circle,draw] (43m) at (9*\x,1*\y) {43};
 \node[circle,draw] (44m) at (10*\x,1*\y) {44};
 \draw[right hook->, line width=.25mm] (41m) -- (32m);
 \draw[right hook->, line width=.25mm] (32m) -- (23m);
 \draw[right hook->, line width=.25mm] (23m) -- (14m);
 \draw[right hook->, line width=.25mm] (31m) -- (22m);
 \draw[right hook->, line width=.25mm] (22m) -- (13m);
 \draw[right hook->, line width=.25mm] (13m) -- (42m);
 \draw[right hook->, line width=.25mm] (21m) -- (12m);
 \draw[right hook->, line width=.25mm] (12m) -- (33m);
 \draw[right hook->, line width=.25mm] (33m) -- (43m);
 \draw[right hook->, line width=.25mm] (11m) -- (24m);
 \draw[right hook->, line width=.25mm] (24m) -- (34m);
 \draw[right hook->, line width=.25mm] (34m) -- (44m);
    \end{tikzpicture}
    \end{center}
\end{example}

\begin{definition}[c.f. \cite{reading2021fundamental}]
    Let $(\Sha, \rightarrow,\twoheadrightarrow, \hookrightarrow)$ be a two-acyclic factorization system. Given $x,y\in \Sha$ we write $x\rightsquigarrow y$ and say $x$ \emph{directly forces} $y$ if and only if either
    \begin{enumerate}[(i)]
        \item $x$ is $\twoheadrightarrow$-minimal in $\{x'\in \Sha \ | \ x'\hookrightarrow y\}$, or
        \item $x$ is $\hookrightarrow$-maximal in $\{x'\in \Sha \ | \ y\twoheadrightarrow x'\}$.
    \end{enumerate}
    \end{definition}

Now we can proceed to the characterization we will use of congruence uniform lattices (which are always semidistributive \cite[Lemma 4.1, Theorem 5.2]{day1979characterizations})
    \begin{corollary}[Corollary 7.4 \cite{reading2021fundamental}]\label{cor:RST_CongruenceUniform} Suppose $(\Sha, \rightarrow,\twoheadrightarrow, \hookrightarrow)$ is a two-acyclic factorization system. Then $\text{Pairs}(\to)$ is congruence uniform if and only if $\rightsquigarrow$ is acyclic. 
    \end{corollary}
    Since the above Corollary is an equivalent characterization, we will actually take Corollary~\ref{cor:RST_CongruenceUniform} as the definition of congruence uniform lattices.
        \begin{definition}
        A Lattice $L$ is \emph{congruence-uniform} if $L$ is isomorphic to $\text{Pairs}(\to)$ for a two-acyclic factorization system  $(\Sha, \rightarrow,\twoheadrightarrow, \hookrightarrow)$ where $\rightsquigarrow$ is acyclic.
    \end{definition}
    For a more detailed background on congruence uniform lattices, we refer the interested reader to the following work~\cite{day1979characterizations}.

    \begin{theorem}
    \label{thm:congruence_uniform}
       $\MTub(C_n)$ is a congruence uniform lattice.
    \end{theorem}
    \begin{proof}
        We show this by demonstrating that the relation $\rightsquigarrow_{\MTub(C_n)}$ of the associated two-acyclic factorization system of $\MTub(C_n)$ is acyclic. 
        Once accomplished, the claim will follow immediately from Corollary~\ref{cor:RST_CongruenceUniform}.
        Let $(Ji(\MTub(C_n)),\twoheadrightarrow_{\MTub(C_n)}, \hookrightarrow_{\MTub(C_n)})$ be the associated two-acylic factorization system of $\MTub(C_n)$ and let $\cj\rightsquigarrow_{\MTub(C_n)} \cj'$ denote the relation that $\cj$ directly forces $\cj'$.
        
        For the purposes of notation, we now drop the subscripts when dealing with the relations. 
        Let $i,k$ and $s,t$ be such that $\cj_{i,k} \rightsquigarrow \cj_{s,t}$. We will prove that $\rightsquigarrow$ is acyclic by showing that $(i+k,k)<(s+t,t)$ in lexicographic order.
        
        The relations $\hookrightarrowcn$ and $\twoheadrightarrowcn$ are unions of chains. Furthermore, $i=s$ and $c_i(k) = c_s(t)$ cannot be simultaneously true, so a chain in $\hookrightarrowcn$ is an antichain in $\twoheadrightarrowcn$ and vice-versa. In particular, ${\{\cj\in Ji(\MTub(C_n)) \ | \  \cj\hookrightarrowcn \cj_{s,t}\}}$ is an antichain in $\twoheadrightarrowcn$, and thus every element is minimal. Similarly, in ${\{\cj\in Ji(\MTub(C_n))\  | \ \cj_{s,t}\twoheadrightarrowcn \cj\}}$ every element is $\hookrightarrowcn$--maximal.  Thus, $\cj_{i,k}\rightsquigarrowcn \cj_{s,t}$ if and only if $\cj_{s,t} \twoheadrightarrowcn \cj_{i,k}$ or $\cj_{i,k}\hookrightarrowcn \cj_{s,t}$. \par 
        By Lemma \ref{lem:hook-double_rels}, we get in both cases that $(i+k,k) <_\text{lex} (s+t,t)$, and we have the claim. \qedhere

    \end{proof}
    \section{Future Work}
    The poset $\MTub(C_n)$ shares many properties with the weak order on a (type A) Coxeter group as well as its quotient, the Tamari lattice.
    All three posets are semidistributive, congruence uniform and their Hasse diagrams form a regular graph.
    Both the weak order and the Tamari lattice can also be realized as a lattice of certain subcategories called \emph{torsion classes} $\tors \Lambda$ for a $\tau$-tilting finite algebra $\Lambda$. See \cite{Mizuno} and \cite{Thomas_Torsion}.  
    It would be interesting to know whether there exists $\Lambda_{C_n}$ such that $\MTub(C_n)$ is isomorphic to $\tors(\Lambda_{C_n})$.  

    Recently, several authors have studied the action of pop-stack sorting on the weak order on a type A Coxeter group and on the Tamari lattice (and more generally $c$-Cambrian lattices).
    See \cite{Barnard_Defant_Hanson, Choi_Sun, Defan_Lin,Defant, Hong}
    The dynamics of pop-stack sorting on the weak order and the Tamari lattice have a number of pleasant enumerative properties.
    For example, the size image of the pop-stack sorting map on the Tamari lattice is a Motzkin number, and the number of elements with the maximal size orbit is equal to a Catalan number.
    It would be interesting to study the dynamics of the pop-stack sorting map on $\MTub(C_n)$.

    Finally, in \cite[Section~7]{Reading_Lattice_Theory}, Reading defines an alternative partial order on any congruence uniform lattice $L$, which is often called the \emph{shard intersection order on $L$}.
    For both the weak order and the Tamari lattice, the corresponding shard intersection has been shown to be EL-shellable \cite{Bancroft, Petersen}.
    For both of these lattices, the shard intersection order can also be realized a partial order on certain polyhedral cones.
    It would be interesting to know whether the shard intersection order of $\MTub(C_n)$ has a nice geometric interpretation, and whether it is also EL-shellable.

\bibliographystyle{amsalpha}
\bibliography{arxiv_bib}
\end{document}